\documentclass[letterpaper,12pt,oneside,reqno]{amsart}
\usepackage[T2A]{fontenc}%
\usepackage[utf8]{inputenc}%
\usepackage[english]{babel}%
\usepackage{amsmath,amssymb,amsthm,amsfonts}%
\usepackage{hyperref}%
\usepackage{enumerate}%
\usepackage{graphicx}
\usepackage[mathscr]{euscript}
\usepackage{color}
\usepackage[DIV14]{typearea}
\usepackage[width=.9\textwidth]{caption}
\allowdisplaybreaks%
\numberwithin{equation}{section}%
 
\usepackage{xypic}
\usepackage{tikz}
\usepackage{adjustbox}
\usetikzlibrary{arrows}
\usetikzlibrary{decorations.pathreplacing}

\newcommand{\Z}{\mathbb{Z}}

\newcommand{\R}{\mathbb{R}}
\DeclareMathOperator{\E}{\mathbb{E}}

\newcommand{\al}{\alpha}

\newcommand{\la}{\lambda}
\newcommand{\La}{\Lambda}
\newcommand{\be}{\beta}
\newcommand{\g}{\gamma}

\newcommand{\ga}{\gamma}
\newcommand{\ka}{\varkappa}

\newcommand{\Sym}{\mathsf{Sym}}
\newcommand{\Yb}{\mathbb{Y}}
\newcommand{\ab}{\boldsymbol\alpha}
\newcommand{\bb}{{\boldsymbol\beta}}
\newcommand{\Pl}{\mathbf{Pl}}

\newcommand{\M}{\mathscr{M}}
\newcommand{\HL}{\mathscr{HL}}
\newcommand{\MM}{\mathscr{MM}}
\newcommand{\MP}{\mathscr{MP}}
\newcommand{\x}{\mathbf{x}}
\newcommand{\Ptab}{\mathscr{P}}

\newcommand{\Ab}{\mathbf{A}}
\newcommand{\Bb}{\mathbf{B}}

\newcommand{\Ps}{\mathsf{P}}
\newcommand{\Qs}{\mathsf{Q}}
\newcommand{\Us}{\mathscr{U}}
\newcommand{\Ds}{\mathscr{D}}

\DeclareMathOperator{\Prob}{\mathrm{Prob}}
\DeclareMathOperator{\qProb}{\mathrm{`Prob'}}
\newcommand{\bi}[2]{\big[\begin{smallmatrix}
	#2\\\rule{0pt}{7.5pt}#1
\end{smallmatrix}\big]}
\newcommand{\A}{\mathcal{A}}
\newcommand{\prech}{\prec_{\mathsf{h}}}
\newcommand{\precb}{\nearrow}
\newcommand{\precv}{\prec_{\mathsf{v}}}
\newcommand{\de}{\mathsf{e}}
\newcommand{\F}{\mathscr{F}}
\DeclareMathOperator{\nt}{\mathsf{next}}

\newcommand{\WW}{\mathsf{W}}
\newcommand{\NN}{\mathsf{N}}
\newcommand{\MMM}{\mathsf{M}}

\newcommand{\rp}{\mathsf{r}}

\newcommand{\anum}{\mathsf{a}}
\newcommand{\bnum}{\mathsf{b}}

\newcommand{\mm}{\mathsf{m}}
\newcommand{\wm}{\mathsf{w}}
\newcommand{\pu}{p_+}
\newcommand{\pd}{p_-}
\newcommand{\pz}{p_0}

\newcommand{\T}{\mathsf{T}}
\newcommand{\Q}{\mathsf{Q}}

\newcommand{\q}{\mathfrak{q}}
\newcommand{\Ub}{\mathbb{U}}

\def\note#1{\textup{\textsf{\Large\color{blue}}}}

\newtheorem{proposition}{Proposition}[section]
\newtheorem{lemma}[proposition]{Lemma}

\newtheorem{theorem}[proposition]{Theorem}
\theoremstyle{definition}
\newtheorem{definition}[proposition]{Definition}
\newtheorem{remark}[proposition]{Remark}
\newtheoremstyle{dynrule}
{}
{}
{\slshape}
{}
{\bf}
{.}
{.5em}
{}
\theoremstyle{dynrule}
\newtheorem{sampling}{Sampling algorithm}
\newtheorem{conjecture}[proposition]{Conjecture}

\begin{document}
\title[Law of Large Numbers for Infinite Random Matrices over a Finite Field]{Law of Large Numbers for Infinite\\{}Random Matrices over a Finite Field}

\author[A. Bufetov]{Alexey Bufetov}
\address{A. Bufetov,
Department of Mathematics,
International Laboratory of Representation Theory and Mathematical Physics,
Higher School of Economics,
Vavilova str. 7, Moscow, 117312, Russia,
and
Institute for Information Transmission Problems, Bolshoy Karetny per.~19, Moscow, 127994, Russia}
\email{alexey.bufetov@gmail.com}

\author[L. Petrov]{Leonid Petrov}
\address{L. Petrov,
Department of Mathematics, Northeastern University, 360 Huntington ave., Boston, MA 02115, USA,
and
Institute for Information Transmission Problems, Bolshoy Karet\-ny per.~19, Moscow, 127994, Russia}
\email{lenia.petrov@gmail.com}


\dedicatory{To Grigori Olshanski on the occasion of his 65th birthday}

\begin{abstract}
	Asymptotic representation theory
	of general linear groups
	$GL(n,F_\q)$
	over a finite field
	leads to studying
	probability measures $\rho$ on the
	group $\Ub$ of all infinite uni-uppertriangular
	matrices over $F_\q$,
	with the condition that $\rho$
	is invariant under
	conjugations
	by arbitrary infinite matrices.
	Such probability measures form an
	infinite-dimensional simplex,
	and the description of its extreme points
	(in other words, ergodic measures $\rho$)
	was conjectured by Kerov in connection
	with nonnegative specializations of
	Hall--Littlewood symmetric functions.
	
	Vershik and Kerov also conjectured the following Law of Large Numbers.
	Consider an infinite random matrix drawn from an
	ergodic measure coming from the Kerov's conjectural classification and its $n \times n$ submatrix formed by the first rows and
    columns.
	The sizes of Jordan blocks of the submatrix can be
	interpreted as a (random) partition of $n$,
	or, equivalently, as a (random)
	Young diagram $\la(n)$
	with $n$ boxes. Then, as $n\to\infty$, the rows and columns of
	$\la(n)$ have almost sure limiting frequencies corresponding to parameters of this ergodic measure.

	Our main result is the proof of this Law of Large Numbers.
	We achieve it by analyzing
	a new randomized
	Robinson--Schensted--Knuth (RSK) insertion algorithm
	which samples random Young diagrams $\la(n)$
	coming from ergodic measures.
	The probability
	weights of these Young diagrams are expressed in terms of
	Hall--Littlewood symmetric functions.
	Our insertion algorithm is a modified and extended version
	of a recent construction by
	Borodin and the second author~\cite{BorodinPetrov2013NN}.
	On the other hand, our
	randomized RSK insertion
	generalizes a version of the RSK insertion
	introduced by Vershik and Kerov
	\cite{Vershik1986} in connection
	with asymptotic representation theory
	of symmetric groups (which is governed by nonnegative
	specializations of Schur symmetric functions).
\end{abstract}

\maketitle

\setcounter{tocdepth}{1}
\tableofcontents
\setcounter{tocdepth}{2}

\section{Introduction} 
\label{sec:introduction}

In
\S\S \ref{sub:infinite_random_matrices_over_a_finite_field}--\ref{sub:randomized_rsk_insertion_algorithm} we describe the
setup and our main results, and then in
\S \ref{sub:symmetric_groups}
and
\S \ref{sub:asymptotic_representation_theory_of_linear_groups_over_a_finite_field}
we discuss connections with asymptotic representation theory.

\subsection{Infinite random matrices over a finite field} 
\label{sub:infinite_random_matrices_over_a_finite_field}

Let $\q=p^{d}$ be a prime power, $F_\q$
be the corresponding finite field,
and $GL(n,F_\q)$ be the
group of all invertible $n\times n$ matrices over $F_\q$.

Let $\Ub$ be the group of all
infinite uni-uppertriangular matrices over
$F_\q$, i.e.,
matrices $X=[X_{ij}]_{i,j=1}^{\infty}$
for which $X_{ii}=1$ and $X_{ij}=0$ for $i>j$.
This is a compact group (under the topology of pointwise
convergence of matrix elements).

For any uni-uppertriangular $n\times n$ matrix $g\in GL(n,F_\q)$,
denote by $Cyl_g\subset \Ub$ the cylindrical subset
consisting of all infinite matrices in $\Ub$ whose top
$n\times n$ corner coincides with $g$.
Note that all eigenvalues of such a matrix $g$ are all equal to $1$,
and so its conjugacy class in $GL(n,F_\q)$
is completely determined
by sizes of its Jordan blocks. We will identify
sizes of these blocks
with partitions $\la$ of~$n$
(=~Young diagrams $\la$ with $|\la|=n$ boxes; see \S \ref{sub:young_diagrams} for notation).

\begin{definition}\label{def:central}
	A probability Borel measure $\rho$ on $\Ub$ is called
	\emph{central} if for any finite
	uni-uppertriangular matrix $g$ the measure
	$\rho(Cyl_g)$ depends only on the conjugacy class of
	$g$, i.e., on the partition corresponding to
	sizes of its Jordan blocks.
\end{definition}

Centrality property
means conjugation-invariance in the sense that
if $h\in GL(\infty,F_\q)=\bigcup_{n=1}^{\infty}GL(n,F_\q)$
(i.e., $h$ is an
infinite matrix which differs from the identity matrix
in a finite number of matrix elements)
and $M\subset \Ub$ is a Borel subset
such that
$hMh^{-1}\subset \Ub$, then
it must be that $\rho(M)=\rho(hMh^{-1})$.

Central probability measures on $\Ub$ form a convex set.
Its extreme points (i.e., central measures which cannot
be expressed as nontrivial convex combinations
of other central measures) will be referred to as
\emph{ergodic central measures}.

\medskip

The classification of ergodic central measures on $\Ub$
is a well-known open problem related to the asymptotic
representation theory of
the linear groups $GL(n,F_\q)$, see \S \ref{sub:asymptotic_representation_theory_of_linear_groups_over_a_finite_field}
below
for more discussion and references.
A conjectural answer to the problem
is given by Kerov
\cite{Kerov1992HL},
\cite[Ch.~2.9]{Kerov-book}:\footnote{The conjecture was
originally formulated in equivalent terms of
nonnegative specializations
of Hall--Littlewood symmetric functions.
Another equivalent formulation involves coherent probability
measures on the Young branching graph
with formal edge multiplicities depending on $\q$ (cf. \S \ref{sub:plancherel_specializations_and_young_graph}).
See \cite[Thm. 2.3]{Borodin1995_ext},
\cite{Fulman2000RM}, \cite[Prop. 4.7]{GorinKerovVershikFq2012} for details of these equivalences.}
\begin{conjecture}[Kerov]\label{conj:Spec_HL}
	Ergodic central measures
	on $\Ub$ are in one-to-one correspondence with
	triplets $(\ab;\boldsymbol\be;\ga)\in\R^{2\infty+1}$
	such that
	\begin{align}\label{alpha_beta_gamma_intro1}
		\ab=(\al_1\ge\al_2\ge \ldots\ge0),
		\qquad
		\boldsymbol\be=
		(\be_1\ge\be_2\ge \ldots\ge 0),
		\qquad\ga\ge0,
	\end{align}
	and
	\begin{align}\label{alpha_beta_gamma_intro2}
		\sum_{i=1}^{\infty}\al_i+
		\sum_{i=1}^{\infty}\frac{\be_i}{1-\q^{-1}}+
		\frac{\gamma}{1-\q^{-1}}=1.
	\end{align}
	
	The correspondence is established via the measures of cylindrical sets:
	\begin{align}\label{ergodic_central}
		\rho^{\ab;\boldsymbol\be;\ga}(Cyl_g)=
		\frac{\q^{-n(n-1)/2+\sum_{i}(i-1)\la_i}}
		{(1-\q^{-1})^{n}}
		Q_{\la}(\ab;\bb;\Pl_\ga\,|\, 0,\q^{-1})
	\end{align}
	for any $n\ge1$ and any uni-uppertriangular
	$n\times n$
	matrix $g$, where $\la$, $|\la|=n$, corresponds to Jordan
	block sizes of $g$.
	Here $Q_{\la}(\ab;\bb;\Pl_\ga\,|\, 0,\q^{-1})$
	denotes the specialization of
	the ``$Q$'' Hall--Littlewood symmetric function, see
	\S \ref{sec:preliminaries}.\footnote{Our
	notation of parameters
	$(\ab;\boldsymbol\be;\ga)$
	borrowed from Borodin--Corwin
	\cite{BorodinCorwin2011Macdonald}
	and also used in Borodin--Petrov \cite{BorodinPetrov2013NN}
	differs from the one of Kerov
	\cite{Kerov-book}, see also
	Gorin--Kerov--Vershik
	\cite{GorinKerovVershikFq2012}.
	Details are explained in 
	Remark \ref{rmk:Kerov_different}
	below.}
\end{conjecture}

Probability measures $\rho^{\ab;\boldsymbol\be;\ga}$
on $\Ub$
with cylindrical probabilities \eqref{ergodic_central} exist and
are indeed ergodic, see Gorin--Kerov--Vershik
\cite[Prop. 4.7]{GorinKerovVershikFq2012}.

One example of
a central measure is the uniform measure
on $\Ub$
studied by Borodin
\cite{Borodin1995}, \cite{Borodin1995_ext}.
This measure is ergodic, it corresponds to
taking $\al_i=(1-\q^{-1})\q^{1-i}$, $i=1,2,\ldots$,
and setting all $\be_i$'s and $\gamma$ to zero (see also \S \ref{sub:asymptotic_representation_theory_of_linear_groups_over_a_finite_field} below for connection to unipotent traces).

\medskip

Restating \eqref{ergodic_central} in terms
of the distribution of the Young diagram
$\la$ corresponding to Jordan block sizes, one
arrives at the probability distribution
\begin{align}\label{HL_coherent_intro}
	\HL_n^{\ab;\bb;\Pl_\ga}(\la):=
	n!\,
	Q_\la(\ab;\bb;\Pl_\ga\,|\,0,\q^{-1})
	P_\la(\Pl_{1}\,|\,0,\q^{-1})
\end{align}
on the set $\Yb_n$ of all Young diagrams with $n$
boxes. Here $P_\la$ is
the ``$P$'' Hall--Littlewood symmetric function,\footnote{$P_\la$ is a constant multiple
of $Q_\la$. We use the standard notation of \cite{Macdonald1995}.} and
$P_\la(\Pl_{1}\,|\,0,\q^{-1})$
denotes the so-called Plancherel specialization of $P_\la$
(see \S \ref{sec:preliminaries}).
This Plancherel specialization
incorporates
the number of uni-uppertriangular
matrices from $GL(n,F_\q)$
having the given Jordan blocks sizes determined
by~$\la$.
The passage from
\eqref{ergodic_central} to \eqref{HL_coherent_intro}
follows from
Fulman \cite[Thm. 1]{Fulman2000eigen}
and Gorin--Kerov--Vershik
\cite[\S4]{GorinKerovVershikFq2012},
see also \S \ref{sub:plancherel_specializations_and_young_graph} below for a connection to the Plancherel specialization and
to the Young graph with certain formal edge	multiplicities.

\begin{remark}\label{rmk:coupling_intro}
	Because the measures $\HL_n^{\ab;\bb;\Pl_\ga}$
	for various $n$ come from the
	same distribution $\rho^{\ab;\boldsymbol\be;\ga}$
	on infinite matrices over $F_\q$,
	they satisfy certain \emph{coherency relations}
	(see \S \ref{sub:plancherel_specializations_and_young_graph}). Note that
	for various $n=1,2,\ldots$
	the corresponding
	random Young diagrams
	$\la(n)\in\Yb_n$ distributed according to
	$\HL_n^{\ab;\bb;\Pl_\ga}$
	are defined on \emph{the same probability
	space} (on which the infinite random
	matrix is defined).
\end{remark}

We refer to Gorin--Kerov--Vershik \cite{GorinKerovVershikFq2012}
and
Fulman
\cite{Fulman1999},
\cite{Fulman2000RM},
\cite{Fulma2013NT}
for further connections between
random matrices over a finite field
and Hall--Littlewood symmetric functions.
See also \S \ref{sub:asymptotic_representation_theory_of_linear_groups_over_a_finite_field} below for a
brief discussion of asymptotic
representation theory of
the groups~$GL(n,F_\q)$.

\subsection{Law of Large Numbers} 
\label{sub:law_of_large_numbers}

Our main result is the proof of
the following Law of Large Numbers
for sizes of Jordan blocks
of random infinite uni-uppertriangular matrices
over $F_\q$
under an ergodic central measure:

\begin{theorem}[Vershik--Kerov's conjecture
\cite{VK07}, \cite{GorinKerovVershikFq2012}]\label{thm:LLN_intro}
	Let $(\ab;\boldsymbol\be;\ga)\in\R^{2\infty+1}$
	be any triplet satisfying \eqref{alpha_beta_gamma_intro1}--\eqref{alpha_beta_gamma_intro2}
	such that the third parameter $\gamma$ is zero.
	Let for each $n=1,2,\ldots$, $\la(n)\in\Yb_n$ be the random
	Young diagram
	corresponding to sizes of Jordan blocks
	of the $n\times n$ truncation of the
	random matrix distributed according to
	$\rho^{\ab;\boldsymbol\be;0}$.
	Then, $\rho^{\ab;\boldsymbol\be;0}$-almost surely,
	\begin{align}
		\frac{\la_i(n)}{n}\to\al_i,\qquad
		\frac{\la_i'(n)}{n}\to
		\frac{\be_i}{1-\q^{-1}},\qquad
		i=1,2,\ldots,
		\label{LLN_intro}
	\end{align}
	where $\la_i(n)$ and $\la_i'(n)$ denotes the length
	of the $i$-th row (resp. column) of
	$\la(n)$ (see also \S \ref{sub:young_diagrams}).
\end{theorem}

In fact, we prove 
Theorem \ref{thm:LLN_intro} 
for any value of the parameter $t=\q^{-1}\in[0,1)$ in the measures
\eqref{HL_coherent_intro}
(this $t$
is usually referred to as the Hall--Littlewood parameter), 
not just for an inverse of a prime power.
See Theorem \ref{thm:main_s7}.

\begin{remark}
	The Law of Large Numbers implies
	existence of the asymptotic speeds
	of particles in the
	\emph{$q$-PushTASEP}
	($q$-deformed pushing
	totally asymmetric simple exclusion
	process)
	with varying particle
	speeds. This
	($1+1$)-dimensional
	continuous-time integrable particle
	system was introduced and studied in
	Borodin--Petrov \cite{BorodinPetrov2013NN}
	and Corwin--Petrov \cite{CorwinPetrov2013}.
	See Remark \ref{rmk:qpushtasep}
	for more detail.
\end{remark}

\begin{remark}\label{rmk:branching1}
	Probability measures \eqref{HL_coherent_intro}
	may be viewed as extreme coherent measures on the
	Young graph with certain formal edge multiplicities
	related to Hall--Littlewood polynomials
	(cf. Remarks \ref{rmk:coupling_intro} and \ref{rmk:branching_graphs}).
	In view of this connection, one would expect these measures
	to have asymptotic frequencies
	(as it happens for extreme coherent measures on other branching graphs,
	e.g., see \cite{VK81AsymptoticTheory}, \cite{Kerov1989},
	\cite{VK82CharactersU}, \cite{OkOl1998}, \cite{Gorin2010q},
	\cite{BorodinOlsh2011GT}, \cite{Petrov2012GT}).
	Our Theorem \ref{thm:LLN_intro}
	is exactly a statement about these asymptotic frequencies.
	See also Conjecture \ref{conj:Mac} below for the case of Macdonald edge multiplicities.
\end{remark}

Let us now formulate three conjectures related
to Theorem \ref{thm:LLN_intro}.

\begin{conjecture}[Case $\ga>0$]
	We believe that the technical assumption
	$\ga=0$ can be dropped, and the same convergence
	\eqref{LLN_intro} could be established for any
	triplet
	$(\ab;\boldsymbol\be;\ga)$
	with \eqref{alpha_beta_gamma_intro1}--\eqref{alpha_beta_gamma_intro2}.
\end{conjecture}

\begin{conjecture}[Central Limit Theorem]\label{conj:CLT_intro}
	If the $\alpha$- and the $\beta$-parameters
	are distinct (when they are positive), i.e.,
	$\al_1>\al_2> \ldots$
	and $\be_1>\be_2> \ldots$,
	then the lengths of rows and columns
	of random Young diagrams $\la(n)$
	satisfy a
	Central Limit Theorem: The infinite
	vector
	$\{\la_1(n),\la_2(n),\ldots;\la_1'(n),\la_2'(n),\ldots\}$
	is asymptotically jointly Gaussian
	after subtracting the limiting means \eqref{LLN_intro} and
	normalizing by $\sqrt n$.
	The limiting covariances are equal to:\footnote{Here and below $\mathbf{1}_{A}$ means
	the indicator of $A$.}
	\begin{enumerate}[$\bullet$]
		\item
		$\displaystyle\al_i\mathbf{1}_{i=j}-\al_i\al_j$ between $\la_i$
		and $\la_j$;
		\smallskip
		\item
		$\displaystyle\frac{\be_i}{1-\q^{-1}}\mathbf{1}_{i=j}-
		\frac{\be_i\be_j}{(1-\q^{-1})^{2}}$ between $\la_i'$
		and $\la_j'$;
		\item $\displaystyle-\frac{\al_i\be_j}{1-\q^{-1}}$
		between $\la_i$
		and $\la_j'$.
	\end{enumerate}
\end{conjecture}

One can also replace $\q^{-1}$ by any value of the
Hall--Littlewood parameter $t\in[0,1)$ in the
formulation of Conjecture \ref{conj:CLT_intro}.
In the case of symmetric groups
(corresponding to $t=0$,
see \S \ref{sub:symmetric_groups} below)
a similar Central Limit Theorem was
established by
Feray and Meliot
\cite{FerayMeliot2012}, \cite{MeliotCLT2011}
and
Bufetov \cite{BufetovCLT}.
In this case, the behavior of fluctuations changes
when the
assumptions
on parameters
$(\ab;\boldsymbol\be;\ga)$
are not satisfied. This suggests the same restrictions
on parameters for the $t>0$ case as well.

Conjecture \ref{conj:CLT_intro} should be
accessible by the technique
of the present paper: The main idea behind our proof of Theorem \ref{thm:LLN_intro} 
is that the asymptotic behavior of random Young diagrams is shown to be the same 
as the asymptotic behavior of random words of fixed length
with independently distributed letters 
(see \S \ref{sub:randomized_rsk_insertion_algorithm} below). 
Conjecture \ref{conj:CLT_intro} asserts that the asymptotic behavior of fluctuations coincides as well,
so the covariance matrix for these two models should be the same. 
However, we do not pursue this direction here.

For the uniform measure on $\Ub$,
the Law of Large Numbers (Theorem \ref{thm:LLN_intro})
and the Central Limit Theorem
(Conjecture \ref{conj:CLT_intro})
were established by Borodin
\cite{Borodin1995}, \cite{Borodin1995_ext}.

\medskip

One can replace the Hall--Littlewood
symmetric functions by the Macdonald ones
which depend on two parameters $q,t\in[0,1)$,
see \S \ref{sub:preliminaries}
(note the
difference between the Macdonald parameter $q$ and the
prime power $\q=t^{-1}$ which is the size of the base finite
field).
The corresponding probability
measures $\M_n^{\ab;\bb;\Pl_\ga}$ on Young diagrams with $n$
boxes can be defined similarly to
\eqref{HL_coherent_intro}.
These measures with Macdonald parameters
were introduced by Fulman
\cite{fulman1997probabilistic}
and studied in great detail by
Borodin and Corwin \cite{BorodinCorwin2011Macdonald}.
See also Forrester--Rains \cite{ForresterRains2005Macdonald}.

\begin{conjecture}[Law of Large Numbers with
Macdonald parameters]\label{conj:Mac}
	Let $\la(n)$ be the random
	Young diagram
	distributed according to
	$\M_n^{\ab;\bb;\Pl_\ga}$.
	Then, with almost sure convergence,
	\begin{align}
		\frac{\la_i(n)}{n}\to\al_i,\qquad
		\frac{\la_i'(n)}{n}\to
		\be_i \frac{1-q}{1-t},\qquad
		i=1,2,\ldots.
		\label{LLN_Mac}
	\end{align}
\end{conjecture}


\subsection{Randomized Robinson--Schensted--Knuth (RSK)
insertion} 
\label{sub:randomized_rsk_insertion_algorithm}

Our main technique for studying
probability measures $\HL_n^{\ab;\bb;\Pl_\ga}$
\eqref{HL_coherent_intro}
and proving Theorem~\ref{thm:LLN_intro}
is a certain new sampling algorithm for these measures.
Namely, we introduce a randomized version of the
classical Robinson--Schensted--Knuth (RSK) insertion
algorithm which samples $\HL_n^{\ab;\bb;\Pl_\ga}$.
About the classical RSK, e.g., see Stanley
\cite[Ch. 7]{Stanley1999}, Sagan \cite{sagan2001symmetric},
and also Borodin--Petrov \cite[\S7]{BorodinPetrov2013NN}.

Let us describe this insertion
algorithm in the case when
there are
only finitely many nonzero $\al$
parameters, namely, $\ab=(\al_1,\ldots,\al_N)$,
and that all $\be_i$ and $\gamma$ are zero.
Then \eqref{alpha_beta_gamma_intro2} means
that $\al_1+\ldots+\al_N=1$.
The case of general parameters is described
in \S \ref{sec:rsk_type_algorithm_for_sampling_hl_coherent_measures}.
The input of the algorithm is
a word $w=\xi_1\xi_2 \ldots\xi_n$,
where $\xi_i\in\A=\{1,2,\ldots,N\}$. Applied to a fixed word
$w$, the algorithm produces a \emph{random}
interlacing integer array (see Fig.~\ref{fig:interlacing_intro})
\begin{align*}
	\{\la^{(m)}_{i}\colon m=1,\ldots,N,\; i=1,\ldots,m\},
	\qquad \la^{(m)}_{i+1}\le \la^{(m-1)}_{i}\le
	\la^{(m)}_{i}.
\end{align*}
In an interlacing array we call a particle $\la^{(m)}_{i}$
\emph{blocked} if $\la^{(m)}_{i}=\la^{(m-1)}_{i-1}$.
Otherwise the particle is called \emph{free}.

The interlacing array can be interpreted
as a semistandard Young tableau
$\Ptab$ of shape $\la^{(N)}=(\la^{(N)}_{1}\ge
\ldots\ge\la^{(N)}_{N})$.
For general parameters, the
alphabet $\A$ is different, and the notion
of a semistandard Young tableau $\Ptab$
has to be changed accordingly.
Such more general $\A$-tableaux first appeared in the
work of Vershik and Kerov
\cite{Vershik1986},
see also Berele and Regev \cite{BereleRegev}.
(For definitions
of standard tableaux, semistandard tableaux,
and $\A$-tableaux
see Definition \ref{def:Atableaux} and Remark \ref{rmk:SSYT_SYT}.)
\begin{figure}[htbp]
\begin{center}
\begin{tabular}{rl}
\begin{tikzpicture}
	[scale=.5,thick]
	\node at (0.5,-0.5) {1};
	\node at (1.5,-0.5) {1};
	\node at (2.5,-0.5) {1};
	\node at (3.5,-0.5) {1};
	\node at (4.5,-0.5) {1};
	\node at (5.5,-0.5) {2};
	\node at (6.5,-0.5) {3};
	\node at (7.5,-0.5) {3};
	\node at (8.5,-0.5) {3};
	\node at (9.5,-0.5) {4};
	\node at (0.5,-1.5) {2};
	\node at (1.5,-1.5) {2};
	\node at (2.5,-1.5) {4};
	\node at (3.5,-1.5) {4};
	\node at (4.5,-1.5) {4};
	\node at (5.5,-1.5) {4};
	\node at (0.5,-2.5) {3};
	\node at (1.5,-2.5) {3};
	\node at (0.5,-3.5) {4};
	\draw (0,0) -- (10,0);
	\draw (0,-1) -- (10,-1);
	\draw (0,-2) -- (6,-2);
	\draw (0,-3) -- (2,-3);
	\draw (0,-4) -- (1,-4);
	\draw (0,0) -- (0,-4);
	\draw (1,0) -- (1,-4);
	\draw (2,0) -- (2,-3);
	\draw (3,0) -- (3,-2);
	\draw (4,0) -- (4,-2);
	\draw (5,0) -- (5,-2);
	\draw (6,0) -- (6,-2);
	\draw (7,0) -- (7,-1);
	\draw (8,0) -- (8,-1);
	\draw (9,0) -- (9,-1);
	\draw (10,0) -- (10,-1);
\end{tikzpicture}\\\rule{0pt}{90pt}
\raisebox{7pt}{\begin{tikzpicture}
[scale=1, very thick]
\def\sp{0.09};
\def\x{.5};
\def\y{.7};
\node at (0,0) {5};
\node at (\x,\y) {6};	
\node at (-\x,\y) {2};	
\node at (2*\x,2*\y) {9};	
\node at (0,2*\y) {2};	
\node at (-2*\x,2*\y) {2};	
\node at (-3*\x,3*\y) {1};	
\node at (-1*\x,3*\y) {2};	
\node at (1*\x,3*\y) {6};	
\node at (3*\x,3*\y) {10};	
\end{tikzpicture}}
&\hspace{60pt}
\begin{tikzpicture}
[scale=1, very thick]
\def\sp{0.09};
\def\x{.54};
\def\y{.7};
\def\opac{.45}
\def\wid{1.2}
\draw[line width=\wid, color=blue, opacity=\opac]
(\x,\y)--++(-\x,-\y)--++(-3*\x,\y);
\draw[line width=\wid, color=blue, opacity=\opac]
(4*\x,2*\y)--++(-3*\x,-\y)--++(-4*\x+3/2*\sp,\y)
--++(-3/2*\sp,-\y)--++(-3/2*\sp,\y);
\draw[line width=\wid, color=blue, opacity=\opac]
(5*\x,3*\y)--++(-\x,-\y)--++(-3*\x,\y)
--++(-4*\x+3/2*\sp,-\y)--++(-3/2*\sp,\y)
--++(-3/2*\sp,-\y)--++(-\x+3/2*\sp,\y);
\foreach \pt in
{
(0,0),
(-3*\x,\y), (\x,\y),
(-3*\x-\sp*3/2,2*\y), (-3*\x+\sp*3/2,2*\y), (4*\x,2*\y),
(-4*\x,3*\y), (-3*\x,3*\y), (1*\x,3*\y), (5*\x,3*\y),
}
{
\draw[fill] \pt circle (\sp);
}
\foreach \ll in {1,2,3,4}
{
\draw[->, thick] (-6*\x,\ll*\y-\y) -- (7*\x,\ll*\y-\y)
node[right] {$\la^{(\ll)}$};
}
\foreach \ver in {-5,...,5}
{
\draw[dotted] (\ver*\x,-.3*\y)--(\ver*\x,3.3*\y);
}
\node at (-5*\x-2*\sp,0-3*\sp) {0};
\end{tikzpicture}
\end{tabular}
\end{center}
\caption{A
semistandard Young tableau $\Ptab$
and the corresponding
interlacing integer array of depth $N=4$.
The shape of the Young tableau is
$\la^{(N)}=(10,6,2,1)$.
Particles (on the right) are located at positions $\la^{(m)}_{i}$,
where $m$ and $i$
represent vertical and horizontal coordinates,
respectively. When there are several particles
occupying the same position, we draw them close to each other.
Zigzags indicate the interlacing property.}
\label{fig:interlacing_intro}
\end{figure}
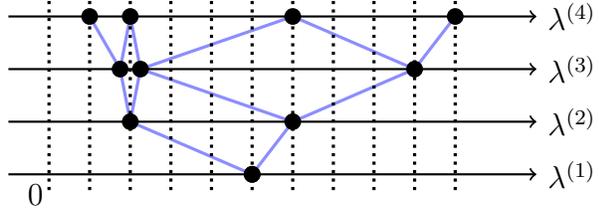

The randomized insertion algorithm starts
from the empty configuration, i.e.,
$\la^{(m)}_{i}=0$ for all $m$ and $i$.
Letters from the word $w=\xi_1\ldots\xi_n$
arrive one by one, and are inserted into the
semistandard tableau $\Ptab$.
When a letter $\xi_j$
is inserted, the
interlacing array (corresponding to $\Ptab$
as on Fig.~\ref{fig:interlacing_intro})
undergoes the following modifications:
\begin{enumerate}[$\bullet$]
	\item
	First, at the level $\xi_j$ the
	leftmost free particle moves to the right by one.
	\item
	After that, modifications propagate upwards to
	all levels $m=\xi_j,\xi_j+1,\ldots,N$ as follows:
	\begin{enumerate}[$\circ$]
		\item If a particle $\la^{(m)}_{i}$ moves to the
		right by one,
		and $\la^{(m)}_{i}=\la^{(m+1)}_{i}$ before the move,
		then the particle $\la^{(m+1)}_{i}$ also immediately
		moves to the right by one with probability one.
		This second move restores the interlacing which was
		broken by the first move
		(\emph{mandatory short-range pushing}).
		\item
		Otherwise, if a particle $\la^{(m)}_{i}$ has moved (to the right by one), then at the
		next level $m+1$
		the first free upper right neighbor of
		$\la^{(m)}_{i}$ immediately moves to the right by one
		with probability
		$r_i(\la^{(m)},\la^{(m+1)}\,|\,\q^{-1})$ (\emph{pushing}),
		or the upper left neighbor
		$\la^{(m+1)}_{i+1}$ moves with the complementary
		probability
		$1-r_i(\la^{(m)},\la^{(m+1)}\,|\,\q^{-1})$ (\emph{pulling}).
		These probabilities depend on $\q$ and on the
		number of particles at levels
		$m$ and $m+1$ which occupy the horizontal
		position of $\la^{(m)}_{i}$ before its move.
		They are
		determined
		as on Fig.~\ref{fig:pushing_intro}
		(see
		\S \ref{sub:pushing_and_pulling_probabilities}
		for a complete description).
	\end{enumerate}
\end{enumerate}
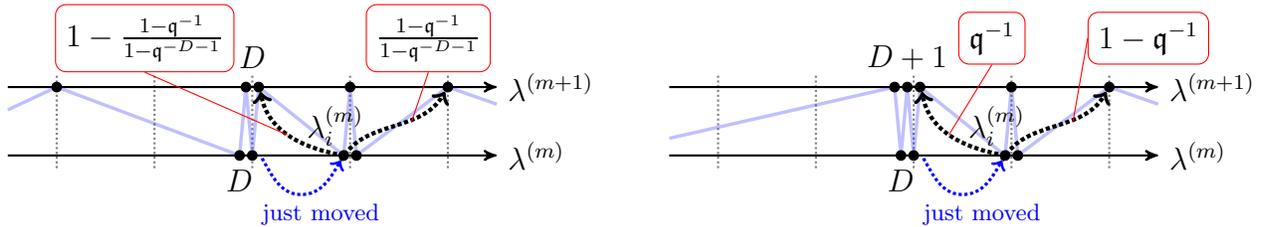
\begin{figure}[htbp]
\begin{center}
	\begin{tabular}{ll}
		\begin{tikzpicture}[
		    scale=1.3,
		    axis/.style={thick, ->, >=stealth'},
		    block/.style ={rectangle, draw=red,
			align=center, rounded corners, minimum height=1em}]
		    \def\y{.7}
		    \draw[axis] (0,0) -- (5,0) node(xline)[right]{$\la^{(m)}$};
		    \draw[axis] (0,\y) -- (5,\y) node(xline)[right]{$\la^{(m+1)}$};
		    \foreach \hh in {.5, 1.5, 2.5, 3.5, 4.5}
		    {
		    	\draw[densely dotted, thick, opacity=.5] (\hh,-.13) -- (\hh,\y+.13);
		    }
		    \def\sp{.13};
		    \def\opac{.25}
			\draw[line width=1.3, color=blue, opacity=\opac]
			(0,2*\y/3)--(.5,\y)--(2.5-\sp,0)--++(\sp/2,\y)--++(\sp/2,-\y)
			--++(\sp/2,\y)--(3.5-\sp/2,0)--++(\sp/2,\y)--++(\sp/2,-\y)
			--(4.5,\y)--(5,3*\y/4);
		    \foreach \pt in
		    {(2.5-\sp,0),(2.5,0),
		    (3.5+\sp/2,0),(3.5-\sp/2,0),
		    (.5,\y),(2.5-\sp/2,\y),(2.5+\sp/2,\y),(3.5,\y),
		    (4.5,\y)}
		    {
		    	\draw[fill] \pt circle (1.4pt);
		    }
		    \draw[->,densely dotted, very thick, color = blue]
	    	(2.5+\sp-.03,-0.03) to [in=180, out=-60] (3,-.4) to [in=-120, out=0] (3.5-\sp/2-.03,-0.05);
	    	\draw[->,densely dotted, ultra thick]
	    	(3.5-\sp/2-.03,-0.05) to [in=-120, out=60] (4.5-.03,\y-0.05);
	    	\draw[->,densely dotted, ultra thick]
	    	(3.5-\sp/2,0) to [in=-80, out=170] (2.5+\sp/2+.02,\y-0.05);
	    	\node at (2.5-\sp,-.26) {$D$};
	    	\node at (2.5,\y+.3) {$D$};
	    	\node at (3.2, -.6) {\scriptsize\color{blue}just moved};
	    	\node at (3.35,.3) {$\la^{(m)}_i$};
	    	\draw (4.3,\y+.5) node[block] (r) {$\frac{1-\q^{-1}}{1-\q^{-D-1}}$};
	    	\draw (1.4,\y+.5) node[block] (l) {$1-\frac{1-\q^{-1}}{1-\q^{-D-1}}$};
	    	\draw[color=red] (l.south) -- (2.86,.21);
	    	\draw[color=red] (r.south) -- (4.13,.35);
		\end{tikzpicture}
		&\hspace{10pt}
		\begin{tikzpicture}[
		    scale=1.3,
		    axis/.style={thick, ->, >=stealth'},
		    block/.style ={rectangle, draw=red,
			align=center, rounded corners, minimum height=1em}]
		    \def\y{.7}
		    \draw[axis] (0,0) -- (5,0) node(xline)[right]{$\la^{(m)}$};
		    \draw[axis] (0,\y) -- (5,\y) node(xline)[right]{$\la^{(m+1)}$};
		    \foreach \hh in {.5, 1.5, 2.5, 3.5, 4.5}
		    {
		    	\draw[densely dotted, thick, opacity=.5] (\hh,-.13) -- (\hh,\y+.13);
		    }
		    \def\sp{.13};
		    \def\opac{.25}
			\draw[line width=1.3, color=blue, opacity=\opac]
			(0,\y/4)--(2.5-3*\sp/2,\y)--(2.5-\sp,0)--++(\sp/2,\y)--++(\sp/2,-\y)
			--++(\sp/2,\y)--(3.5-\sp/2,0)--++(\sp/2,\y)--++(\sp/2,-\y)
			--(4.5,\y)--(5,3*\y/4);
		    \foreach \pt in
		    {(2.5-\sp,0),(2.5,0),
		    (3.5+\sp/2,0),(3.5-\sp/2,0),
		    (2.5-3*\sp/2,\y),(2.5-\sp/2,\y),(2.5+\sp/2,\y),(3.5,\y),
		    (4.5,\y)}
		    {
		    	\draw[fill] \pt circle (1.4pt);
		    }
		    \draw[->,densely dotted, very thick, color = blue]
	    	(2.5+\sp-.03,-0.03) to [in=180, out=-60] (3,-.4) to [in=-120, out=0] (3.5-\sp/2-.03,-0.05);
	    	\draw[->,densely dotted, ultra thick]
	    	(3.5-\sp/2-.03,-0.05) to [in=-120, out=60] (4.5-.03,\y-0.05);
	    	\draw[->,densely dotted, ultra thick]
	    	(3.5-\sp/2,0) to [in=-80, out=170] (2.5+\sp/2+.02,\y-0.05);
	    	\node at (2.5-\sp,-.26) {$D$};
	    	\node at (2.5-\sp/2,\y+.3) {$D+1$};
	    	\node at (3.2, -.6) {\scriptsize\color{blue}just moved};
	    	\node at (3.35,.3) {$\la^{(m)}_{i}$};
	    	\draw (4.9,\y+.5) node[block] (r) {$1-\q^{-1}$};
	    	\draw (3.3,\y+.5) node[block] (l) {$\q^{-1}$};
	    	\draw[color=red] (l.south) -- (2.86,.21);
	    	\draw[color=red] (r.west) -- (4.13,.35);
		\end{tikzpicture}
	\end{tabular}
\end{center}
\caption{Pushing and pulling probabilities in the sampling algorithm
(note that the number $D$ can be zero).}
\label{fig:pushing_intro}
\end{figure}

For example, if the current state of the interlacing
array is as on Fig.~\ref{fig:interlacing_intro}
and the next inserted letter is $\xi_{20}=2$, then
the result of this insertion will be distributed as follows
(values that changed are framed):
\begin{align*}
	\framebox{\scalebox{.9}{
	\begin{tikzpicture}
	[scale=.9, thick]
	\def\y{.7}
	\def\t{.27}
	\def\g{11.5}
	\node at (\g,0) {5};
	\node at (\g+\t,\y) {6};\node at (\g-\t,\y) {\framebox{3}};
	\node at (\g+2*\t,2*\y) {9};\node at (\g-0*\t,2*\y) {\framebox{3}};\node at (\g-2*\t,2*\y) {2};
	\node at (\g+3*\t,3*\y) {10};\node at (\g+1*\t,3*\y) {\framebox{7}};\node at (\g-1*\t,3*\y) {2};\node at (\g-3*\t,3*\y) {1};
	\end{tikzpicture}}}
	\raisebox{30pt}{\quad\mbox{with prob. $\dfrac{1}{1+\q^{-1}}$,}}
	\hspace{50pt}
	\framebox{\scalebox{.9}{
	\begin{tikzpicture}
	[scale=.9, thick]
	\def\y{.7}
	\def\t{.27}
	\def\g{11.5}
	\node at (\g,0) {5};
	\node at (\g+\t,\y) {6};\node at (\g-\t,\y) {\framebox{3}};
	\node at (\g+2*\t,2*\y) {9};\node at (\g-0*\t,2*\y) {\framebox{3}};\node at (\g-2*\t,2*\y) {2};
	\node at (\g+3*\t,3*\y) {10};\node at (\g+1*\t,3*\y) {6};\node at (\g-1*\t,3*\y) {\framebox{3}};\node at (\g-3*\t,3*\y) {1};
	\end{tikzpicture}}}
	\raisebox{30pt}{\quad\mbox{with prob. $\dfrac{\q^{-1}}{1+\q^{-1}}$.}}
\end{align*}
Indeed, the propagation of the move from level $2$ to level $3$
is a mandatory pushing, and
from level $3$ to level $4$
the pushing or pulling probability is
determined as on Fig.~\ref{fig:pushing_intro} (left),
with $D=1$.
See also \S \ref{sub:an_example} for another example
involving general parameters.

\begin{theorem}\label{thm:it_samples_intro}
	If the randomized RSK insertion algorithm
	is applied to a random
	word $w=\xi_1\xi_2 \ldots\xi_n$ with independent
	letters $\xi_i\in\{1,2,\ldots,N\}$
	such that $P(\xi_i=k)=\al_k$,
	then the distribution of the
	top row $\la^{(N)}$ of the array
	is exactly $\HL_n^{\ab;{\boldsymbol 0};\Pl_0}$,
	i.e., the measure
	\eqref{HL_coherent_intro}
	with the parameters $\ab=(\al_1,\ldots,\al_N)$.
\end{theorem}
We prove this theorem
along with the corresponding statement for general parameters
in \S \ref{sub:full_sampling_algorithm}
(see Theorem \ref{thm:it_samples}).
It implies that the measures \eqref{HL_coherent_intro}
we are interested in are
images of Bernoulli measures on words
under a certain randomized algorithm
(this property is the same for general parameters).
Analyzing this algorithm, we prove that the
main contributions to lengths of rows and columns
under \eqref{HL_coherent_intro}
come from the Bernoulli measure part, and
thus arrive at the Law of Large Numbers
(Theorem \ref{thm:LLN_intro}).

Let us now make a number of remarks on
our randomized RSK insertion algorithm. 

\begin{remark}[Discrete analogue of Dyson's Brownian motion]
	In the course of our randomized RSK insertion algorithm,
	the top row $\la^{(N)}$ (i.e., the
	shape of the corresponding Young tableau)
	evolves according to a certain Markov chain
	which first appeared in Fulman
	\cite{fulman1997probabilistic}.
	This process (which we describe in
	\S \ref{sub:univariate_dynamics}) can be viewed as a
	certain discrete analogue
	of the Dyson's Brownian motion
	\cite{dyson1962brownian}.
\end{remark}

\begin{remark}[Symmetric groups]
	If the Hall--Littlewood parameter $t$
	(which is equal to $\q^{-1}$ for the purposes of
	studying random matrices over $F_\q$)
	is set to zero, then the randomized insertion
	becomes \emph{deterministic},
	and coincides with the classical RSK insertion
	with column insertion. We discuss the $t=0$ degeneration
	in detail
	in \S \ref{sub:symmetric_groups} below.
\end{remark}

\begin{remark}[Sampling algorithms for measures involving Hall--Littlewood functions]
\label{rmk:Borodin_Fulman}
	In the previous years, various
	(probabilistic) algorithms were constructed for
	sampling probability measures
	related to Hall--Littlewood symmetric
	functions (such as our measures $\HL_n^{\ab;\bb;\Pl_\ga}$).
	Analyzing these algorithms,
	one manages to extract certain specific properties
	of the sampled probability measures.
	A sampling algorithm
	for random Young diagrams corresponding to
	the uniform measure on $\Ub$ was
	introduced
	by Kirillov \cite{KirillovTriangular}
	and studied
	by Borodin
	\cite{Borodin1995},
	\cite{Borodin1995_ext}.
	It is well-adapted to proving
	Theorem \ref{thm:LLN_intro} for this uniform measure.

	There are also sampling algorithms constructed by
	Fulman
	\cite{Fulman1999},
	\cite{Fulman2000RM},
	which
	allow to obtain information
	about certain other probability measures
	on Young diagrams related
	to Hall--Littlewood symmetric functions.
	Namely, it is possible to derive
	(in some form) distributions
	of observables
	of random Young diagrams
	such as their row or column lengths.
	See also Fulman \cite{Fulma2013NT} for a recent connection
	of measures involving Hall--Littlewood symmetric functions
	to the Cohen-Lenstra heuristics of
	Number Theory.
\end{remark}

\begin{remark}[Randomized RSK insertions]
	Other randomized RSK insertion algorithms
	were also developed and studied by O'Connell
	and Pei \cite{OConnellPei2012},
	\cite{Pei2013Symmetry}.
	The latter paper
	also explains the
	randomized RSK insertion algorithms
	in a more traditional language of
	Young tableaux and
	Fomin's growth diagrams
	(about the latter see Fomin \cite{fomin1979thesis},
	\cite{Fomin1986},
	\cite{fomin1994duality},
	\cite{fomin1995schensted}).

	A family of algorithms	
	for sampling the measures
	$\HL_n^{\ab;\boldsymbol 0;\Pl_0}$
	was constructed in Borodin--Petrov \cite{BorodinPetrov2013NN}.
	A priori these algorithms involved negative transition probabilities.
	The sampling algorithm described above is present in \cite{BorodinPetrov2013NN} in a
	hidden form: it is singled out by requiring nonnegative transition probabilities.
	In \S\S \ref{sec:macdonald_processes_and_bivariate_continuous_time_dynamics_}--\ref{sec:three_particular_dynamics_on_macdonald_processes} we also generalise the mechanism of \cite{BorodinPetrov2013NN}
	to sample measures
	$\HL_n^{\ab;\bb;\Pl_\ga}$.
	The passage from parameters $\ab$ to parameters $\bb$
	is possible via
    a certain duality related to the
	transposition of Young diagrams, see
	\S \ref{sub:duality}. The Plancherel part $\Pl_\ga$ can be added by considering
	interlacing particle arrays with ``continuous floors'', cf. \S \ref{sub:adding_a_plancherel_parameter}.
	Our sampling algorithm for the measures $\HL_n^{\ab;\bb;\Pl_\ga}$
	lives on interlacing arrays generalizing the ones on Fig.~\ref{fig:interlacing_intro},
	see Fig.~\ref{fig:A_particle} in \S \ref{sec:rsk_type_algorithm_for_sampling_hl_coherent_measures}.
\end{remark}

\begin{remark}[Branching graphs]\label{rmk:branching_graphs}
	It is worth noting that the algorithm
	of Borodin which samples measures
	\eqref{HL_coherent_intro} for special parameters
	(Remark \ref{rmk:Borodin_Fulman}),
	differs significantly from our construction.
	Namely, the former
	employs the coherency
	property
	on the Young graph with Hall--Littlewood formal edge multiplicities
	which corresponds to adding one box to a Young diagram
	(cf. Remarks \ref{rmk:coupling_intro}, \ref{rmk:branching1},
	and \S \ref{sub:plancherel_specializations_and_young_graph}).
	Moreover, this algorithm can be viewed as a univariate
	dynamics of \S \ref{sub:univariate_dynamics}.

	On the other hand,
	our construction benefits from a
	connection to Gelfand-Tsetlin
	graph with another type of branching
	corresponding to adding horizontal strips to
	a Young diagram
	(with edge multiplicities also related
	to Hall--Littlewood symmetric functions).
	One can say that we
	work with Markov dynamics on paths
	in the Gelfand-Tsetlin graph
	(these are interlacing particle configurations),
	and our constructions are in the spirit of
	dynamics on interlacing particle arrays
	of
	Borodin--Ferrari
	\cite{BorFerr2008DF}, Borodin \cite{Borodin2010Schur},
	Borodin--Olshanski \cite{BorodinOlshanski2010GTs},
	Borodin--Corwin
	\cite{BorodinCorwin2011Macdonald}
	(they are based on an idea of
	Diaconis--Fill \cite{DiaconisFill1990}),
	and
	more general multivariate dynamics developed in
	Borodin--Petrov
	\cite{BorodinPetrov2013NN}.

	The interplay between
	the
	coherency
	property on the Young graph and
	its connections to the Gelfand-Tsetlin
	graph was employed
	in, e.g., Borodin--Gorin
	\cite{BG2011non}.
	Representation-theoretic
	consequences of connections
	between the Young and Gelfand-Tsetlin graphs are
	discussed in
	Borodin--Olshanski
	\cite{BorodinOlsh2011Bouquet}.
\end{remark}



\subsection{Symmetric groups} 
\label{sub:symmetric_groups}

Here and in the next subsection we briefly summarize
representation-theoretic constructions which lead to the
classification problem of \S \ref{sub:infinite_random_matrices_over_a_finite_field}.
We start with an analogous problem for
symmetric groups.
This part of the introduction is not essential for understanding
our main results and constructions.

\medskip

One of the central problems of the asymptotic
representation theory of symmetric groups $S(n)$
is to classify \emph{irreducible characters}
of the infinite symmetric group
$S(\infty)=\bigcup_{n=1}^{\infty}S(n)$.
Elements of $S(\infty)$ are permutations of the
infinite set $\{1,2,\ldots\}$ which move only finitely many numbers.
A character of $S(\infty)$ is a positive definite central function $\chi$
on $S(\infty)$ which is normalized by $\chi(e)=1$.
Characters of $S(\infty)$
form a convex set, and irreducible characters
are (by definition) extreme points of this set.
Irreducible characters correspond to finite factor representations
of $S(\infty)$, e.g., see \cite{VK81AsymptoticTheory}.


\begin{theorem}[Edrei \cite{Edrei1952}\footnote{See also Aissen--Edrei--Schoenberg--Whitney
\cite{AESW51}.} and Thoma \cite{Thoma1964}]\label{thm:thoma}
Irreducible characters of $S(\infty)$
are in one-to-one correspondence with
triplets $\tilde{\Ab}:=(\tilde\ab;\tilde{\boldsymbol\be};\tilde\ga)\in\R^{2\infty+1}$
such that
\begin{align}\label{alpha_beta_gamma_intro1_s}
	\tilde\ab=(\tilde\al_1\ge\tilde\al_2\ge \ldots\ge0),
	\qquad
	\tilde{\boldsymbol\be}=
	(\tilde\be_1\ge\tilde\be_2\ge \ldots\ge 0),
	\qquad\tilde\ga\ge0,
\end{align}
and
\begin{align}\label{alpha_beta_gamma_intro2_s}
	\sum_{i=1}^{\infty}\tilde\al_i+
	\sum_{i=1}^{\infty}\tilde\be_i+
	\tilde\gamma=1.
\end{align}
The correspondence is established
by restricting $\chi$ to the subgroup
$S(n)\subset S(\infty)$ permuting
the first $n$ numbers.
The restriction $\chi|_{_{S(n)}}$
can be decomposed into a convex combination of normalized irreducible
characters of $S(n)$ which are indexed by
Young diagrams $\la\in\Yb_n$:
\begin{align*}
	\chi|_{_{S(n)}}=\sum_{\la\in\Yb_n}
	\mathcal{S}_n^{\tilde{\Ab}}(\la)\frac{\chi_\la}{\dim\chi_\la},
\end{align*}
and the coefficients of this combination (this is a probability
measure on $\Yb_n$) are
\begin{align}
	\mathcal{S}_n^{\tilde{\Ab}}(\la)=
	n!\,
	s_\la(\tilde{\Ab})
	s_\la(\Pl_{1}),
	\label{Schur_coherent_intro}
\end{align}
where $s_\la$ is the Schur symmetric function, and the measure above is
given as the product of two specializations of $s_\la$ (see definitions in \S \ref{sec:preliminaries}).
\end{theorem}

The problem of classifying irreducible
characters of $S(\infty)$ can be also formulated in
equivalent terms of nonnegative specializations of
Schur symmetric functions \cite{VK1981Characters}, \cite{VK81AsymptoticTheory},
see also \cite{Kerov1998}.

Measures \eqref{Schur_coherent_intro} also satisfy a certain coherency
property on the Young graph, see \S \ref{sub:plancherel_specializations_and_young_graph}.
At the level of formulas \eqref{Schur_coherent_intro} and \eqref{HL_coherent_intro},
Theorem \ref{thm:thoma} is the degeneration of Conjecture \ref{conj:Spec_HL}
when the parameter $t=\q^{-1}$ is set to zero (then the Hall--Littlewood $P_\la$
and $Q_\la$ both become the
Schur function $s_\la$).

\begin{remark}
	For $t=0$, the randomized RSK insertion algorithm
	we develop
	is deterministic, it
	was introduced by Vershik and Kerov
	\cite{Vershik1986}
	in connection with Theorem \ref{thm:thoma}
	(and further exploited in, e.g., Sniady
	\cite{Sniady2013},
	see also
 	Romik--Sniady \cite{RomikSniady2011}).
	The Law of Large Numbers
	in this setting (i.e., an analogue of Theorem \ref{thm:LLN_intro}
	for $t=\q^{-1}$ being zero)
	was obtained earlier also by Vershik and Kerov \cite{VK1981Characters},
	by a direct investigation of measures
	\eqref{Schur_coherent_intro} (they are simpler than
	\eqref{HL_coherent_intro} in that they admit more direct explicit formulas).
	Bufetov	
	\cite{BufetovCLT}
	used the deterministic RSK insertion
	algorithm of \cite{Vershik1986}
	to establish a corresponding
	Central Limit Theorem (Conjecture
	\ref{conj:CLT_intro} for $t=0$).
\end{remark}

\begin{remark}
	We also note that a problem
	of classifying ergodic conjugation-invariant measures
	on Hermitian matrices over the \emph{complex numbers}
	(instead of $F_\q$ as in Conjecture \ref{conj:Spec_HL})
	was considered by Olshanski and Vershik
	\cite{OlVer1996}.
	This setup is also deeply
	related to Schur symmetric functions.
	Moreover, it arises as
	a degeneration in a certain sense
	of the problem coming from
	asymptotic representation theory of
	unitary groups (over complex numbers).
	About the latter problem, see Edrei
	\cite{Edrei53} (and also Aissen--Edrei--Schoenberg--Whitney \cite{AESW51},
	Aissen--Schoenberg--Whitney
	\cite{ASW52}),
	Voiculescu \cite{Voiculescu1976},
	Vershik--Kerov \cite{VK82CharactersU},
	Boyer \cite{Boyer1983},
	Okounkov--Olshanski \cite{OkOl1998},
	Borodin--Olshanski \cite{BorodinOlsh2011GT},
	Petrov \cite{Petrov2012GT},
	Gorin--Panova \cite{GorinPanova2012}.
\end{remark}


\subsection{Asymptotic representation theory of linear groups over a finite field} 
\label{sub:asymptotic_representation_theory_of_linear_groups_over_a_finite_field}

The desire to construct a meaningful
asymptotic representation theory of the groups
$GL(n,F_\q)$
(which is in some sense
a deformation of the corresponding theory for
symmetric groups; the latter one was briefly described in \S \ref{sub:symmetric_groups})
leads to considering various groups of infinite
matrices over $F_\q$
which play the role of a natural
$n=\infty$ analogue of
the groups
$GL(n,F_\q)$.
A direct analogue of $S(\infty)$, the group $GL(\infty,F_\q)$ (see the discussion after
Definition \ref{def:central}),
in fact leads to a poor representation theory, see Thoma \cite{Thoma1972} and
Skudlarek \cite{Skudlarek1976}.

First example of a ``right'' $n=\infty$ analogue
is the
group $\mathbb{GLB}$
(see Vershik and Kerov \cite{VK98}, and also Vershik's historical preface in \cite{GorinKerovVershikFq2012})
of all invertible almost upper-triangular matrices
over $F_\q$. Namely,
$\mathbb{GLB}$
consists of all matrices $X=[X_{ij}]_{i,j=1}^{\infty}$
whose upper $n\times n$ corner is invertible for a
large enough $n$, and, moreover, $X_{ij}=0$ for $i>j$ and $i>n$, and
$X_{ii}\ne 0$ for $i>n$.

A very similar representation
theory arises for another group, $\mathbb{GLU}$,
which consists of all matrices
$X=[X_{ij}]_{i,j=1}^{\infty}\in\mathbb{GLB}$ for which
$X_{ii}=1$ for large enough $i$.
For both groups, there is a natural notion of characters
which are traces of the so-called Schwartz-Bruhat algebra
of the group.
Principal (unipotent)
extreme
traces of this Schwartz-Bruhat algebra
are parametrized by
the same triplets
$(\tilde\ab;\tilde{\boldsymbol\be};\tilde\ga)\in\R^{2\infty+1}$
satisfying \eqref{alpha_beta_gamma_intro1_s}--\eqref{alpha_beta_gamma_intro2_s}
as for the infinite symmetric group,
see \cite{VK98}, \cite[Thm. 2.24]{GorinKerovVershikFq2012}.
A posteriori, when the classification is known, these unipotent extreme traces
are
identified with
extreme traces of the infinite-dimensional Iwahori-Hecke algebra
$\mathcal{H}_{\infty}(\q)$. Traces of the latter were classified in
Vershik--Kerov \cite{VK89Hecke}
and Meliot \cite[\S7]{MeliotCLT2011}.
See also \cite[\S3.3]{GorinKerovVershikFq2012} for the identification of two classifications.

Let us now make connection of this classification of extreme unipotent traces
to Conjecture \ref{conj:Spec_HL} which is open.
By \cite[Theorems 4.2 and 4.6]{GorinKerovVershikFq2012},
to every unipotent trace of $\mathbb{GLU}$
indexed by $(\tilde\ab;\tilde{\boldsymbol\be};\tilde\ga)$
as above corresponds a unique central probability
measure on $\Ub\subset\mathbb{GLU}$ (Definition \ref{def:central}).
Moreover, this central probability measure is ergodic, and
it is indexed by parameters
\begin{align*}
	\big\{\al_r\big\}_{r=1}^{\infty}=\big\{
	\tilde\al_i(1-\q^{-1})\q^{1-j}
	\big\}_{i,j=1}^{\infty},\qquad
	\be_i=\tilde\be_i(1-\q^{-1}),\qquad
	\ga=\tilde\ga(1-\q^{-1})
\end{align*}
in the sense of Conjecture \ref{conj:Spec_HL}.
(Note that this transformation is \emph{different} from the more straightforward
reparametrization
described in Remark \ref{rmk:Kerov_different} below.)

We refer to Vershik--Kerov
\cite{VK98} and Gorin--Kerov--Vershik \cite{GorinKerovVershikFq2012}
for further details and connections to asymptotic representation theory.


\subsection{Outline of the paper}

In \S \ref{sec:preliminaries} and \S \ref{sec:coherent_measures_on_partitions} we recall
necessary objects related to Young diagrams and Macdonald
(and Hall--Littlewood) symmetric functions.
In particular, in \S \ref{sec:coherent_measures_on_partitions}
we discuss our main object: \emph{coherent measures}
on Young diagrams related to Macdonald
symmetric functions (this is a generalization of
the measures \eqref{HL_coherent_intro}).
In \S \ref{sec:macdonald_processes_and_bivariate_continuous_time_dynamics_} we recall and extend
the general formalism of
\cite{BorodinPetrov2013NN}
for constructing Markov dynamics
which map coherent measures onto each other.
In \S \ref{sec:three_particular_dynamics_on_macdonald_processes}
and \S \ref{sec:rsk_type_algorithm_for_sampling_hl_coherent_measures} we construct our
randomized RSK insertion algorithm for sampling coherent measures.
In \S \ref{sec:proof_of_the_law_of_large_numbers} we employ this
sampling algorithm to prove the Law of Large Numbers.

\subsection{Acknowledgments}

This work was started at the 2013
Cornell Probability Summer School,
and we would like to thank the organizers for the
invitation and warm hospitality.
We are very grateful to Alexei Borodin, Jason Fulman, Vadim Gorin, Grigori Olshanski, and Anatoly Vershik for
helpful discussions.
We also would like to thank the anonymous referee
for extremely valuable suggestions on improving the presentation of our results.

A.B. was partially supported by Simons Foundation--IUM scholarship, by Moebius Foundation
for Young Scientists, by ``Dynasty'' foundation,
and by the RFBR grant 13-01-12449.



\section{Preliminaries} 
\label{sec:preliminaries}

\subsection{Young diagrams} 
\label{sub:young_diagrams}

Let $\Yb$ denote the set of all \emph{partitions}, i.e., integer sequences of the form
$\la=(\la_1\ge \la_2\ge \ldots\ge\la_{\ell(\la)}>0)$, where $\la_i\in\Z_{\ge0}$.
We always identify partitions with \emph{Young diagrams}
as in \cite[I.1]{Macdonald1995}, see also Fig.~\ref{fig:transpose}.
The number $\ell(\la)$
of nonzero components
of $\la$ is called the \emph{length} of the partition.
Also, let $|\la|:=\sum_{i=1}^{\ell(\la)}\la_i$ be the number of
boxes in the corresponding Young diagram.
When needed, we will append partitions by zeroes, and identify $\la$ with
$(\la_1,\ldots,\la_{\ell(\la)},0,0,\ldots)$.
The empty partition is denoted by $\varnothing =(0,0,\ldots)$.
For $n\ge0$, let $\Yb_n:=\{\la\in\Yb\colon |\la|=n\}$ be the set of Young diagrams with
$n$ boxes.

For two Young diagrams $\mu,\la$ such that $\ell(\mu)\le\ell(\la)$ and
$\mu_i\le\la_i$ for all $i=1,\ldots,\ell(\la)$, we
will write $\mu\subseteq\la$.
In this case, the set difference
of the diagram $\la$ and the diagram $\mu$
is denoted by $\la/\mu$ and called a \emph{skew Young diagram}.

If $\mu\subseteq\la$ and, moreover,
\begin{align}\label{interlace}
	\la_1\ge\mu_1\ge\la_2\ge\mu_2 \ge\ldots\ge
	\la_{\ell(\la)-1}\ge\mu_{\ell(\mu)}\ge
	\la_{\ell(\la)}
\end{align}
(this implies that $\ell(\la)=\ell(\mu)$ or $\ell(\la)=\ell(\mu)+1$),
then we say that the diagram $\la$ is obtained from $\mu$
by adding a \emph{horizontal strip} (or, equivalently, that
\emph{the skew diagram $\la/\mu$ is a horizontal strip}),
and denote this by $\mu\prech\la$.

\begin{figure}[htbp]
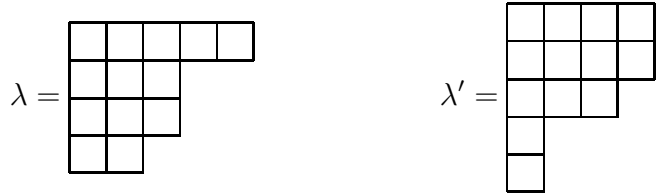

\begin{equation*}
	\la=\begin{array}{|c|c|c|c|c|}
	    \hline
	    \ &\ &\ &\ &\ \\
	    \hline
	    \ &\ &\ \\
	    \cline{1-3}
	    \ &\ &\ \\
		\cline{1-3}
		\ &\ \\
		\cline{1-2}
	\end{array}
	\qquad
	\qquad
	\qquad
	\la'=\begin{array}{|c|c|c|c|}
	    \hline
	    \ &\ &\ &\ \\
	    \hline
	    \ &\ &\ &\ \\
	    \hline
	    \ &\ &\ \\
	    \cline{1-3}
	    \ \\
		\cline{1-1}
		\ \\
		\cline{1-1}
	\end{array}
\end{equation*}
\caption{Young diagram $\la=(5,3,3,2)$ and its transpose
$\la'=(4,4,3,1,1)$.}
\label{fig:transpose}
\end{figure}

If $\la\in\Yb$ is represented by a Young diagram,
then, reflecting it with respect to the main diagonal,
one gets the \emph{transposed diagram} (see Fig.~\ref{fig:transpose}).

We say that two diagrams $\mu\subseteq\la$ differ by a \emph{vertical strip}
(equivalently, that \emph{the skew diagram $\la/\mu$ is a vertical strip})
and denote this by $\mu\precv\la$,
iff the transposed diagrams $\mu'\subseteq\la'$ differ by a horizontal strip.

For a Young diagram $\la$,
let $\Us(\la)$ (respectively, $\Ds(\la)$) denote the
set of all boxes that can be added to
(respectively, removed from) the diagram $\la$
in such a way that the result is \emph{again} a Young diagram (see Fig.~\ref{fig:U_D}).
\begin{figure}[htbp]
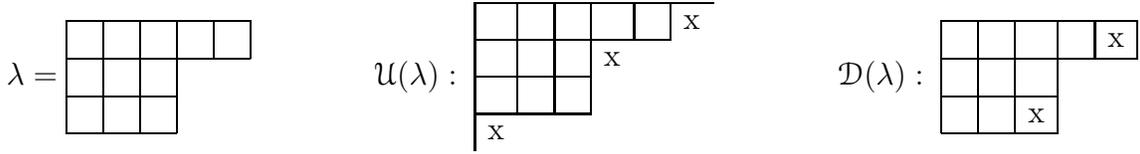

\begin{align*}
	\la=\begin{array}{|c|c|c|c|c|}
	    \hline
	    \ &\ &\ &\ &\ \\
	    \hline
	    \ &\ &\ \\
	    \cline{1-3}
	    \ &\ &\ \\
		\cline{1-3}
	\end{array}
	\qquad \qquad
	\Us(\la):\;\begin{array}{|c|c|c|c|c|c}
	    \cline{1-6}
	    \ &\ &\ &\ & \ &\mathrm{x}\\
	    \cline{1-5}
	    \ &\ &\ &\multicolumn{1}{c}{\mathrm{x}}\\
	    \cline{1-3}
	    \ &\ &\ \\
		\cline{1-3}
		\multicolumn{1}{|c}{\mathrm{x}}
	\end{array}
	\qquad \qquad
	\Ds(\la):\;\begin{array}{|c|c|c|c|c|}
	    \hline
	    \ &\ &\ &\ & \mathrm{x}\\
	    \hline
	    \ &\ &\  \\
	    \cline{1-3}
	    \ &\ &\mathrm{x} \\
		\cline{1-3}
	\end{array}
\end{align*}
\caption{Boxes that can be added to (or removed from)
a Young diagram.}
\label{fig:U_D}
\end{figure}
For $\la,\nu\in\Yb$,
we will write $\la\precb\nu$ if $\nu$ is obtained from $\la$ by
adding a box.
Note that adding a box is a very particular case of adding a horizontal
(or vertical) strip.\footnote{However, note that a horizontal
or a vertical strip is allowed to be empty.}
The operation of adding a box will be also denoted as
$\nu=\la+\square$, or, equivalently, as $\la=\nu-\square$.


\subsection{Macdonald symmetric functions} 
\label{sub:preliminaries}

Probability measures we consider in the present paper are described in
terms of nonnegative specializations of
Macdonald (and, in particular, Hall--Littlewood) symmetric functions.
Let us briefly recall the necessary definitions.
We refer to \cite{Macdonald1995} and \cite[\S2]{BorodinCorwin2011Macdonald}
for details.

By $\Sym$ denote the \emph{algebra of symmetric functions} over $\R$ \cite[I.2]{Macdonald1995}.
It is a commutative algebra
$\R[p_1,p_2,\ldots]$
generated by $1$ and by the (algebraically independent)
Newton power sums
\begin{align*}
	p_k(x_1,x_2,\ldots)=\sum_{i=1}^{\infty}x_i^{k},\qquad k=1,2,\ldots.
\end{align*}
Products of power sums $p_\la:=p_{\la_1}p_{\la_2}\ldots p_{\la_{\ell(\la)}}$,
where $\la\in\Yb$ (with the agreement $p_{\varnothing}=1$), form a linear basis in $\Sym$.
(All linear bases in $\Sym$
that we consider
will be
indexed by the set $\Yb$.)
The algebra $\Sym$ possesses a natural grading
which is defined by setting
$\deg p_k=k$, $k=1,2,\ldots$.

\begin{remark}\label{rmk:Sym_proj_lim}
	Alternatively,
	each element of $\Sym$ may be viewed as a
	symmetric
	formal power series
	in $x_1,x_2,\ldots$
	in which degrees of all monomials
	are bounded.
	If in a symmetric function
	$f(x_1,x_2,\ldots)$ all but finitely many (say, $N$) of the variables
	are set to zero, then we get a usual symmetric polynomial
	$f(x_1,\ldots,x_N)=f(x_1,\ldots,x_N,0,0,\ldots)$
	in finitely many variables.
	
	Moreover, every $f\in\Sym$ can be understood as a sequence of
	symmetric polynomials $f_N$ in $N$ variables, $N=1,2,\ldots$,
	such that $\sup_N\deg f_N<\infty$
	and the polynomials $\{f_N\}$ are compatible in the sense that
	$f_{N+1}(x_1,\ldots,x_N,0)=f_{N}(x_1,\ldots,x_N)$.\footnote{This
	means that $\Sym$
	is the \emph{projective limit}
	(in the category of graded algebras)
	of algebras of symmetric polynomials
	in growing number of variables.}
\end{remark}

A remarkable two-parameter family of linear bases in $\Sym$
is formed by the \emph{Macdonald symmetric functions} \cite[VI]{Macdonald1995}.
Let $q,t\in[0,1)$. Consider a bilinear scalar product
$\langle\cdot,\cdot\rangle_{q,t}$
in $\Sym$ defined on $\{p_\la\}$
by
\begin{align*}
	\langle p_\la,p_\mu\rangle_{q,t}:=\mathbf{1}_{\la=\mu}z_{\la}(q,t),\qquad
	z_\la(q,t):=\bigg(\prod_{i\ge1}i^{m_i}(m_i)!\bigg)
	\cdot
	\bigg(\prod_{i=1}^{\ell(\la)}
	\frac{1-q^{\la_i}}{1-t^{\la_i}}\bigg),
\end{align*}
where $\la=(1^{m_1}2^{m_2}\ldots)$ means that $\la$ has $m_1$ parts equal to 1, $m_2$ parts equal to 2, etc.

\begin{definition}\label{def:Macd_symm_f}
	The \emph{Macdonald symmetric functions}
	$P_\la(\x\,|\,q,t)$
	(where $\x=(x_1,x_2,\ldots)$
	and $\la$ runs over all partitions)
	form a unique family of
	homogeneous symmetric functions
	such that:
	\begin{enumerate}[(1)]
		\item The functions $\{P_\la\}$ are pairwise
		orthogonal with respect to
		the scalar product $\langle\cdot,\cdot\rangle_{q,t}$.
		\item For every $\la$, we have
		\begin{align*}
			P_\la(\x\,|\,q,t)=
			x_1^{\la_1}\ldots x_{\ell(\la)}^{\la_{\ell(\la)}}{}+
			{}\mbox{lower monomials in lexicographic order}.
		\end{align*}
		The dependence on the parameters $(q,t)$
		is in coefficients of the
		lexicographically
		lower monomials.\footnote{Lexicographic order
		means that, for example, $x_1^{2}$ is higher
		than $\mathrm{const}\cdot x_1x_2$
		which is in turn higher
		than $\mathrm{const}\cdot x_2^{2}$.}
	\end{enumerate}
	When this does not lead to a confusion,
	we will omit
	the notation $(q,t)$,
	and simply write $P_\la(\x)$
	or $P_\la$ instead of $P_\la(\x\,|\,q,t)$.
	
	Also define
	$Q_\la:=
	{P_\la}/{\langle P_\la,P_\la \rangle_{q,t}}$,
	so that the functions $P_\la$ and $Q_\mu$ are
	orthonormal.
\end{definition}

In view of Remark \ref{rmk:Sym_proj_lim}, one can also speak about the Macdonald
symmetric polynomials $P_\la(x_1,\ldots,x_N\,|\,q,t)$. They can be alternatively
defined as eigenfunctions of certain $q$-difference operators
\cite[VI.3]{Macdonald1995}.

There are several
important special cases of the
parameters $(q,t)$. We are mainly interested in
one of them corresponding to setting the first parameter
$q$ to zero.
Then the Macdonald symmetric functions become the \emph{Hall--Littlewood
symmetric functions} \cite{Littlewood1961}, \cite[III]{Macdonald1995}.
If one further sets $t=0$ (or, equivalently, takes the Macdonald symmetric functions with $q=t$),
then one gets the \emph{Schur symmetric functions}.
In contrast with the general
Macdonald case,
both the Hall--Littlewood and Schur symmetric functions
admit rather explicit formulas
(see I.(3.1) and III.(2.1) in \cite{Macdonald1995},
respectively),
but we will not use them.
See also \S \ref{sub:completeness} for
other interesting particular cases of the Macdonald parameters $(q,t)$.

\begin{definition}\label{def:skew_Macd_symm_f}
	A \emph{skew Macdonald symmetric function} $Q_{\la/\mu}$ indexed by $\mu,\la\in\Yb$ is defined as the only symmetric function such that $\langle Q_{\la/\mu},P_\nu\rangle_{q,t}=\langle Q_{\la},P_\mu P_\nu\rangle_{q,t}$ for all $\nu\in\Yb$.
	The $P$ version is then defined through $Q_{\la/\mu}$
	as $P_{\la/\mu}:=
	\dfrac{\langle P_\la,P_\la\rangle_{q,t}}
	{\langle P_\mu,P_\mu\rangle_{q,t}}Q_{\la/\mu}$.
	Skew functions vanish unless $\mu\subseteq\la$.
	One also has $P_{\la/\varnothing}=P_\la$ and $Q_{\la/\varnothing}=Q_\la$.
\end{definition}


\subsection{Specializations of $\Sym$} 
\label{sub:specializations_of_sym_}

By a \emph{specialization} of the algebra $\Sym$ we mean an algebra homomorphism $\Ab\colon\Sym\to\R$. Such a map is completely determined by its values $\Ab(p_k)$ on the power sums. The \emph{trivial} specialization $\varnothing$ is defined as taking value 1 at the constant function $1\in\Sym$ and sending all the power sums $p_k$, $k\ge1$, to zero.

For two specializations $\Ab_1$ and $\Ab_2$, we define their \emph{union} $\Ab=(\Ab_1,\Ab_2)$
(sometimes we will also use the notation $\Ab_1\cup\Ab_2$)
as the specialization defined on power sums as
\begin{align*}
p_k(\Ab_1,\Ab_2)=p_k(\Ab_1)+p_k(\Ab_2), \qquad k\ge 1.
\end{align*}

If $\Ab$ is a specialization, define its \emph{multiple} $a\cdot \Ab$ (where $a\in\R$)
by requiring that on homogeneous functions $f\in\Sym$,
$f(a\cdot \Ab)=a^{\deg f}f(\Ab)$.

Important examples of specializations are the so-called \emph{finite length specializations}
$\Ab_{y_1,\ldots,y_N}$, where
$y_1,\ldots,y_N\in\R$,
defined as follows.
For $f\in\Sym$, let $f_N$ be the corresponding symmetric polynomial
in $N$ variables (see~Remark \ref{rmk:Sym_proj_lim}).
The image of $f$ under $\Ab_{y_1,\ldots,y_N}$
is
\begin{align}\label{finite_length_spec}
	f\mapsto f_N(y_1,\ldots,y_N).
\end{align}
The finite length specializations suggest the notation:
For $f\in\Sym$ and a specialization $\Ab$ we will write
$f(\Ab)$ instead of $\Ab(f)$.
For finite length specializations we will use a more
intuitive notation $f(y_1,\ldots,y_N)$ instead of
$f(\Ab_{y_1,\ldots,y_N})$.

\begin{definition}\label{def:nonneg_spec}
	A specialization $\Ab$ of $\Sym$ is said to be
	\emph{$(q,t)$-nonnegative}\footnote{Sometimes we will also use the term \emph{Macdonald-nonnegative},
	cf. \cite[\S 2.2.1]{BorodinCorwin2011Macdonald}.}
	if $P_{\la/\mu}(\Ab\,|\,q,t)\ge0$ for any partitions
	$\la,\mu\in\Yb$.
	The set
	\begin{align}\label{spec_support}
		\Yb(\Ab):=\{\la\in\Yb\colon
		P_\la(\Ab\,|\,q,t)>0\}
	\end{align}
	is the \emph{support} of a specialization $\Ab$.
\end{definition}
There is no known classification of $(q,t)$-nonnegative specializations.
However,
a wide class of such specializations
was introduced by Kerov \cite[II.9]{Kerov-book},
and he
conjectured that they exhaust all Macdonald-nonnegative specializations
(see also \S \ref{sub:completeness} below for more discussion).

These specializations depend on nonnegative parameters
$\{\al_i\}_{i\ge1}$, $\{\be_i\}_{i\ge1}$ and $\ga$ such that
$\sum_{i=1}^{\infty}(\al_i+\be_i)<\infty$.
For definiteness, we will always assume that
$\al_1\ge\al_2\ge \ldots\ge0$ and
$\be_1\ge\be_2\ge \ldots\ge0$.
The corresponding specialization is
defined on the power sums via the exponent of a generating function
(in a formal variable $u$) as follows:
\begin{align}\label{Pi_nonneg_spec}
	\exp\bigg(
	\sum_{n=1}^{\infty}\frac{1}{n}
	\frac{1-t^{n}}{1-q^{n}}p_n(\Ab)u^n
	\bigg)=
	\exp(\gamma u) \prod_{i\ge 1} \frac{(t\alpha_iu;q)_\infty}{(\alpha_i u;q)_\infty}\,(1+\beta_i u)=: \Pi(u;\Ab),
\end{align}
where the (infinite) $q$-Pochhammer symbol is defined as
\begin{align*}
	(a;q)_{\infty}:=\prod_{i=0}^{\infty}(1-aq^{i})=(1-a)(1-aq)(1-aq^{2})\ldots.
\end{align*}
In more detail, \eqref{Pi_nonneg_spec} means that
\begin{align}\label{al_be_ga_Newton}
	p_1(\Ab)&=\sum_{i\ge1}\al_i+\bigg(\ga+\sum_{i\ge1}\be_i\bigg)\frac{1-q}{1-t},
	&p_k(\Ab)=\sum_{i\ge1}\al_i^{k}+
	(-1)^{k-1}\frac{1-q^{k}}{1-t^{k}}\sum_{i\ge1}\be_i^{k},
\end{align}
where $k=2,3,\ldots$.
It can be verified that \eqref{Pi_nonneg_spec} defines $(q,t)$-nonnegative specializations, cf. \cite[Prop. 2.2.2]{BorodinCorwin2011Macdonald}.
In the Hall--Littlewood case (i.e., when $q=0$),
the product in \eqref{Pi_nonneg_spec} turns into
$\Pi(u;\Ab)=
e^{\ga u} \prod_{i\ge 1}
\frac{1-t\al_i u}{1-\al_i u}(1+\beta_i u)$.

\begin{remark}\label{rmk:finite_length_spec_alpha}
	When $\ga=0$, all $\be_i=0$, and only finitely many of the $\al_i$'s are nonzero, then the specialization defined by \eqref{Pi_nonneg_spec} reduces to a finite length specialization
	\eqref{finite_length_spec}.
\end{remark}
In view of Remark \ref{rmk:finite_length_spec_alpha}, we will refer to the $\al_i$'s as to the
\emph{usual variables}. We will also call the $\be_i$'s the \emph{dual variables}
(the name is motivated by the presence of a certain duality
involving transposition of Young diagrams, see
\S \ref{sub:duality} below).\footnote{Sometimes 
to emphasize that we are working with dual variables,
we will use the hat notation. For example, 
a dual variable equal to $1$ will be denoted by $\hat1$.} The parameter $\ga$
will be called the \emph{Plancherel parameter}.
We will denote by
$(\ab;\bb;\Pl_\ga)$
the specialization
defined by \eqref{Pi_nonneg_spec}
with parameters $\ab=(\al_1,\al_2,\ldots)$, $\boldsymbol\be=(\be_1,\be_2,\ldots)$, and $\ga$.
We will always assume that the specialization $(\ab;\bb;\Pl_\ga)$
is \emph{nontrivial} (i.e., not all of the parameters are equal to zero).
This property is equivalent to requiring that
$p_1(\ab;\bb;\Pl_\ga)>0$.
If all $\al_i$ and all $\be_j$ are zero,
we will call such a specialization a \emph{pure Plancherel}
specialization, and will denote it simply by $\Pl_\ga$.

Clearly, a multiple $a\cdot (\ab;\bb;\Pl_\ga)$ of the specialization
$(\ab;\bb;\Pl_\ga)$ corresponds to multiplying all the parameters
$\al_i,\be_i$, and $\ga$ by this factor $a$.
Union of specializations
$(\ab;\bb;\Pl_\ga)\cup(\ab';\bb';\Pl_{\ga'})$
leads to the new parameters $\ab\cup\ab'$, $\boldsymbol\be\cup\boldsymbol\be'$
(these are unions as sets),
and to the addition of the Plancherel parameters $\ga$ and $\ga'$.

\begin{remark}\label{rmk:Kerov_different}
	Note that our notation (borrowed from \cite{BorodinCorwin2011Macdonald}
	and also used in \cite{BorodinPetrov2013NN})
	differs from the one used by Kerov
	\cite{Kerov-book}, see also, e.g.,
	\cite{GorinKerovVershikFq2012}.
	Namely, take a specialization $\Ab=(\ab;\bb;\Pl_\ga)$
	described by 
	\eqref{al_be_ga_Newton}, and consider 
	other parameters $\tilde\al_i$, $\tilde\be_j$
	and $\tilde \ga$ defined as
	\begin{align*}
		\tilde\gamma=\frac{1-q}{1-t}\gamma,
		\qquad
		\big\{\tilde \al_{r}\big\}_{r=1}^{\infty}=
		\big\{\al_{i}\big\}_{i=1}^{\infty}
		\cup
		\big\{{-q}\be_{i}t^{j-1}\big\}_{i,j=1}^{\infty}
		,\qquad
		\big\{\tilde \be_{r}\big\}_{r=1}^{\infty}=
		\big\{\be_i t^{j-1}\big\}_{i,j=1}^{\infty}.
	\end{align*}
	(In the Hall--Littlewood ($q=0$) case,
	passing to these new parameters
	reduces to rescaling the Plancherel parameter
	and replacing each $\be_i$ by the geometric
	sequence $\be_i,\be_it,\be_it^{2},\ldots$.)
	Then the map
	\begin{align}
		p_k\mapsto\tilde\ga\mathbf{1}_{k=1}+\sum_{i\ge1}\tilde\al_i^{k}+
		(-1)^{k-1}\sum_{i\ge1}\tilde\be_i^{k}, \qquad k\ge1
		\label{al_be_ga_Newton_Kerov}
	\end{align}
	is the same as \eqref{al_be_ga_Newton}.
	References 
	\cite{Kerov-book}
	and 
	\cite{GorinKerovVershikFq2012}
	use parametrization \eqref{al_be_ga_Newton_Kerov}
	of specializations. 
\end{remark}


\subsection{Remark: Completeness of the list of $(q,t)$-nonnegative specializations} 
\label{sub:completeness}

The fact that specializations \eqref{Pi_nonneg_spec} indeed exhaust all possible
$(q,t)$-nonnegative specializations was established in the following particular cases
of parameters $q$ and $t$:
\begin{enumerate}[(1)]
	\item $t=q^{\theta}$ and $q\to 1$, where $\theta >0$ is a new parameter \cite{Kerov1998}.
	In this case the Macdonald symmetric functions reduce to the Jack symmetric functions
	introduced in
	\cite{Jack1}, \cite{Jack2} (see also
	\cite[VI.10]{Macdonald1995}).
	\item When $\theta=1$ in (1), the Jack symmetric functions become the Schur symmetric functions.
	The statement about nonnegative specializations in this case is equivalent to the
	classification of totally nonnegative triangular Toeplitz matrices
	\cite{Edrei1952}, \cite{AESW51},
	and to the classification of extreme characters of the infinite symmetric
	group
	\cite{Thoma1964}. See also
	\cite{VK81AsymptoticTheory}, \cite{VK1981Characters}.
	\item When $q=0$ and $t=1$, the Macdonald polynomials degenerate to the
	monomial symmetric functions. The classification of $(0,1)$-nonnegative specializations
	is equivalent to classification of partition structures in the sense of Kingman \cite{Kingman1978}
	(see also \cite{Kerov1989}).
	In this case, the parameters $\{\be_i\}$ do not enter the classification.
	One can also view this as a particular case $\theta=0$ in (1).
	\item Another interesting particular case is $q=0$ and $t=-1$, and the
	corresponding
	classification
	result is given in
	\cite{Nazarov1992}
	(see also \cite{IvanovNewYork3517-3530}).
	The $(0,-1)$-nonnegative specializations are related to projective characters
	of the infinite symmetric group.
\end{enumerate}
Cases (2), (3), and (4) above fall under the general Hall--Littlewood
picture which corresponds to $q=0$.\footnote{However, in the present paper
we restrict ourselves to $t\in[0,1)$, which excludes cases (3) and (4)
from the consideration.}
The classification result for $(0,t)$-nonnegative specializations
(we also refer to them as to \emph{HL-nonnegative specializations})
has not been proven for general values of the parameter~$t$.

When $t=\q^{-1}=p^{-d}\in(0,1)$ is the inverse of a prime power, the classification
of HL-nonnegative specializations is related
to random infinite triangular matrices over the finite
field $F_{\q}$, see \S \ref{sub:infinite_random_matrices_over_a_finite_field}.



\section{Coherent measures on partitions} 
\label{sec:coherent_measures_on_partitions}

\subsection{$(q,t)$-coherent measures} 
\label{sub:_q_t_coherent_measures}

For each fixed $n\ge0$,
consider a probability measure on partitions
with $n$ boxes defined as follows:
\begin{align}\label{coherent_def}
	\M_n^{\ab;\bb;\Pl_\ga}(\la):=
	\frac{n!}{\big(p_1(\ab;\bb;\Pl_\ga)\big)^{n}}\,
	P_\la(\ab;\bb;\Pl_\ga\,|\,q,t)
	Q_\la(\Pl_{1}\,|\,q,t)
	,\qquad
	\la\in\Yb_n.
\end{align}
Here the
first specialization
$(\ab;\bb;\Pl_\ga)$ is any Macdonald nonnegative specialization
defined by
\eqref{Pi_nonneg_spec},
the second specialization
$\Pl_{1}$ is the pure Plancherel specialization with parameter
$\ga=1$,
and $p_1(\ab;\bb;\Pl_\ga)$ is given in \eqref{al_be_ga_Newton}.
Note that $p_1(\ab;\bb;\Pl_\ga)$ also depends on $(q,t)$,
but we omit this dependence.

\begin{lemma}\label{lemma:sum_to_one}
	Expression \eqref{coherent_def} indeed defines a probability measure
	on $\Yb_n$, i.e.,
	\begin{align*}
		\mbox{$\M_n^{\ab;\bb;\Pl_\ga}(\la)\ge 0$\quad for all $\la\in\Yb_n$,\quad and\quad}\sum_{\la\in\Yb_n}
		\M_n^{\ab;\bb;\Pl_\ga}(\la)=1.
	\end{align*}
\end{lemma}
\begin{proof}
	The nonnegativity follows from the fact that $(\ab;\bb;\Pl_\ga)$
	is a nonnegative specialization.
	To show that the weights sum to one, we use the identity
	\begin{align*}
		n!\sum_{\la\in\Yb_n}
		P_\la(\x\,|\,q,t)
		Q_\la(\Pl_{1}\,|\,q,t)
		=(p_1(\x))^{n}
	\end{align*}
	which is a particular case of Lemma \ref{lemma:skew_Plancherel} below
	(corresponding to setting $\la=\varnothing$ in \eqref{skew_Plancherel_identity}),
	and take $\x=(\ab;\bb;\Pl_\ga)$.
\end{proof}
We will call
\eqref{coherent_def} the \emph{$(q,t)$-coherent measures}
(about the name, see \S \ref{sub:plancherel_specializations_and_young_graph}).
In the Hall--Littlewood case $q=0$, we will refer to the $(0,t)$-coherent measures as
to the \emph{HL-coherent measures}, and will denote them by
$\HL_n^{\ab;\bb;\Pl_\ga}$.
The HL-coherent measures are the main object of the present paper,
they are
related to
random infinite triangular
matrices over a finite field, see~\S \ref{sub:infinite_random_matrices_over_a_finite_field}, and \eqref{HL_coherent_intro}
in particular.\footnote{Note that
because $P_\la$ is a multiple of $Q_\la$,
\eqref{coherent_def} reduces to
\eqref{HL_coherent_intro} when $q=0$,
$t=\q^{-1}$, and $p_1(\ab;\bb;\Pl_\ga)=1$.}

\begin{remark}\label{rmk:scale_inv_of_coherent}
	By the homogeneity of
	$p_1$ and $P_\la$ in \eqref{coherent_def},
	the measure $\M_n^{\ab;\bb;\Pl_\ga}$
	is invariant under multiplication of
	the specialization $(\ab;\bb;\Pl_\ga)$
	by any positive number.
	Thus, to simplify certain formulas below,
	we will sometimes assume that
	the specialization
	is such that $p_1(\ab;\bb;\Pl_\ga)=1$.
\end{remark}


\subsection{Poissonization and Macdonald measures} 
\label{sub:poissonization_and_macdonald_measures}

Let $\tau>0$ be a new parameter (later it will play the role of time),
and let us \emph{mix} the measures
$\M_n^{\ab;\bb;\Pl_\ga}$
by means of the Poisson distribution with the parameter
$\tau p_1(\ab;\bb;\Pl_\ga)$
on the set of indices $n$:\footnote{The reason
for the multiplication of $\tau$ by $p_1(\ab;\bb;\Pl_\ga)$ is the future convenience of certain formulas.
Note that for now we are not assuming that $p_1(\ab;\bb;\Pl_\ga)=1$
(see Remark \ref{rmk:scale_inv_of_coherent}).}
\begin{align}\label{poissonized_coherent_def}
	\MM_\tau^{\ab;\bb;\Pl_\ga}:=
	e^{-\tau  p_1(\ab;\bb;\Pl_\ga)}
	\sum_{n=0}^{\infty}
	\frac{\big(\tau  p_1(\ab;\bb;\Pl_\ga)\big)^{n}}{n!}
	\,\M_n^{\ab;\bb;\Pl_\ga}.
\end{align}
That is, $\MM_\tau^{\ab;\bb;\Pl_\ga}$ is the probability measure
on the set $\Yb$ of \emph{all} Young diagrams, and
from \eqref{coherent_def} we get
\begin{align}\label{poissonized_of_lambda}
	\MM_\tau^{\ab;\bb;\Pl_\ga}(\la)=
	e^{-\tau  p_1(\ab;\bb;\Pl_\ga)}
	P_\la(\ab;\bb;\Pl_\ga\,|\,q,t)
	Q_\la(\Pl_{\tau}\,|\,q,t),\qquad
	\la\in\Yb.
\end{align}
The poissonized measures \eqref{poissonized_of_lambda}
belong to the class of \emph{Macdonald measures}
of \cite{BorodinCorwin2011Macdonald},
see also \cite{fulman1997probabilistic}.
(This is why we use the notation $\MM$.)
One can recover $\M_n$ from $\MM_\tau$
by conditioning on the event that the Young diagram
$\la$ distributed according to $\MM_\tau$
has exactly $n$ boxes.
\begin{remark}\label{rmk:depoiss}
	The passage from
	$\MM_\tau$ to $\M_n$ may be called \emph{de-poissonization}.
	There are analytic tools relating poissonized
	and de-poissonized
	measures (e.g., see \cite{baik1999distribution}),
	but for the purposes of the Law of Large Numbers
	(Theorem \ref{thm:LLN_intro})
	we do not need to employ them.
\end{remark}
We continue the discussion of Macdonald measures in \S \ref{sub:macdonald_measures} below.



\section{Macdonald processes and bivariate continuous-time `dynamics'} 
\label{sec:macdonald_processes_and_bivariate_continuous_time_dynamics_}

Here we recall and extend the general formalism of
\cite{BorodinPetrov2013NN}
for constructing \emph{formal} continuous-time Markov jump `dynamics'
which map Macdonald processes to Macdonald processes
(with evolved parameters).
Formality means that we allow `dynamics'
to have negative `jump rates' or `transition probabilities'
(we will indicate the absence
of the positivity
assumption with single quotation marks).
All our results
can be restated in linear algebraic terms,
as statements about action of formal Markov semigroups (that is, we allow the presence of negative numbers
in transition matrices)
on probability measures.
However, to make the discussion more
understandable, we will use probabilistic
language
even when speaking
about formal `dynamics'.

\begin{remark}
The
sampling algorithm
which we construct
at the Hall--Littlewood ($q=0$) level in \S\ref{sec:rsk_type_algorithm_for_sampling_hl_coherent_measures} below
involves only nonnegative probabilities.
\end{remark}

\subsection{Macdonald measures and Macdonald processes} 
\label{sub:macdonald_measures}

Let $\Ab$ be a $(q,t)$-non\-negative specialization of the algebra of symmetric functions
(\S \ref{sec:preliminaries}).
Let $\tau\ge0$ be a parameter.
We will consider the following \emph{Macdonald measures}
on Young diagrams:
\begin{align}\label{Macdonald_Measures}
	\MM_\tau^{\Ab}(\la):=\frac{P_\la(\Ab\,|\,q,t)Q_\la(\Pl_\tau\,|\,q,t)}{\Pi(\Ab;\Pl_\tau)},\qquad
	\la\in\Yb.
\end{align}
Here for any two specializations $\Ab,\Bb$ we have set
\begin{align}\label{Pi_A_B}
	\Pi(\Ab;\Bb):=\exp\left(\sum_{n=0}^{\infty}\frac{1}{n}\frac{1-t^n}{1-q^n}\,p_n(\Ab)p_n(\Bb)\right)
\end{align}
provided that this expression is
finite.\footnote{
Note that the expression $\Pi(u;\Ab)$ in \eqref{Pi_nonneg_spec}
is a particular case of \eqref{Pi_A_B}
corresponding to $\Bb=(u)$,
a specialization into a single usual variable.}
The normalization of the measures \eqref{Macdonald_Measures}
follows from
the Cauchy identity
\cite[VI]{Macdonald1995}
\begin{align}\label{Cauchy}
	\sum_{\la\in\Yb}
	P_\la(\Ab\,|\,q,t)Q_\la(\Bb\,|\,q,t)=\Pi(\Ab;\Bb).
\end{align}

Note that we consider only a particular case of the Macdonald measures
when one of the specializations is a pure Plancherel specialization.
In this case
$\Pi(\Ab;\Pl_\tau)=e^{\tau p_1(\Ab)}<\infty$ for any specialization $\Ab$,
cf. \S \ref{sub:poissonization_and_macdonald_measures}.
See also \cite[\S2]{BorodinCorwin2011Macdonald}, \cite{BCGS2013} about more general Macdonald measures.

\begin{remark}\label{rmk:tau=0}
	For $\tau=0$, the measure $\MM_{\tau}^{\Ab}$ \eqref{Macdonald_Measures}
	is concentrated on the empty diagram $\varnothing\in\Yb$.
	For any $\tau>0$, the support of
	this measure
	coincides with the support $\Yb(\Ab)\subseteq\Yb$
	of the specialization $\Ab$ (see \eqref{spec_support}) because the value of $Q_\la(\Pl_\tau\,|\,q,t)$ is strictly positive for all $\lambda \in \mathbb Y$.	
\end{remark}

Let now $\Ab$ and $\Bb$ be two
$(q,t)$-nonnegative
specializations with $\Pi(\Ab;\Bb)<\infty$.
There is a certain
\emph{stochastic link}
mapping the measure
$\MM_\tau^{\Ab\cup\Bb}$ to $\MM_\tau^{\Ab}$.
(Here $\Ab\cup\Bb$ is the union of specializations, cf. \S \ref{sub:specializations_of_sym_}.)
Namely, consider the following matrices with rows and columns indexed by Young diagrams:
\begin{align}\label{stoch_links}
	\La^{\Ab\cup\Bb}_{\Ab}(\la,\bar\la):=
	\frac{P_{\bar\la}(\Ab)}
	{P_\la(\Ab\cup\Bb)}P_{\la/\bar\la}(\Bb),\qquad
	\la\in\Yb(\Ab\cup\Bb),\quad \bar\la\in\Yb(\Ab).
\end{align}
\begin{proposition}[{\cite[\S 2.3.1]{BorodinCorwin2011Macdonald}}]\label{prop:links_properties}
	The quantities
	$\La^{\Ab\cup\Bb}_{\Ab}(\la,\bar\la)$
	are nonnegative and
	\begin{align*}
		\sum_{\bar\la\in\Yb(\Ab)}\La^{\Ab\cup\Bb}_{\Ab}(\la,\bar\la)=1
	\end{align*}
	(hence the name ``stochastic link'').
	Moreover,
	\begin{align}\label{MM_stoch_links}
		\MM_\tau^{\Ab\cup\Bb}\La^{\Ab\cup\Bb}_{\Ab}=\MM_\tau^{\Ab}.
	\end{align}
	The latter identity is understood in the matrix sense, and the measures
	$\MM_\tau^{\Ab\cup\Bb}$ and $\MM_\tau^{\Ab}$ should be viewed as
	row vectors.
\end{proposition}

One should understand \eqref{MM_stoch_links} as a \emph{compatibility relation}
between the Macdonald measures $\MM_\tau^{\Ab\cup\Bb}$ on $\Yb(\Ab\cup\Bb)$
and $\MM_\tau^{\Ab}$ on $\Yb(\Ab)$.
Now let us consider the \emph{joint distribution}
of the pair of Young diagrams $(\bar\la,\la)$ which arises from that
relation.
This joint distribution is supported on the subset\footnote{Note
that the condition
$\bi{\bar\la}{\la}\in\Yb^{(2)}(\Ab;\Bb)$
implies that $\bar\la\subseteq\la$.}
\begin{align}\label{2_support}
	\Yb^{(2)}(\Ab;\Bb):=\left\{
	\bi{\bar\la}{\la}\in
	\Yb(\Ab)\times\Yb(\Ab\cup\Bb)\colon
	\La^{\Ab\cup\Bb}_{\Ab}(\la,\bar\la)>0\right\}
	\subseteq \Yb(\Ab)\times\Yb(\Ab\cup\Bb),
\end{align}
and is given by
\begin{align}\label{Macdonald_process}
	\MP_{\tau}^{\Ab;\Bb}(\bar\la;\la):=
	\MM_\tau^{\Ab\cup\Bb}(\la)
	\La^{\Ab\cup\Bb}_{\Ab}(\la,\bar\la)=
	\frac{P_{\bar\la}(\Ab)P_{\la/\bar\la}(\Bb)Q_{\la}(\Pl_\tau)}
	{\Pi(\Ab;\Pl_\tau)\Pi(\Bb;\Pl_\tau)},\qquad
	\bi{\bar\la}{\la}\in\Yb^{(2)}(\Ab;\Bb).
\end{align}
This distribution
is a particular case of a
\emph{Macdonald process}
introduced in
\cite[\S 2.2]{BorodinCorwin2011Macdonald} (see also~\cite{BCGS2013}).

If in the definition
\eqref{Macdonald_process} one replaces the measure
$\MM_\tau^{\Ab\cup\Bb}(\la)$ by \emph{any} probability measure
on $\Yb(\Ab\cup\Bb)$,
then the resulting
measure on
$\Yb^{(2)}(\Ab;\Bb)$
will be \emph{compatible} with $\La^{\Ab\cup\Bb}_{\Ab}$
in a way similar to \eqref{MM_stoch_links}.
We will refer to this wider class of measures
as to the \emph{Gibbs measures}.


\subsection{Univariate dynamics} 
\label{sub:univariate_dynamics}

Let us now describe certain continuous-time Markov jump dynamics
on $\Yb$
which act nicely on Macdonald measures.\footnote{These
dynamics may be viewed as discrete $(q,t)$-analogues
of the classical Dyson Brownian motion
\cite{dyson1962brownian}
from random matrix theory, e.g., see
\cite{BorodinPetrov2013Lect}.}
Let $\Ab$ be a $(q,t)$-nonnegative specialization.
The \emph{univariate} continuous-time Markov dynamics introduced in \cite[\S 2.3.1]{BorodinCorwin2011Macdonald}
(see also \cite[\S 4.3]{BorodinPetrov2013NN})
lives on the set of Young diagrams
$\Yb(\Ab)$ and (during time $\sigma\ge0$) maps the Macdonald measure
$\MM^{\Ab}_{\tau}$ into the measure
$\MM^{\Ab}_{\tau+\sigma}$ with evolved time parameter $\tau+\sigma$.
This univariate dynamics is defined through the jump rate matrix
having the form
\begin{align}\label{Q_A}
	\Qs_\Ab(\la,\nu):=\begin{cases}
		\dfrac{P_\nu(\Ab)}{P_\la(\Ab)}\psi'_{\nu/\la},
		&\mbox{if $\la\precb\nu$};
		\\
		-\displaystyle\sum_{\square\in\Us(\la)}\Qs_\Ab(\la,\la+\square),&\mbox{if $\nu=\la$};
		\\
		0,&\mbox{otherwise}.
	\end{cases}
\end{align}
Here
\begin{align}\label{psi_prime}
	\psi'_{\nu/\la}=\psi'_{\nu/\la}(q,t):=Q_{\nu/\la}(\hat 1\,|\,q,t)
\end{align}
is the value of the skew Macdonald symmetric function under
the specialization into one dual variable (equal to one),
it is given by
\cite[VI.(6.24.iv)]{Macdonald1995},
see also \S \ref{sub:_q_t_quantities} below.

We summarize properties of the univariate dynamics in the following proposition:
\begin{proposition}\label{prop:univariate_properties}
	{\rm{}\bf{}(1)} Jump rates $\Qs_\Ab$ define a Feller Markov jump process with semigroup
	$\{\Ps_\Ab(\tau)\}_{\tau\ge0}$, where
	$\Ps_\Ab(\tau)=\exp(\tau\Qs_\Ab)$.

	\noindent{\rm{}\bf{}(2)} The action of the univariate dynamics on Macdonald measures is given
	by
	\begin{align}\label{MM_P_update}
		\MM^{\Ab}_{\tau}\Ps_\Ab(\sigma)=
		\MM^{\Ab}_{\tau+\sigma},\qquad \sigma\ge0.
	\end{align}

	\noindent{\rm{}\bf{}(3)} The univariate dynamics are compatible with the stochastic links in the
	sense that (as
	$\Yb(\Ab\cup\Bb)\times\Yb(\Ab)$ matrices)
	\begin{align}\label{Q_La_commute}
		\La^{\Ab\cup\Bb}_{\Ab}\Qs_{\Ab}=
		\Qs_{\Ab\cup\Bb}\La^{\Ab\cup\Bb}_{\Ab}
		\qquad\mbox{and}\qquad
		\La^{\Ab\cup\Bb}_{\Ab}\Ps_{\Ab}(\tau)=
		\Ps_{\Ab\cup\Bb}(\tau)\La^{\Ab\cup\Bb}_{\Ab},
		\quad \tau\ge0.
	\end{align}
	In other words, ``the following diagram is commutative'':
	\begin{equation*}
		\xymatrixcolsep{4pc}\xymatrix{
		\Yb(\Ab\cup\Bb) \ar@{-->}[d]^{\La^{\Ab\cup\Bb}_{\Ab}} \ar@{-->}[r]^{\Ps_{\Ab\cup\Bb}(\tau)}
		&
		\Yb(\Ab\cup\Bb)\ar@{-->}[d]^{\La^{\Ab\cup\Bb}_{\Ab}}\\
		\Yb(\Ab) \ar@{-->}[r]^{\Ps_{\Ab}(\tau)} &\Yb(\Ab)}
	\end{equation*}
\end{proposition}
\begin{proof}
	See \cite[\S 2.3.1]{BorodinCorwin2011Macdonald} and \cite[\S 4.3]{BorodinPetrov2013NN}.
\end{proof}


\subsection{Infinitesimal skew Cauchy identity} 
\label{sub:infinitesimal_skew_cauchy_identity}

The skew Cauchy identity \cite[VI.7]{Macdonald1995}
states that for two Macdonald nonnegative specializations $\Ab$ and $\Bb$
with $\Pi(\Ab;\Bb)<\infty$ one has
\begin{align}\label{skew_Cauchy}
	\sum_{\kappa\in\Yb}P_{\kappa/\la}(\Ab)Q_{\kappa/\nu}(\Bb)
	=\Pi(\Ab;\Bb)
	\sum_{\mu\in\Yb}Q_{\la/\mu}(\Bb)P_{\nu/\mu}(\Ab)
\end{align}
for any $\la,\nu\in\Yb$.
When $\la=\nu=\varnothing$, this identity
turns into the usual Cauchy identity
\eqref{Cauchy}.
We will need the following infinitesimal version of \eqref{skew_Cauchy}:
\begin{proposition}[infinitesimal skew Cauchy identity]
	Let $\Bb=\hat\varepsilon$ be the specialization into one dual variable equal to $\varepsilon$.
	Taking the coefficient
	by $\varepsilon$ in both sides of
	\eqref{skew_Cauchy} yields the following identity
	for any $\la,\nu\in\Yb$:
	\begin{align}\label{inf_skew_Cauchy}
		\sum_{\square\in\Us(\nu)}
		P_{\nu+\square/\la}(\Ab)\psi'_{\nu+\square/\nu}
		=
		p_1(\Ab)P_{\nu/\la}(\Ab)+
		\sum_{\bar\square\in\Ds(\la)}
		P_{\nu/\la-\bar\square}(\Ab)\psi'_{\la/\la-\bar\square}.
	\end{align}
\end{proposition}
\begin{proof}
	We have (see  \eqref{al_be_ga_Newton} and \eqref{Pi_A_B})
	\begin{align*}
		\Pi(\Ab;\hat\varepsilon)=
		\exp\left(
		\sum_{n\ge0}p_n(\Ab)(-1)^{n-1}\varepsilon^{n}
		\right)=
		1+\varepsilon p_1(\Ab)+O(\varepsilon^{2}).
	\end{align*}
	Now \eqref{skew_Cauchy}
	takes the form
	\begin{align*}
		\sum_{\kappa\in\Yb}P_{\kappa/\la}(\Ab)\varepsilon^{|\ka|-|\nu|}\psi'_{\kappa/\nu}
		=\big(
		1+\varepsilon p_1(\Ab)+O(\varepsilon^{2})\big)
		\sum_{\mu\in\Yb}\psi'_{\la/\mu}\varepsilon^{|\la|-|\mu|}(\Bb)P_{\nu/\mu}(\Ab).
	\end{align*}
	The desired claim follows by considering the
	coefficient by $\varepsilon$ in the above identity.
\end{proof}
One can readily see that
\eqref{inf_skew_Cauchy}
is essentially equivalent to the
commutation relation
\eqref{Q_La_commute} between
$\La^{\Ab\cup\Bb}_{\Ab}$ and
the univariate jump rate matrices
\eqref{Q_A}.
See \cite[\S2.4]{BorodinPetrov2013NN} for more
detail.

Identity \eqref{inf_skew_Cauchy} readily implies that
the diagonal elements of the
jump rate matrix \eqref{Q_A} are
\begin{align}\label{diagonal_Q}
	\Qs_\Ab(\nu,\nu)=-p_1(\Ab).
\end{align}
Indeed, one needs to put $\la=\varnothing$ in \eqref{inf_skew_Cauchy}
(this kills the sum in the right-hand side),
and divide both sides by $P_{\nu/\la}(\Ab)=P_{\nu}(\Ab)$.

	
\subsection{Bivariate `dynamics'} 
\label{sub:bivariate_dynamics}

This subsection
extends results of
\cite[\S 2 and \S 5]{BorodinPetrov2013NN}. See
also
\cite{DiaconisFill1990},
\cite{BorFerr2008DF}, and
\cite[\S8]{BorodinOlshanski2010GTs}
for related constructions, and \cite{BorodinPetrov2013Lect} for a survey.

We call a continuous-time Markov
`dynamics'\footnote{This actually is the first place
when we drop the
nonnegativity assumption.}
on the space $\Yb^{(2)}(\Ab;\Bb)$ \eqref{2_support}
with matrix of `jump rates' $\Qs^{(2)}_{\Ab;\Bb}$
a \emph{bivariate `dynamics'} if the following three conditions are
satisfied:

\begin{enumerate}[(1)]
\item
The `dynamics' $\Qs^{(2)}_{\Ab;\Bb}$ preserves the class of Gibbs measures
on $\Yb^{(2)}(\Ab;\Bb)$.

\item Assume that $\Qs^{(2)}_{\Ab;\Bb}$
starts from a Gibbs measure on $\Yb^{(2)}(\Ab;\Bb)$.
Then on the \emph{upper level} $\Yb(\Ab\cup\Bb)$,
the `dynamics' $\Qs^{(2)}_{\Ab;\Bb}$ must reduce to the univariate dynamics $\Qs_{\Ab\cup\Bb}$.

\item The `dynamics' $\Qs^{(2)}_{\Ab;\Bb}$ evolves according to a
\emph{sequential update}, with interaction propagating from the lower to the upper 
level. Note that by Proposition \ref{prop:univariate_properties}.(3),
sequential update property plus the above condition (2) 
imply that 
on the lower level
$\Yb(\Ab)$, $\Qs^{(2)}_{\Ab;\Bb}$ must reduce to the corresponding univariate dynamics
$\Qs_{\Ab}$. Hence,
the `jump rates' of the bivariate `dynamics' must have the form
(here $\bi{\bar\la}{\la},\bi{\bar\nu}{\nu}\in\Yb^{(2)}(\Ab;\Bb)$):
\begin{align}\label{Q2}
	\Qs^{(2)}_{\Ab;\Bb}
	\big(\bi{\bar\la}{\la},\bi{\bar\nu}{\nu}
	\big)=
	\begin{cases}
		W(\la,\nu\,|\,\bar\nu),&\mbox{if $\bar\la=\bar\nu$};\\
		\Qs_\Ab(\bar\la,\bar\nu)V(\la,\nu\,|\,\bar\la,\bar\nu),
		&\mbox{if $\bar\la\ne\bar\nu$};\\
		\Qs_\Ab(\bar\nu,\bar\nu)+
		W(\nu,\nu\,|\,\bar\nu),
		&\mbox{if $\bar\la=\bar\nu$ and $\la=\nu$}.
	\end{cases}
\end{align}
Here $W(\la,\nu\,|\,\bar\nu)$ is the
`rate' of an independent jump
$\la\to\nu$ on the upper level (given
that there were no jumps
on the lower level, so
$\bar\nu=\bar\la$).
We assume that these `rates' satisfy
\begin{align}\label{W_conditions}
	\sum_{\nu\ne\la}
	W(\la,\nu\,|\,\bar\nu)=-W(\la,\la\,|\,\bar\nu)\qquad
	\mbox{for all
	$\bi{\bar\nu}\la\in\Yb^{(2)}(\Ab;\Bb)$}.
\end{align}
The quantity $V(\la,\nu\,|\,\bar\la,\bar\nu)$ is the
`conditional probability' that the jump $\bar\la\to\bar\nu$
on the lower level
triggers an instantaneous move $\la\to\nu$
on the
upper level. Note that we do not forbid the possibility that
$\la=\nu$, i.e., that the jump does not propagate upwards
(we will soon forbid such moves, see \S \ref{sub:rsk_type_dynamics} below).
The `probabilities' of triggered moves must satisfy
\begin{align}\label{V_condition}
	V(\la,\nu\,|\,\bar\la,\bar\la)=
	\mathbf{1}_{\la=\nu}
	,\qquad \qquad
	\sum_{\nu}V(\la,\nu\,|\,\bar\la,\bar\nu)=1
\end{align}
(where $\bi{\bar\la}{\la},\bi{\bar\la}{\nu}\in\Yb^{(2)}(\Ab;\Bb)$ and
$\bi{\bar\la}{\la},\bi{\bar\nu}{\nu}\in\Yb^{(2)}(\Ab;\Bb)$
in the first and in the second equality, respectively).

Note
that by \eqref{Q_A},
\eqref{W_conditions}, and \eqref{V_condition},
$$\displaystyle\sum\nolimits_{\bi{\bar\la}{\la}\ne\bi{\bar\nu}{\nu}}
\Qs^{(2)}_{\Ab;\Bb}
\big(
\bi{\bar\la}{\la},\bi{\bar\nu}{\nu}\big)
=-\Qs^{(2)}_{\Ab;\Bb}
\big(
\bi{\bar\la}{\la},\bi{\bar\la}{\la}\big),$$
as it should be for a matrix of `jump rates'.
\end{enumerate}

We see that bivariate `dynamics'
describe ways to \emph{stitch together} the
univariate dynamics $\Qs_\Ab$ and $\Qs_{\Ab\cup\Bb}$
into a Markov `dynamics' on the
space $\Yb^{(2)}(\Ab;\Bb)$.
Such a stitching is not unique,
and all possible ways to construct
a bivariate `dynamics' can be characterized as follows:
\begin{theorem}[\cite{BorodinPetrov2013NN}]
\label{thm:general_identity}
	The `jump rates' $W$ and the
	`probabilities' of triggered moves $V$
	satisfying \eqref{W_conditions} and \eqref{V_condition}
	correspond to a bivariate `dynamics' if and
	only if
	\begin{align}\label{general_identity}
		\begin{array}{ll}
		&\displaystyle
		\sum_{\bar\square\in\Ds(\bar\nu)}
		V(\la,\la+\square\,|\,\bar\nu-\bar\square,\bar\nu)
		P_{\la/\bar\nu-\bar\square}(\Bb)\psi'_{\bar\nu/\bar\nu-\bar\square}
		\\
		&\hspace{160pt}{}+W(\la,\la+\square\,|\,\bar\nu)P_{\la/\bar\nu}(\Bb)=
		P_{\la+\square/\bar\nu}(\Bb)\psi'_{\la+\square/\la},		
	\end{array}
	\end{align}
	for all $\bar\nu\in\Yb(\Ab)$, $\la\in\Yb(\Ab\cup\Bb)$,
	and all $\square\in\Us(\la)$, such that
	$\bi{\bar\nu}{\la+\square}\in\Yb^{(2)}(\Ab;\Bb)$.
\end{theorem}
\begin{proof}[Idea of proof]
	This is established in the same way as
	\cite[Prop. 5.3]{BorodinPetrov2013NN}
	with the help of the general discussion of \cite[\S2.4]{BorodinPetrov2013NN}.
	The idea of the proof is the following.
	Fix $\la\in\Yb(\Ab\cup\Bb)$ and sample $\bar\la\in\Yb(\Ab)$
	according to the Gibbs property. 
	The new configuration
	$\bi{\bar\nu}{\nu}$
	(arising after an infinitesimal amount of time)
	can be reached in two ways: either by running the 
	bivariate dynamics 
	$\Qs^{(2)}_{\Ab;\Bb}$ \eqref{Q2},
	or by running the univariate dynamics 
	$\Qs_{\Ab\cup\Bb}$ to get from $\la$ to $\nu$,
	and then sampling $\bar\nu$
	according to the Gibbs property. 
	Thus, we arrive at two 
	different expressions 
	for the infinitesimal probability
	of the resulting configuration $\bi{\bar\nu}{\nu}$,
	and \eqref{general_identity} is the equality between them.
	In this equality the Young diagrams $\bar\nu=\la+\square$, $\nu$, and $\la$
	are fixed, and the summation
	goes over $\bar\la=\bar\nu-\bar\square$.
\end{proof}
\begin{remark}\label{rmk:general_identity}
	Let us make several comments about the general identity \eqref{general_identity}:
	
	\begin{enumerate}[(1)]
		\item
	 Due to properties (1)--(2) of the bivariate `dynamics',
	during an infinitesimally small time interval under $\Qs^{(2)}_{\Ab;\Bb}$, at most one box
	can be added to each of the Young diagrams on the lower and on the upper level.
	In \eqref{general_identity}, these added boxes are
	denoted by $\bar\square$ and $\square$, respectively.
	
	\item Identity \eqref{general_identity} is written down for
	each fixed new state $\bar\nu$ on the lower level and
	all possible moves $\la\to\la+\square$ on the upper level.
	The summation in the left-hand side is over
	all ``histories'' $\bar\nu-\bar\square\to\bar\nu$
	on the lower level.
		
	\item
	Observe that \eqref{general_identity}
	is essentially independent of the lower
	specialization $\Ab$.
	That is, it
	depends on $\Ab$
	only through the requirement that $\bar\nu\in\Yb(\Ab)$,
	so $\Ab$ does not affect the form of
	the identity \eqref{general_identity}
	which is written down
	for each fixed $\bar\nu$ separately.

	\item If we sum \eqref{general_identity} over all $\square\in\Us(\la)$,
	then we get the infinitesimal
	skew Cauchy identity \eqref{inf_skew_Cauchy}.
	Thus, one may think
	that bivariate `dynamics' correspond to refinements of
	\eqref{inf_skew_Cauchy} (or, equivalently, of the commutation relations
	\eqref{Q_La_commute}).

	\item
	Identity \eqref{general_identity}
	extends the results of \cite{BorodinPetrov2013NN}
	in the sense that the latter paper
	deals only with the case when $\Bb$ is a specialization
	into a single usual variable.
	For such $\Bb$,
	the Young diagrams $\bar\nu$ and $\la+\square$
	in \eqref{general_identity}
	must differ by a horizontal strip
	(and hence the pair $\bi{\bar\nu}{\la+\square}$ can be represented
	as a configuration of
	interlacing particles on two levels,
	cf. \S \ref{sub:young_diagrams} and also \S \ref{sub:adding_a_usual_variable} below).
	For other $\Bb$'s,
	condition
	$\bi{\bar\nu}{\la+\square}\in\Yb^{(2)}(\Ab;\Bb)$
	will be different (and can be more complicated).
	\end{enumerate}
\end{remark}


\subsection{RSK-type `dynamics'} 
\label{sub:rsk_type_dynamics}

In the present paper we will deal
only with the following subclass of
bivariate `dynamics':
\begin{definition}\label{def:RSK_type}
	A bivariate `dynamics' $\Qs^{(2)}_{\Ab;\Bb}$ is called
	\emph{RSK-type} if for any
	$\bi{\bar\nu}{\la}\in\Yb^{(2)}(\Ab;\Bb)$
	and any $\bar\square\in\Ds(\bar\nu)$,
	one has
	\begin{align*}
		V(\la,\la\,|\,\bar\nu-\bar\square,\bar\nu)=0.
	\end{align*}
	This means that a jump on the lower level
	always propagates to the upper level. See
	\S \ref{sec:rsk_type_algorithm_for_sampling_hl_coherent_measures} below and also
	\cite{BorodinPetrov2013NN} for more discussion including
	connections to the classical RSK (Robinson--Schensted--Knuth) insertion algorithm.
\end{definition}

\begin{proposition}\label{prop:RSK-type}
	For an RSK-type bivariate `dynamics',
	the diagonal elements of the
	`jump rate' matrix are given by
	\begin{align*}
		\Qs^{(2)}_{\Ab;\Bb}
		\big(
		\bi{\bar\la}{\la},
		\bi{\bar\la}{\la}\big)=
		\Qs_{\Ab\cup\Bb}(\la,\la),
		\qquad
		\bi{\bar\la}{\la}\in\Yb^{(2)}(\Ab;\Bb).
	\end{align*}
\end{proposition}
Using \eqref{diagonal_Q}, we also have
$\Qs^{(2)}_{\Ab;\Bb}
\big(
\bi{\bar\la}{\la},
\bi{\bar\la}{\la}\big)=-p_1(\Ab\cup\Bb)$,
and $W(\la,\la\,|\,\bar\la)=-p_1(\Bb)$.
\begin{proof}
	Follows from \cite[(2.20)]{BorodinPetrov2013NN}
	with $x_k=y_k=\la$ and $y_{k-1}=\bar\la$.
	By the definition of RSK-type `dynamics',
	the sum over $x_{k-1}$ in
	that identity
	reduces to only one summand corresponding to $x_{k-1}=y_{k-1}$.
	This leads to
	$\Qs_{\Ab}(\bar\la,\bar\la)+W(\la,\la\,|\,\bar\la)=\Qs_{\Ab\cup\Bb}(\la,\la)$,
	which is equivalent to the desired claim
	(see \eqref{Q2}).
\end{proof}



\section{Three particular bivariate ‘dynamics’ on Macdonald processes} 
\label{sec:three_particular_dynamics_on_macdonald_processes}

Here we present three explicit examples of RSK-type bivariate
`dynamics'
$\Qs^{(2)}_{\Ab;(\al)}$,
$\Qs^{(2)}_{\Ab;(\be)}$, and
$\Qs^{(2)}_{\Ab;\Pl_\ga}$,
(where
$\Ab$ is an arbitrary $(q,t)$-nonnegative specialization,
and $\al,\be,\ga>0$).
In \S \ref{sec:rsk_type_algorithm_for_sampling_hl_coherent_measures}
below we will use them
as building blocks for our RSK-type sampling algorithm.

\subsection{$(q,t)$-quantities and their properties} 
\label{sub:_q_t_quantities}

Let us write down
explicit formulas for various quantities
related to Macdonald polynomials, and also list some relations
between them. 
Let us define
$\psi_{\la/\mu}(q,t)$,
$\varphi_{\la/\mu}(q,t)$,
$\psi'_{\la/\mu}(q,t)$, and
$\varphi'_{\la/\mu}(q,t)$
by the following one-variable specialization formulas
for any $\la,\mu\in\Yb$:
\begin{equation}\label{PQ_and_psi}
	\begin{array}{rcl}
		P_{\la/\mu}(\al\,|\,q,t)&=& \al^{|\la|-|\mu|}\psi_{\la/\mu}(q,t)
		\mathbf1_{\mu\prech\la};\\
		\rule{0pt}{14pt}
		Q_{\la/\mu}(\al\,|\,q,t)&=& \al^{|\la|-|\mu|}\varphi_{\la/\mu}(q,t)
		\mathbf1_{\mu\prech\la};\\
		\rule{0pt}{14pt}
		P_{\la/\mu}(\hat\be\,|\,q,t)&=&
		\be^{|\la|-|\mu|}\varphi'_{\la/\mu}(q,t)
		\mathbf1_{\mu\precv\la};\\
		\rule{0pt}{14pt}
		Q_{\la/\mu}(\hat\be\,|\,q,t)&=&
		\be^{|\la|-|\mu|}\psi'_{\la/\mu}(q,t)
		\mathbf1_{\mu\precv\la}
	\end{array}
\end{equation}
(recall that the notation $\hat\be$
emphasizes that we are specializing 
polynomials into one
dual variable).
All formulas in the rest of this subsection
describe various properties of the above quantities.
They follow from properties of Macdonald symmetric functions,
and we refer to 
\cite[VI]{Macdonald1995} for details.

Let us first list properties which do not involve swapping
the parameters $q$ and $t$. Define
\begin{align*}
	b_\la(q,t):= \frac{1}{\langle P_\la(\cdot\,|\,q,t),
	P_\la(\cdot\,|\,q,t)\rangle_{q,t}}=
	\frac{Q_\la(\cdot\,|\,q,t)}{P_\la(\cdot\,|\,q,t)},\qquad \la\in\Yb,
\end{align*}
see \S \ref{sub:preliminaries}.
(The dot means that we
take symmetric functions in arbitrary variables.)
An explicit formula for $b_\la(q,t)$
may be found in \cite[VI.(6.19)]{Macdonald1995},
but we do not need it.
We have
\begin{align}
	Q_{\la/\mu}(\cdot\,|\,q,t)&=\frac{b_\la(q,t)}{b_\mu(q,t)}P_{\la/\mu}(\cdot\,|\,q,t);\label{P_Q_bb}\\
	\varphi_{\la/\mu}(q,t)&=\frac{b_\la(q,t)}{b_\mu(q,t)}\psi_{\la/\mu}(q,t);\qquad
	\varphi'_{\la/\mu}(q,t)=\frac{b_\mu(q,t)}{b_\la(q,t)}\psi'_{\la/\mu}(q,t);\label{phi_psi_bb}.
\end{align}
In the special case when $\mu\precb\la$, one also has
\begin{align}\label{phi_psi_one_box}
	\varphi_{\la/\mu}(q,t)=\frac{1-t}{1-q}\psi'_{\la/\mu}(q,t),\qquad
	\varphi'_{\la/\mu}(q,t)=\frac{1-q}{1-t}\psi_{\la/\mu}(q,t).
\end{align}

Now let us turn to formulas involving swapping of $q$ with $t$:
\begin{align}
	b_\la(q,t)=\frac{1}{b_{\la'}(t,q)};\label{bb_and_psi_swap}
	\qquad \psi'_{\la/\mu}(q,t)&=\psi_{\la'/\mu'}(t,q);
	\qquad
	\varphi'_{\la/\mu}(q,t)=\varphi_{\la'/\mu'}(t,q);
	\\
	P_{\la/\mu}(\hat\be\,|\,q,t)&=
	Q_{\la'/\mu'}(\be\,|\,t,q).\label{PQ_swap}
\end{align}

\begin{remark}\label{rmk:endo}
	The ``symmetry''
	between the parameters $q$ and $t$
	in
	\eqref{bb_and_psi_swap}--\eqref{PQ_swap}
	follows from
	the existence
	of an endomorphism of
	$\Sym$ defined
	by its action
	on the power sums as
	\begin{align*}
		\omega_{q,t}\colon p_k\mapsto(-1)^{k-1}\frac{1-q^{k}}{1-t^{k}}p_k,\qquad k=1,2,\ldots.
	\end{align*}
	We have $\omega_{q,t}\omega_{t,q}=\mathrm{id}$, and
	\begin{align*}
		\omega_{q,t}P_{\la/\mu}(\x\,|\,q,t)
		=Q_{\la'/\mu'}(\x\,|\,t,q),\qquad
		\omega_{q,t}Q_{\la/\mu}(\x\,|\,q,t)=P_{\la'/\mu'}(\x\,|\,t,q),\qquad
		\la,\mu\in\Yb.
	\end{align*}

	One readily sees that applying the endomorphism $\omega_{t,q}\colon \Sym\to\Sym$ and then a $(q,t)$-nonnegative specialization
	$(\ab;\bb;\Pl_\ga\,|\,q,t)$
	mapping
	$\Sym$ to $\R$, one gets another specialization which is now $(t,q)$-nonnegative
	(note the swapping of the usual and dual variables):
	\begin{align*}
		(\ab;\bb;\Pl_\ga\,|\,q,t)
		\circ\omega_{t,q}=
		(\boldsymbol\be;{\ab};\Pl_{\frac{1-q}{1-t}\ga}\,|\,t,q).
	\end{align*}
\end{remark}

Finally, let us list several explicit formulas
for the above $(q,t)$-quantities
which we will use.
For $\mu\prech\la$,
\begin{align}\label{psi_horiz_strip}
	\psi_{\la/\mu}(q,t)&=\prod_{1\le i\le j\le \ell(\mu)}
	\frac{f(q^{\mu_i-\mu_j}t^{j-i})f(q^{\la_i-\la_{j+1}}t^{j-i})}
	{f(q^{\la_i-\mu_j}t^{j-i})f(q^{\mu_i-\la_{j+1}}t^{j-i})},
	\qquad \qquad
	f(u):=\frac{(tu;q)_{\infty}}{(qu;q)_{\infty}}.
\end{align}
Let $\la=\mu+\square$ for some box $\square\in\Us(\mu)$,
and $j$ be the row number of that box.
We will denote this situation as $\la=\mu+\de_j$.
In this case,
\begin{align}\label{psip_one_box}
	\psi'_{\mu+\de_j/\mu}(q,t)=
	\prod_{i=1}^{j-1}
	\frac{(1-q^{\mu_i-\mu_j}t^{j-i-1})
	(1-q^{\la_i-\la_j}t^{j-i+1})}
	{(1-q^{\mu_i-\mu_j}t^{j-i})
	(1-q^{\la_i-\la_j}t^{j-i})}.
\end{align}

We will also need certain combinations of the
quantities $\psi_{\la/\mu}$ and $\psi'_{\mu+\square/\mu}$
which were employed in \cite{BorodinPetrov2013NN} (in particular, see \cite[\S 5.4]{BorodinPetrov2013NN}).
Namely, let $\bar\nu,\la\in\Yb$
be such that $\bar\nu\prech\la$.
Let a box $\bar\square\in\Ds(\bar\nu)$ belong to row number $i$
(so we can use the notation $\bar\nu-\bar\square=\bar\nu-\bar\de_i$). Define
\begin{align}
	\nonumber
	&
	T_{i}(\bar\nu,\la\,|\,q,t):=\frac{\psi_{\la/\bar\nu-\bar\de_i}(q,t)}
	{\psi_{\la/\bar\nu}(q,t)}
	\psi'_{\bar\nu/\bar\nu-\bar\de_i}(q,t)=
	\frac{(1-q^{\la_i - \bar\nu_i}t)(1-q^{\bar\nu_i-\la_{i+1}})}
	{(1-q^{\la_i - \bar\nu_i+1})
	(1-q^{\bar\nu_i-1-\la_{i+1}}t)}
	\\&
	\label{T_i}\hspace{30pt}\times
	\prod_{r=1}^{i-1}
	\frac{
	(1-q^{\la_r - \bar \nu_i} t^{i - r+1})(1-q^{\bar\nu_r-\bar\nu_i+1}t^{i-r-1})}
	{
	(1-q^{\la_r - \bar \nu_i + 1} t^{i - r})
	(1-q^{\bar\nu_r-\bar\nu_i}t^{i-r})}
	\prod_{s=i+1}^{\ell(\bar\nu)}
	\frac{(1-q^{\bar\nu_i - \bar\nu_s-1} t^{s - i+1})
	(1-q^{\bar\nu_i - \la_{s + 1}} t^{s -i})}
	{(1-q^{\bar\nu_i - \bar\nu_s} t^{s - i})
	(1-q^{\bar\nu_i - \la_{s + 1}-1} t^{s -i+1})}.
\end{align}
Also, for $\bar\nu\prech\la$, let a box $\square\in\Us(\la)$ belong to row number $j$.
Define
\begin{align}
	&
	S_{j}(\bar\nu,\la\,|\,q,t):=\nonumber
	\frac{\psi_{\la+\de_j/\bar\nu}(q,t)}
	{\psi_{\la/\bar\nu}(q,t)}
	\psi'_{\la+\de_j/\la}(q,t)=
	\prod_{r=1}^{j-1}
	\frac{(1-q^{\bar\nu_r-\la_j}t^{j-r-1})
	(1-q^{\la_r-\la_j-1}t^{j-r+1})
	}
	{(1-q^{\bar\nu_r-\la_j-1}t^{j-r})
	(1-q^{\la_r-\la_j}t^{j-r})
	}
	\\&
	\label{S_j}
	\hspace{245pt}\times
	\prod_{s=j}^{\ell(\bar\nu)}
	\frac{(1-q^{\la_j-\la_{s+1}+1}t^{s-j})
	(1-q^{\la_j-\bar\nu_s}t^{s-j+1})
	}
	{(1-q^{\la_j-\la_{s+1}}t^{s-j+1})
	(1-q^{\la_j-\bar\nu_s+1}t^{s-j})}.
\end{align}
The product formulas in \eqref{T_i}--\eqref{S_j} readily follow from \eqref{psi_horiz_strip}
and
\eqref{psip_one_box}.


\subsection{Bivariate `dynamics' with a usual variable} 
\label{sub:adding_a_usual_variable}

In this subsection we present functions $(W_{(\al)},V_{(\al)})$
corresponding (as in \S \ref{sub:bivariate_dynamics}) to an RSK-type bivariate `dynamics'
$\Qs^{(2)}_{\Ab;(\al)}$, where $\Ab$ is an arbitrary $(q,t)$-nonnegative
specialization and $\al>0$.

\begin{lemma}\label{lemma:2_support_alpha}
	The state space
	\eqref{2_support}
	of a bivariate `dynamics'
	$\Qs^{(2)}_{\Ab;(\al)}$
	can be characterized as
	\begin{align}\label{2_support_alpha}
		\Yb^{(2)}\big(\Ab;(\al)\big)=
		\left\{
		\bi{\bar\la}{\la}\colon
		\text{$\bar\la\in\Yb(\Ab)$ and $\bar\la\prech\la$}
		\right\}.
	\end{align}
\end{lemma}
\begin{proof}
	Clearly, the conditions $P_{\bar\la}(\Ab\,|\,q,t)>0$
	and $P_{\la/\bar\la}(\al\,|\,q,t)>0$ are satisfied under conditions \eqref{2_support_alpha}.
	Using identity \cite[VI.(7.9')]{Macdonald1995} which can be written in the form
	\begin{align}\label{sequential_P_polys}
		\sum_{\bar\la\colon\bar\la\prech\la}
		P_{\la/\bar\la}(\al\,|\,q,t)P_{\bar\la}(\Ab\,|\,q,t)=
		P_{\la}\big(\Ab\cup(\al)\,|\,q,t\big),
	\end{align}
	we conclude that the condition
	$P_{\la}\big(\Ab\cup(\al)\,|\,q,t\big)>0$ also holds in \eqref{2_support_alpha}.
\end{proof}

Let us fix $h\in\{1,2,\ldots\}\cup\{+\infty\}$. Our
functions $(W_{(\al)},V_{(\al)})=(W_{(\al)}^{h},V_{(\al)}^{h})$ will depend on $h$
as a parameter.

We will
describe
the values
$W_{(\al)}^{h}(\la,\cdot\,|\,\bar\nu)$ and
$V_{(\al)}^{h}(\la,\cdot\,|\,\cdot,\bar\nu)$
for all
meaningful
pairs of Young diagrams $\bar\nu\in\Yb(\Ab)$
and
$\la\in\Yb\big(\Ab\cup\{\al\}\big)$
(entering \eqref{W_conditions}, \eqref{V_condition}, and \eqref{general_identity}).
Looking at \eqref{Q2}, we see that these two diagrams must satisfy
$\bi{\bar\nu}{\la+\square}\in\Yb^{(2)}\big(\Ab;\{\al\}\big)$ for at least one box
$\square\in\Us(\la)$.

If
$\bi{\bar\nu}{\la}\notin\Yb^{(2)}\big(\Ab;\{\al\}\big)$,
then
it means that there is a unique $\bar\square\in\Ds(\bar\nu)$ such that
$\bi{\bar\nu-\bar\square}{\la}\in\Yb^{(2)}\big(\Ab;\{\al\}\big)$.
Pairs $\bi{\bar\nu}{\la}\notin\Yb^{(2)}\big(\Ab;\{\al\}\big)$
do not enter the definition of $W_{(\al)}^{h}$, and
the only nonzero value of $V_{(\al)}^{h}$ in this case is
\begin{align}\label{short_range_alpha}
	V^{h}_{(\al)}(\la,\la+\square\,|\,\bar\nu-\bar\square,\bar\nu)=1,\qquad
	\bi{\bar\nu}{\la}\notin\Yb^{(2)}\big(\Ab;\{\al\}\big),
	\quad
	\bi{\bar\nu}{\la+\square},\bi{\bar\nu-\bar\square}{\la}\in\Yb^{(2)}\big(\Ab;\{\al\}\big).
\end{align}
In words, this means that if
a jump $\bar\nu-\bar\square\to\bar\nu$
on the lower level breaks the condition
$\bi{\bar\nu}{\la}\in\Yb^{(2)}\big(\Ab;\{\al\}\big)$, then there is a unique
jump $\la\to\la+\square$ which must happen (almost surely) to restore this condition.
In \cite[\S 5.3]{BorodinPetrov2013NN}
the property \eqref{short_range_alpha} was called the \emph{short-range pushing}.

In view of \eqref{short_range_alpha}, it remains to define
$W_{(\al)}^{h}(\la,\cdot\,|\,\bar\nu)$ and
$V_{(\al)}^{h}(\la,\cdot\,|\,\cdot,\bar\nu)$
for all possible pairs $\bi{\bar\nu}{\la}\in\Yb^{(2)}\big(\Ab;\{\al\}\big)$.
Denote $\ell:=\ell(\bar\nu)$, so $\ell(\la)\le\ell+1$.
We represent pairs $\bi{\bar\nu}\la$
as interlacing particle configurations
$\{\bar\nu_1,\ldots,\bar\nu_{\ell}\}\sqcup\{\la_1,\ldots,\la_{\ell+1}\}\subset\Z\sqcup\Z$
with $\ell$ and $\ell+1$
particles on the lower and on the upper levels, respectively
(appending $\la$ by zeroes if necessary).
See Fig.~\ref{fig:interlace}.
\begin{figure}[htbp]
\begin{center}
\begin{tikzpicture}[
	    scale=2.2,
	    axis/.style={ultra thick, ->, >=stealth'}]
	    \def\y{.5}
	    \draw[axis] (0,0) -- (5,0) node(xline)[right]{$\bar\nu$};
	    \draw[axis] (0,\y) -- (5,\y) node(xline)[right]{$\la$};
	    \foreach \hh in {.5, 1, 1.5, 2, 2.5, 3, 3.5, 4, 4.5}
	    {
	    	\draw[densely dotted, thick, opacity=.5] (\hh,-.12) -- (\hh,\y+.12);
	    }
	    \def\sp{.09};
	    \def\opac{.55}
	    \draw[line width=.7, color=blue, opacity=\opac]
		(5,\y/2)--(4.5,\y)--
		(3.5+\sp/3,0)--++(-2*\sp/3,\y)--
		(3+\sp/3,0)--++(-2*\sp/3,\y)
		--(2,0)--(1+\sp,\y)--++(-\sp,-\y)--++(-\sp,\y)
		--(0,\y/2);
	    \foreach \pt in
	    {(2,0), (1+\sp,\y), (1-\sp,\y), (1,0),
	    (3+\sp/3,0),(3-\sp/3,\y),
	    (3.5+\sp/3,0),(3.5-\sp/3,\y),
	    (4.5,\y)}
	    {
	    	\draw[fill] \pt circle (1.55pt);
	    }
	\end{tikzpicture}
\end{center}
\caption{An interlacing particle configuration
(zigzag illustrates the interlacing).}
\label{fig:interlace}
\end{figure}
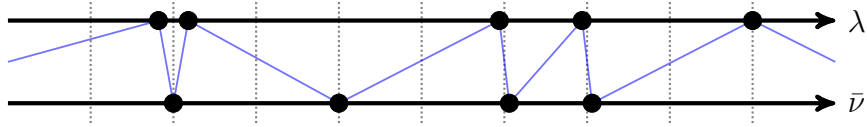

Clearly, adding a box to a Young diagram (say, $\la$)
corresponds to one of the particles $\la_j$ jumping to the right by one.
Denote by $\F(\bar\nu,\la)\subset\{1,\ldots,\ell+1\}$ the set of all indices
of upper particles $\la_j$
which are \emph{free to jump}
(i.e., which can jump to the right without breaking the interlacing with lower particles).
For example, on Fig.~\ref{fig:interlace} we have $\F(\bar\nu,\la)=\{1,2,4\}$.
Denote for $m\in\{1,2,\ldots\}\cup\{+\infty\}$:
\begin{align}\label{next}
	\nt(m):=\max\{j\in\F(\bar\nu,\la)\colon j\le m\}.
\end{align}
In words, this is the index of the first free particle to the right of $\la_m$ (including $\la_m$).

Now we can complete the definition of
the functions
$W_{(\al)}^{h}$ and
$V_{(\al)}^{h}$.
The only nonzero value of $W_{(\al)}^{h}(\la,\cdot\,|\,\bar\nu)$
is
\begin{align}\label{W_alpha}
	W_{(\al)}^{h}(\la,\la+\de_{\nt(h)}\,|\,\bar\nu)=\al.
\end{align}
The values of $V_{(\al)}^{h}$ are given
for any $m,j+1\in\F(\bar\nu,\la)$ by
\begin{align}\label{V_alpha}
	V_{(\al)}^{h}
	(\la,\la+\de_{m}\,|\,\bar\nu-\bar\de_j,\bar\nu)
	=\rp_j^{h}(\bar\nu,\la)\mathbf{1}_{m=\nt(j)}+
	\big(1-\rp_j^{h}(\bar\nu,\la)\big)\mathbf{1}_{m=j+1},
\end{align}
with
\begin{align}\label{r_j_h}
	\rp_j^{h}(\bar\nu,\la):=
	\frac{1}{T_j(\bar\nu,\la\,|\,q,t)}
	\left(
	\sum_{i=1}^{j}S_i(\bar\nu,\la\,|\,q,t)
	-\sum_{i=1}^{j-1}T_i(\bar\nu,\la\,|\,q,t)
	-\mathbf{1}_{j\ge h}
	\right),
\end{align}
where the quantities $S_{i}$, $T_{i}$ are defined by \eqref{T_i}--\eqref{S_j}
(with the understanding that if $\la_i=\bar\nu_{i-1}$, i.e.,
the particle $\la_i$ is blocked and cannot move to the right,
then $S_i=T_{i-1}=0$).

\begin{proposition}\label{prop:alpha_bi}
	For each $h\in\{1,2,\ldots\}\cup\{+\infty\}$,
	the `rates of independent jumps'
	$W_{(\al)}^{h}$ \eqref{W_alpha} and the
	`probabilities of triggered moves'
	$V_{(\al)}^{h}$ \eqref{short_range_alpha}, \eqref{V_alpha}--\eqref{r_j_h}
	define an RSK-type bivariate `dynamics'
	$\Qs^{(2)}_{\Ab;(\al)}$
	via the construction explained in \S \ref{sub:bivariate_dynamics}.
\end{proposition}
\begin{proof}
	This is proven in \cite[\S 6.4.2 and \S 6.5.3]{BorodinPetrov2013NN}.
\end{proof}
The `dynamics' corresponding to
$(W_{(\al)}^{h},V_{(\al)}^{h})$
defined by \eqref{W_alpha}--\eqref{r_j_h} above
can be intuitively interpreted in the following way.\footnote{We are using
probabilistic terms despite the fact that some `probabilities' can be negative (see the beginning of \S \ref{sec:macdonald_processes_and_bivariate_continuous_time_dynamics_} for more detail). When we speak about conditioning on an
event which possibly can have negative probability, this should be understood 
as an intuitive appeal 
to the product rule 
\eqref{Q2}
defining the jump rates via 
the quantities \eqref{W_alpha}--\eqref{r_j_h}.}
Let $\bi{\bar\la}\la\in\Yb^{(2)}\big(\Ab;(\al)\big)$
denote the current state of the `dynamics'.
We argue in terms of interlacing arrays, cf. Fig.~\ref{fig:interlace}.

The only particle that can jump on the upper level
is $\la_{\nt(h)}$, and it jumps to the right by one
according to an exponential clock of rate
$\al$.\footnote{Setting $h=+\infty$ means that the last (i.e., the leftmost) particle
jumps independently.}
One can equivalently say that the particle $\la_h$ itself tries to jump (with rate $\al$), but if it is blocked, then
it \emph{donates} its jump to the first free particle to the right of itself.

On the lower level, if a particle $\bar\la_{j}$
moves to the right by one,\footnote{We assume that the lower level particles evolve according
to the univariate dynamics $\Qs_{\Ab}$. The functions
$(W_{(\al)}^{h},V_{(\al)}^{h})$ then
provide the necessary ``induction step''
leading to the upper univariate dynamics $\Qs^{(2)}_{\Ab;(\al)}$.} then it instantaneously pushes
(to the right by one)
its first free upper right neighbor
$\la_{\nt(j)}$ with probability $\rp_j^{h}(\bar\la+\bar\de_j,\la)$ \eqref{r_j_h},
or pulls (also to the right by one)
its upper left neighbor $\la_{j+1}$
with the complementary probability
$1-\rp_j^{h}(\bar\la+\bar\de_j,\la)$.
See Fig.~\ref{fig:rl}.

\begin{figure}[htbp]
\begin{center}
	\begin{tikzpicture}[
	    scale=2.2,
	    axis/.style={ultra thick, ->, >=stealth'}]
	    \def\y{.5}
	    \draw[axis] (0,0) -- (5,0) node(xline)[right]{$\bar\nu$};
	    \draw[axis] (0,\y) -- (5,\y) node(xline)[right]{$\la$};
	    \foreach \hh in {.5, 1, 1.5, 2, 2.5, 3, 3.5, 4, 4.5}
	    {
	    	\draw[densely dotted, thick, opacity=.5] (\hh,-.12) -- (\hh,\y+.12);
	    }
	    \def\sp{.09};
	    \def\opac{.25}
	    \draw[line width=.7, color=blue, opacity=\opac]
		(5,\y/2)--(4.5,\y)--
		(3.5+\sp/3,0)--++(-2*\sp/3,\y)--
		(3+\sp/3,0)--++(-2*\sp/3,\y)
		--(2,0)--(1+\sp,\y)--++(-\sp,-\y)--++(-\sp,\y)
		--(0,\y/2);
	    \foreach \pt in
	    {(2,0), (1+\sp,\y), (1-\sp,\y), (1,0),
	    (3+\sp/3,0),(3-\sp/3,\y),
	    (3.5+\sp/3,0),(3.5-\sp/3,\y),
	    (4.5,\y)}
	    {
	    	\draw[fill] \pt circle (1.55pt);
	    }
	    \draw[->,densely dotted, ultra thick, color = blue]
	    (1.53,-0.03) to [in=180, out=-60] (1.75,-.25) to [in=-120, out=0] (1.96,-0.05);
	    \draw[->,densely dotted, line width=2] (2,0) to [in=225, out=60] (4.5-.04,\y-.045);
	    \draw[->,densely dotted, line width=2] (2,0) to [in=-45, out=110] (1+\sp+.04,\y-.045);
	    \node at (3.87,\y/2+.105) {\small{}$\rp_j^{h}(\bar\nu,\la)$};
	    \node at (0.99,\y/2-.06) {\small{}$1-\rp_j^{h}(\bar\nu,\la)$};
	    \node at (2.11,-.2) {$\bar\nu_j$};
	    \node at (1.21,\y+.21) {$\la_{j+1}$};
	    \node at (4.7,\y+.21) {$\la_{\nt(j)}$};
	    \node at (1.4, -.4) {\color{blue}just moved};
	\end{tikzpicture}
\end{center}
\caption{Pushing and pulling `probabilities' in \eqref{V_alpha}, $\bar\nu_j=\bar\la_j+1$.}
\label{fig:rl}
\end{figure}
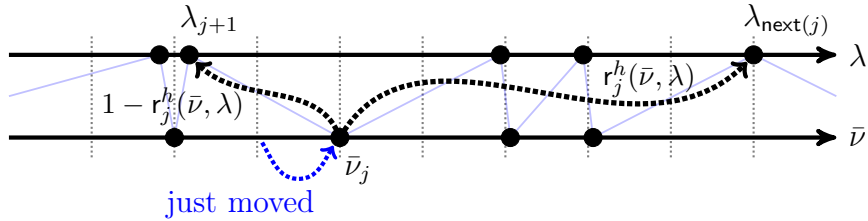

\begin{remark}\label{rmk:NN_dynamics}
	The `dynamics' $\Qs^{(2)}_{\Ab;(\al)}$
	on interlacing configurations
	belongs to the class of so-called \emph{nearest neighbor}
	`dynamics', in which a moving lower level particle can
	(with some probabilities)
	either
	push its first free upper right neighbor, or pull its
	immediate upper left neighbor.
	The main result of \cite{BorodinPetrov2013NN}
	is a
	complete classification of
	nearest neighbor `dynamics'
	of the form $\Qs^{(2)}_{\Ab;(\al)}$, where $\Ab$
	is a finite length specialization (see \eqref{finite_length_spec}).
	It is
	possible to extend that classification
	to arbitrary $(q,t)$-nonnegative specializations
	$\Ab$, and also to the cases
	when the usual specialization
	$(\al)$
	is replaced by a dual
	or a Plancherel one
	(see \S \ref{sub:duality} and \S \ref{sub:adding_a_plancherel_parameter} below).
	We do not pursue this direction here.
\end{remark}


\subsection{Transposition, and bivariate `dynamics' with a dual variable} 
\label{sub:duality}

Using a `dynamics'
with a usual variable $\al$ (\S \ref{sub:adding_a_usual_variable})
together with identities
from \S \ref{sub:_q_t_quantities},
it is possible to construct a `dynamics' $\Qs^{(2)}_{\Ab;(\be)}$
with a dual variable $\be>0$ to the specialization.

\begin{lemma}\label{lemma:2_support_beta}
	The state space
	\eqref{2_support}
	of a bivariate `dynamics'
	$\Qs^{(2)}_{\Ab;(\be)}$
	can be characterized as
	\begin{align}\label{2_support_beta}
		\Yb^{(2)}\big(\Ab;(\be)\big)=
		\left\{
		\bi{\bar\la}{\la}\colon
		\text{$\bar\la\in\Yb(\Ab)$ and $\bar\la\precv\la$}
		\right\}.
	\end{align}
\end{lemma}
\begin{proof}
	Similar to the proof of Lemma \ref{lemma:2_support_alpha}.
\end{proof}

To describe functions
$(W_{(\be)},
V_{(\be)})$
satisfying
\eqref{W_conditions}, \eqref{V_condition}, and \eqref{general_identity},
let us rewrite \eqref{general_identity} in terms of transposed Young diagrams:
\begin{lemma}\label{lemma:general_identity_beta}
	For $\Bb=(\be)$, identity \eqref{general_identity} can be rewritten in the following
	equivalent form:
	\begin{align}\label{general_identity_beta}
		\begin{array}{ll}
		&\displaystyle
		\sum_{\bar\square\in\Ds(\bar\nu)}
		V(\la,\la+\square\,|\,\bar\nu-\bar\square,\bar\nu)
		P_{\la'/(\bar\nu-\bar\square)'}(\be\,|\,t,q)
		\psi'_{\bar\nu'/(\bar\nu-\bar\square)'}(t,q)
		\\&\hspace{40pt}{}+
		\dfrac{1-t}{1-q}
		W(\la,\la+\square\,|\,\bar\nu)P_{\la'/\bar\nu'}(\be\,|\,t,q)=
		P_{(\la+\square)'/\bar\nu'}(\be\,|\,t,q)\psi'_{(\la+\square)'/\la'}(t,q),
	\end{array}
	\end{align}
	for all $\bar\nu\in\Yb(\Ab)$, $\la\in\Yb\big(\Ab\cup(\be)\big)$,
	and all $\square\in\Us(\la)$, such that
	$\bi{\bar\nu}{\la+\square}\in\Yb^{(2)}\big(\Ab;(\be)\big)$.
\end{lemma}
Note that \eqref{general_identity_beta} involves specializations
of Macdonald polynomials
into one \emph{usual} variable equal to $\be$.
\begin{proof}
	Using identities from \S \ref{sub:_q_t_quantities}, we can rewrite
	expressions entering \eqref{general_identity} as follows:
	\begin{align*}
		P_{\la/\bar\nu-\bar\square}(\hat\be\,|\,q,t)\psi'_{\bar\nu/\bar\nu-\bar\square}(q,t)
		&=
		Q_{\la'/(\bar\nu-\bar\square)'}(\be\,|\,t,q)
		\psi_{\bar\nu'/(\bar\nu-\bar\square)'}(t,q)
		\\&=
		P_{\la'/(\bar\nu-\bar\square)'}(\be\,|\,t,q)
		\frac{b_{\la'}(t,q)}{b_{(\bar\nu-\bar\square)'}(t,q)}
		\varphi_{\bar\nu'/(\bar\nu-\bar\square)'}(t,q)
		\frac{b_{(\bar\nu-\bar\square)'}(t,q)}{b_{\bar\nu'}(t,q)}
		\\&=
		\frac{b_{\la'}(t,q)}{b_{\bar\nu'}(t,q)}
		P_{\la'/(\bar\nu-\bar\square)'}(\be\,|\,t,q)
		\frac{1-q}{1-t}
		\psi'_{\bar\nu'/(\bar\nu-\bar\square)'}(t,q).
	\end{align*}
	In the last equality we have used \eqref{phi_psi_one_box} with
	parameters $q$ and $t$ interchanged, hence the coefficient $\frac{1-q}{1-t}$.
	Rewriting the other two coefficients in \eqref{general_identity}
	in a similar way, one gets \eqref{general_identity_beta}.
\end{proof}

\begin{proposition}\label{prop:usual_dual}
	Functions $(W_{(\be)},
	V_{(\be)})$
	correspond (as in \S \ref{sub:bivariate_dynamics})
	to an RSK-type bivariate `dynamics'
	$\Qs^{(2)}_{\Ab;(\be)}$
	with a dual variable $\be$ if and only if they have the form
	\begin{align*}
		W_{(\be)}(\la,\la+\square\,|\,\bar\nu)&=
		\left.W_{(\al)}(\la',(\la+\square)'\,|\,\bar\nu')
		\right\vert_{q\leftrightarrow t},
		&
		\bi{\bar\nu}\la,
		\bi{\bar\nu}{\la+\square}\in\Yb^{(2)}\big(\Ab;(\be)\big);
		\\
		V_{(\be)}(\la,\la+\square\,|\,\bar\nu-\bar\square,\bar\nu)&=
		\left.V_{(\al)}(\la',(\la+\square)'\,|\,(\bar\nu-\bar\square)',\bar\nu')
		\right\vert_{q\leftrightarrow t},
		&
		\bi{\bar\nu-\bar\square}\la,
		\bi{\bar\nu}{\la+\square}\in\Yb^{(2)}\big(\Ab;(\be)\big);
	\end{align*}
	where
	$\bar\square\in\Ds(\bar\nu)$,
	$\square\in\Us(\la)$,
	and the
	functions $(W_{(\al)},
	V_{(\al)})$
	correspond to
	a bivariate `dynamics'
	with the usual variable
	$\al:=\frac{1-q}{1-t}\be$,
	but
	in the setting with the
	swapped Macdonald parameters
	$(q,t)\to(t,q)$.
\end{proposition}
\begin{proof}
	Readily follows from Lemma \ref{lemma:general_identity_beta}
	and Theorem \ref{thm:general_identity}.
\end{proof}
In words,
Proposition \ref{prop:usual_dual} means that
to construct a bivariate `dynamics'
with a dual variable $\be$ and parameters $(q,t)$,
one can take
a bivariate `dynamics' with a \emph{usual} variable
$\frac{1-q}{1-t}\be$ and \emph{swapped} parameters $(t,q)$, and
run the latter `dynamics' in terms of \emph{columns} of
Young diagrams instead of rows.

Denote by
$(W_{(\be)}^{h},
V_{(\be)}^{h})$, where
$h\in\{1,2,\ldots\}\cup\{+\infty\}$, the functions
corresponding
(via Proposition \ref{prop:usual_dual})
to the functions
$(W_{(\al)}^{h}, V_{(\al)}^{h})$
from \S \ref{sub:adding_a_usual_variable}
(here $\al=\frac{1-q}{1-t}\be$).


\subsection{Plancherel specialization and Young graph} 
\label{sub:plancherel_specializations_and_young_graph}

We will need two lemmas describing Plancherel specializations
of skew Macdonald symmetric functions:
\begin{lemma}\label{lemma:skew_Plancherel}
	For any $\la,\kappa\in\Yb$ with $\la\subseteq\kappa$,
	the expression
	$(|\kappa|-|\la|)!\cdot P_{\kappa/\la}(\Pl_1)$ is equal to the
	coefficient of $Q_{\kappa}$ in the expansion of
	$p_1^{|\kappa|-|\la|}Q_\la$ with respect to the linear basis
	$\{Q_{\mu}\}_{\mu\in\Yb}$ of the algebra $\Sym$.
\end{lemma}
\begin{proof}
	Fix $\la\in\Yb$,
	and put $\nu=\varnothing$ and $\Ab=\Pl_1$
	in the skew Cauchy identity \eqref{skew_Cauchy}:
	\begin{align*}
		\sum_{\kappa\in\Yb}P_{\kappa/\la}(\Pl_1)Q_{\kappa}(\Bb)
		=e^{p_1(\Bb)}
		Q_{\la}(\Bb)
	\end{align*}
	(we also used \eqref{Cauchy}).
	Considering terms in both sides which are homogeneous with respect to $\Bb$, we get
	the following identities:
	\begin{align}\label{skew_Plancherel_identity}
		{p_1^{n}}\,Q_{\la}=\sum_{\kappa\in\Yb_{|\la|+n}}
		n!\,P_{\kappa/\la}(\Pl_1)\,Q_\kappa,
		\qquad n=1,2,\ldots.
	\end{align}
	(since $\Bb$ is any specialization, we could write identities in the algebra $\Sym$).
\end{proof}

\begin{lemma}\label{lemma:skew_Plancherel_2}
	For $\la,\kappa\in\Yb$ with $\la\subseteq\kappa$ and $n:=|\kappa|-|\la|$, one has
	\begin{align}\label{skew_Plancherel_2}
		n!\,P_{\kappa/\la}(\Pl_1)=
		\sum_{\la=\mu^{(0)}\precb \mu^{(1)}
		\precb\ldots\precb
		\mu^{(n-1)}\precb\mu^{(n)}=\kappa}
		\varphi'_{\mu^{(1)}/\mu^{(0)}}
		\varphi'_{\mu^{(2)}/\mu^{(1)}}
		\ldots
		\varphi'_{\mu^{(n)}/\mu^{(n-1)}},
	\end{align}
	where the quantities $\varphi'_{\mu^{(i+1)}/\mu^{(i)}}
	=\varphi'_{\mu^{(i+1)}/\mu^{(i)}}(q,t)$ are defined in \eqref{PQ_and_psi}.
\end{lemma}
\begin{proof}
	For $n=1$, we have
	by Lemma \ref{lemma:skew_Plancherel}:
	\begin{align*}
		p_1Q_{\la}=\sum_{\square\in\Us(\la)}P_{\la+\square/\la}(\Pl_1)Q_{\la+\square}.
	\end{align*}
	Comparing this with the Pieri formula \cite[VI.(6.24.iii)]{Macdonald1995}, we see that
	$P_{\la+\square/\la}(\Pl_1)=\varphi'_{\la+\square/\la}$.
	The general $n$ statement is readily established by induction.
\end{proof}
\begin{remark}\label{rmk:first_power}
	In particular,
	$P_{\la+\square/\la}(\Pl_1)$
	is equal to the coefficient by the first power of $\be$
	in $P_{\la+\square/\la}(\hat\be\mid q,t)$, and also
	(by \eqref{phi_psi_one_box})
	to $\frac{1-q}{1-t}$ times the coefficient
	by the first power of $\al$
	in $P_{\la+\square/\la}(\al\mid q,t)$, see \eqref{PQ_and_psi}.
\end{remark}

Lemmas \ref{lemma:skew_Plancherel} and \ref{lemma:skew_Plancherel_2}
reflect the structure of the Young graph
($=$~the lattice of all Young diagrams ordered by inclusion)
with formal Macdonald $(q,t)$ edge multiplicities:
The multiplicity of the edge $\la\nearrow\la+\square$
is given by $P_{\la+\square/\la}(\Pl_1\,|\,q,t)=\varphi'_{\la+\square/\la}(q,t)$.
Moreover, for any $\la,\mu\in\Yb$, the quantity
$(|\la|-|\mu|)!\cdot P_{\la/\mu}(\Pl_1\,|\,q,t)$
is equal to the total number of paths from $\mu$ to $\la$
(counted with these edge multiplicities).
See \cite{Kerov-book} and, e.g.,
\cite[\S9.1]{Petrov2011sl2} for more detail.

The problem of classifying $(q,t)$-nonnegative
specializations of the algebra $\Sym$ (\S \ref{sub:specializations_of_sym_})
is equivalent to classifying certain \emph{coherent measures}
on the Young graph with these edge multiplicities.
The coherency property of a family of probability measures
$\M_n$ on $\Yb_n$ is formulated as
\begin{align*}
	\sum_{\la\colon\la\searrow\mu}\M_n(\la)
	\,\frac{P_\mu(\Pl_1\,|\,q,t)}{n\cdot
	P_\la(\Pl_1\,|\,q,t)}\varphi'_{\la/\mu}(q,t)= \M_{n-1}(\mu)
	\qquad
	\mbox{for all $n\ge1$ and all $\mu\in\Yb_{n-1}$}.
\end{align*}
One can readily check that the coherent
measures defined in \S \ref{sub:_q_t_coherent_measures}
satisfy this relation by reducing it to
the Pieri formula \cite[VI.(6.24.iii)]{Macdonald1995}
(hence the name for the measures \eqref{coherent_def}).
See, e.g., \cite{Kerov1998}, \cite{Borodin2000}, \cite{Kerov-book} for more detail
about coherent measures and boundaries of branching graphs.

Connection between the above coherency relation and
stochastic links \eqref{stoch_links}--\eqref{MM_stoch_links}
is explained in \cite{BorodinOlsh2011Bouquet}
in the Schur ($q=t$) case,
when it highlights the
interplay between
representation theory of
the infinite symmetric group $S(\infty)$
and the infinite-dimensional unitary group
$U(\infty)$.

In the Hall--Littlewood ($q=0$) case, the above coherency
property may be interpreted as coming from
central measures on infinite uni-uppertriangular
matrices over a finite field, see
\S \ref{sub:infinite_random_matrices_over_a_finite_field}
and Remark \ref{rmk:coupling_intro} in particular.


\subsection{Bivariate `dynamics' with a Plancherel parameter} 
\label{sub:adding_a_plancherel_parameter}

We will now
construct bivariate `dynamics'
$\Qs^{(2)}_{\Ab;\Pl_\ga}$, where $\Ab$ is an arbitrary
$(q,t)$-nonnega\-tive specialization and $\ga>0$.
First, observe that
condition
$\bi{\bar\la}\la\in\Yb^{(2)}(\Ab;\Pl_\ga)$
means precisely that
$\bar\la\in\Yb(\Ab)$ and $\bar\la\subseteq\la$
(this follows from Lemma \ref{lemma:skew_Plancherel_2}, see also Lemmas \ref{lemma:2_support_alpha}
and \ref{lemma:2_support_beta}). Our construction will involve two steps,
``infinitesimal'' and ``general''.

\subsubsection{Infinitesimal step} 
\label{ssub:infinitesimal_step}

Let
$(W_{\circ},V_{\circ})$ be functions satisfying the following
identities:
\begin{align}&
	\sum_{\nu\colon \bar\nu\precb\nu}
	W_\circ(\bar\nu,\nu)=1\label{W_circ};\\&
	\sum_{\nu\colon \la\precb\nu,\, \bar\nu\precb\nu}
	V_\circ(\la,\nu\,|\,\bar\la,\bar\nu)=1,&
	\label{V_circ}\bar\la\precb\la,\quad\bar\la\precb\bar\nu;\\&
	\sum_{\bar\la\colon\bar\la\precb\bar\nu}
	V_\circ(\la,\nu\,|\,\bar\la,\bar\nu)
	\psi_{\la/\bar\la}\psi'_{\bar\nu/\bar\la}
	+W_\circ(\la,\nu)\mathbf{1}_{\la=\bar\nu}=\psi_{\nu/\bar\nu}\psi'_{\nu/\la},
	&\bar\nu\precb\nu,\quad\la\precb\nu.
	\label{general_identity_circ}
\end{align}
Note that \eqref{general_identity_circ} can be viewed as an infinitesimal version of the
general identity \eqref{general_identity}
corresponding to taking $\Bb=(\al)$ or $(\be)$
and considering the coefficient by the first power
of $\al$ or $\be$, respectively (cf. Remark \ref{rmk:first_power}).

One can choose $(W_{\circ},V_{\circ})$
using the functions from \S \ref{sub:adding_a_usual_variable}
and \S \ref{sub:duality}:
\begin{align}\label{Plancherel_alpha_choice}
	V_{\circ}^{\al;h}(\la,\nu\,|\,\bar\la,\bar\nu):=V_{(\al)}^{h}(\la,\nu\,|\,\bar\la,\bar\nu),
	\qquad
	W_{\circ}^{\al;h}(\la,\nu):=\frac{1}{p_1(\al)}W_{(\al)}^{h}(\la,\nu\,|\,\la),
\end{align}
or
\begin{align}\label{Plancherel_beta_choice}
	V_{\circ}^{\be;h}(\la,\nu\,|\,\bar\la,\bar\nu):=V_{(\be)}^{h}(\la,\nu\,|\,\bar\la,\bar\nu),
	\qquad
	W_{\circ}^{\be;h}(\la,\nu):=\frac{1}{p_1(\hat\be)}W_{(\be)}^{h}(\la,\nu\,|\,\la).
\end{align}
(Each of the two families
$(W_{\circ}^{h},V_{\circ}^{h})$ depends on
$h\in\{1,2,\ldots\}\cup\{+\infty\}$.)
Then \eqref{general_identity}
readily implies identity \eqref{general_identity_circ}.
Observe that in both cases
the values of the functions
$(W_{\circ},V_{\circ})$ do not depend on the
parameters $\al,\be$.

\begin{remark}\label{rmk:row_col}
	One should think that the
	choice
	\eqref{Plancherel_alpha_choice}
	corresponds
	to the ``row insertion'',
	while
	\eqref{Plancherel_beta_choice}
	leads to the ``column insertion'',
	cf. the idea of transposing Young diagrams employed in
	\S \ref{sub:duality} above (see also connections
	to the classical RSK insertion algorithms discussed in
	\cite[\S7]{BorodinPetrov2013NN}).
	Different choices
	\eqref{Plancherel_alpha_choice}
	and
	\eqref{Plancherel_beta_choice}
	lead to different `dynamics' with a Plancherel parameter.
\end{remark}


\subsubsection{General step: construction of the `dynamics'} 
\label{ssub:construction_of_the_dynamics}

Assume now that we have chosen functions $(W_{\circ},V_{\circ})$ as
in \S \ref{ssub:infinitesimal_step}.
We will now explain
the construction of
functions
$W_{\Pl_\ga}(\la,\cdot\,|\,\bar\la)$ and
$V_{\Pl_\ga}(\la,\cdot\,|\,\bar\la,\cdot)$
corresponding to the desired bivariate `dynamics'
$\Qs^{(2)}_{\Ab;\Pl_\ga}$.
Assume that the pair
$\bi{\bar\la}\la\in\Yb^{(2)}(\Ab;\Pl_\ga)$
is fixed, and denote
$n:=|\la|-|\bar\la|$.
Our construction of the desired
functions is probabilistic and consists of the following steps:
\begin{enumerate}[(1)]
	\item
	Sample random intermediate Young diagrams
	\begin{align*}
		\bar\la=\mu^{(0)}\nearrow\mu^{(1)}\nearrow\mu^{(2)}
		\nearrow \ldots\nearrow\mu^{(n)}=\la
	\end{align*}
	according to the distribution
	(cf. Lemma \ref{lemma:skew_Plancherel_2})
	\begin{align}\label{Plancherel_intermediate}
		\Prob\big(\mu^{(i)}\colon i=0,1,\ldots,n\big)=
		\frac{\varphi'_{\mu^{(1)}/\mu^{(0)}}
		\varphi'_{\mu^{(2)}/\mu^{(1)}}
		\ldots
		\varphi'_{\mu^{(n)}/\mu^{(n-1)}}}
		{n!\,P_{\la/\bar\la}(\Pl_1)}.
	\end{align}
	Let us add auxiliary Young diagrams
	$\mu^{(i+\frac12)}$, $i=0,1,\ldots,n$ as follows:
	\begin{align*}
		\bar\la=\mu^{(0)}=\mu^{(\frac12)}\nearrow\mu^{(1)}=\mu^{(\frac32)}\nearrow\mu^{(2)}=\mu^{(\frac52)}
		\nearrow \ldots\nearrow\mu^{(n)}=\mu^{(n+\frac12)}=\la.
	\end{align*}
	\item Independent jumps $\la\to\la+\square$
	happen according to an exponential clock with rate $p_1(\Pl_\ga)=\ga\frac{1-q}{1-t}$.
	When this clock rings,
	we pick a uniformly random
	number $m\in\{\frac12,\frac32,\ldots,n+\frac12\}$,
	and add a box $\square^{(m)}$ to $\mu^{(m)}$
	with `probability'
	\begin{align*}
		\qProb\big(\mu^{(m)}\to\mu^{(m)}+\square^{(m)}\big)
		=W_\circ
		(\mu^{(m)},\mu^{(m)}+\square^{(m)}).
	\end{align*}
	If $m\le n-\frac12$, any such move will propagate to all the higher levels
	according to the next rule.
	\item Any move happening at any level (recall that the bottommost diagram
	$\bar\la=\mu^{(0)}$  itself
	evolves according to the univariate dynamics $\Qs_{\Ab}$)
	\begin{align*}
		j=0,\tfrac12,1,\tfrac32,2,\ldots,n-\tfrac12,n,n+\tfrac12
	\end{align*}
	almost surely propagates
	all the way to the uppermost diagram $\la=\mu^{(n+\frac12)}$ according to the
	conditional `probabilities'
	\begin{align*}
		&
		\qProb(\mu^{(i+\frac12)}\to\mu^{(i+\frac12)}+\square^{(i+\frac12)}\,|\,
		\mu^{(i)}\to\mu^{(i)}+\square^{(i)})\\&\hspace{49pt}=
		V_\circ
		(\mu^{(i+\frac12)},\mu^{(i+\frac12)}+\square^{(i+\frac12)}\,|\,
		\mu^{(i)},\mu^{(i)}+\square^{(i)}),\qquad
		i=j,j+\tfrac12,j+1,\ldots,n-\tfrac12,n.
	\end{align*}
	Note that if $\mu^{(i+\frac12)}=\mu^{(i)}$ above, then it must be
	$\square^{(i+\frac12)}=\square^{(i)}$ (i.e., the above probability
	is equal to $\mathbf{1}_{\square^{(i+\frac12)}=\square^{(i)}}$).
	This is similar to the short-range pushing mechanism,
	cf. \eqref{short_range_alpha}.
\end{enumerate}
Averaging over
the $\mu^{(\cdot)}$'s with distribution \eqref{Plancherel_intermediate},
one arrives at certain functions $W_{\Pl_\ga}(\la,\cdot\,|\,\bar\la)$ and
$V_{\Pl_\ga}(\la,\cdot\,|\,\bar\la,\cdot)$
describing independent jumps and triggered moves
for the pair $\bi{\bar\la}\la$.

\begin{remark}
	It is worth noting that the above construction describes
	the evolution during a small time interval.
	In particular, each jump of the bivariate `dynamics'
	requires sampling the intermediate Young diagrams $\mu^{(i)}$ again.
	In fact, it is possible to formulate the `dynamics'
	without such an excessive sampling of intermediate diagrams
	(with the help of continuous levels $\mu^{(s)}$, where $s\in[0,1]$).
	We do this (in a slightly different language) in
	\S \ref{sub:full_sampling_algorithm} below.
\end{remark}

\begin{proposition}\label{prop:Plancherel_proof}
	Thus defined functions $(W_{\Pl_\ga},V_{\Pl_\ga})$
	correspond (as in \S \ref{sub:bivariate_dynamics})
	to an RSK-type bivariate `dynamics' $\Qs^{(2)}_{\Ab;\Pl_\ga}$
	with the Plancherel parameter $\ga>0$.
\end{proposition}
\begin{proof}
	We will check that these functions satisfy \eqref{W_conditions},
	\eqref{V_condition}, and \eqref{general_identity}
	with $\Bb=\Pl_\ga$.
	The first and second identities are straightforward.
	So, we must show that (see \eqref{general_identity})
	\begin{align}\label{general_identity_Plancherel}
		\begin{array}{ll}
		&\displaystyle
		\sum_{\bar\square\in\Ds(\bar\nu)}
		V_{\Pl_\ga}(\la,\la+\square\,|\,\bar\nu-\bar\square,\bar\nu)
		P_{\la/\bar\nu-\bar\square}(\Pl_1)\psi'_{\bar\nu/\bar\nu-\bar\square}
		\\
		&\hspace{100pt}{}+\ga^{-1}W_{\Pl_\ga}
		(\la,\la+\square\,|\,\bar\nu)P_{\la/\bar\nu}(\Pl_1)=
		P_{\la+\square/\bar\nu}(\Pl_1)\psi'_{\la+\square/\la},		
	\end{array}
	\end{align}
	for all $\bar\nu\in\Yb(\Ab)$, $\la\in\Yb(\Ab\cup\Bb)$,
	and all $\square\in\Us(\la)$, such that
	$\bi{\bar\nu}{\la+\square}\in\Yb^{(2)}(\Ab;\Bb)$.

	To simplify the argument, let us establish \eqref{general_identity_Plancherel}
	in the simplest nontrivial case
	$|\la|-|\bar\nu|=1$. The general case is analogous.

	We can write
	\begin{align}&
		V_{\Pl_\ga}(\la,\nu\,|\,\bar\la,\bar\nu)
		=
		\sum_{\mu^{(1)}}
		\frac{\varphi'_{\mu^{(1)}/\bar\la}
		\varphi'_{\la/\mu^{(1)}}}{2P_{\la/\bar\la}(\Pl_1)}
		\sum_{\kappa\colon\kappa\searrow\mu^{(1)}}
		V_\circ(\la,\nu\,|\,
		\mu^{(1)},\kappa)
		V_\circ(\mu^{(1)},\kappa\,|\,
		\bar\la,\bar\nu),\quad
		\bar\la\precb\bar\nu,\ \la\precb \nu.
		\label{VPL}
	\end{align}
	Indeed, in the definition of $V_{\Pl_\ga}$ above
	we have $n=|\la|-|\bar\la|=2$.
	In the summation, $\kappa$ is the new state of the
	diagram $\mu^{(1)}$, which means that the whole transition looks as
	$\left[\begin{smallmatrix}
		\la\\
		\rule{0pt}{7.5pt}\mu^{(1)}\\
		\rule{0pt}{7.5pt}\bar\la
	\end{smallmatrix}\right]\to
	\left[\begin{smallmatrix}
		\nu\\
		\rule{0pt}{9.5pt}\kappa\\
		\rule{0pt}{9.5pt}\bar\nu
	\end{smallmatrix}\right]$.

	On the other hand,
	the independent jump `rate' $W_{\Pl_\ga}$ is given by
	\begin{align}
		\label{WPL}
		W_{\Pl_\ga}(\la,\nu\,|\,\bar\nu)=
		\ga \frac{1-q}{1-t}\left(
		\frac12\sum_{\kappa}
		W_\circ(\bar\nu,\kappa)
		V_\circ(\la,\nu\,|\,\bar\nu,\kappa)
		+\frac12W_\circ(\la,\nu\,|\,\bar\nu)\right).
	\end{align}
	Indeed, this time $n=|\la|-|\bar\nu|=1$ in the definition of $W_{\Pl_\ga}$.
	The diagram $\kappa$ now represents the new
	state of $\mu^{(\frac12)}$. The
	sum over $\kappa$ corresponds
	to an independent jump of $\mu^{(\frac12)}$, and the
	second summand corresponds to an independent jump of $\mu^{(\frac32)}$.
	The factor $\ga \frac{1-q}{1-t}$ is simply the total
	rate of an independent jump of $\la$
	in the pair $\bi{\bar\nu}\la$.

	Plugging \eqref{VPL} into a part
	\eqref{general_identity_Plancherel},
	we obtain
	\begin{align*}&
		\sum_{\bar\square\in\Ds(\bar\nu)}
		V_{\Pl_\ga}(\la,\la+\square\,|\,\bar\nu-\bar\square,\bar\nu)
		P_{\la/\bar\nu-\bar\square}(\Pl_1)\psi'_{\bar\nu/\bar\nu-\bar\square}
		\\&\hspace{13pt}=
		\sum_{\bar\square,\mu^{(1)},\kappa}
		\frac12{\varphi'_{\mu^{(1)}/\bar\nu-\bar\square}
		\varphi'_{\la/\mu^{(1)}}}
		V_\circ(\la,\la+\square\,|\,
		\mu^{(1)},\kappa)
		V_\circ(\mu^{(1)},\kappa\,|\,
		\bar\nu-\bar\square,\bar\nu)
		\psi'_{\bar\nu/\bar\nu-\bar\square}
		\\&\hspace{13pt}=
		\sum_{\mu^{(1)},\kappa}
		\frac12
		\varphi'_{\la/\mu^{(1)}}
		V_\circ(\la,\la+\square\,|\,
		\mu^{(1)},\kappa)
		\left[\varphi'_{\kappa/\bar\nu}\psi'_{\kappa/\mu^{(1)}}-
		\frac{1-q}{1-t}W_\circ(\mu^{(1)},\kappa)\mathbf{1}_{\mu^{(1)}=\bar\nu}\right]
		\\&\hspace{13pt}=
		\sum_{\mu^{(1)},\kappa}
		\frac12
		\varphi'_{\la/\mu^{(1)}}
		V_\circ(\la,\la+\square\,|\,
		\mu^{(1)},\kappa)
		\varphi'_{\kappa/\bar\nu}\psi'_{\kappa/\mu^{(1)}}
		-\frac{1-q}{1-t}\sum_{\kappa}
		\frac12
		\varphi'_{\la/\bar\nu}
		V_\circ(\la,\la+\square\,|\,
		\bar\nu,\kappa)
		W_\circ(\bar\nu,\kappa)
		\\&\hspace{13pt}=
		\sum_{\kappa}
		\frac12
		\varphi'_{\kappa/\bar\nu}
		\varphi'_{\la+\square/\kappa}\psi'_{\la+\square/\la}
		\\&\hspace{80pt}-
		\frac{1-q}{1-t}\left[\frac{1}{2}\varphi'_{\la/\bar\nu}
		W_\circ(\la,\la+\square)
		+\sum_{\kappa}
		\frac12
		\varphi'_{\la/\bar\nu}
		V_\circ(\la,\la+\square\,|\,
		\bar\nu,\kappa)
		W_\circ(\bar\nu,\kappa)\right]
	\end{align*}
	(using \eqref{general_identity_circ} and
	\eqref{phi_psi_one_box}, we summed over $\bar\square$, and then over $\mu^{(1)}$).
	By Lemma \ref{lemma:skew_Plancherel_2},
	we see that the first summand above is
	equal to $\psi'_{\la+\square/\la}P_{\la+\square/\bar\nu}(\Pl_1)$, which is the
	right-hand side of \eqref{general_identity_Plancherel}.
	The two remaining summands above cancel with
	\begin{align*}
		\ga^{-1}W_{\Pl_\ga}
		(\la,\la+\square\,|\,\bar\nu)P_{\la/\bar\nu}(\Pl_1)
		=
		\ga^{-1}W_{\Pl_\ga}
		(\la,\la+\square\,|\,\bar\nu)\varphi'_{\la/\bar\nu},
	\end{align*}
	see \eqref{WPL}. This concludes the proof.
\end{proof}

\begin{remark}\label{rmk:indep_of_A}
	One can readily see that
	the functions $(W_{\Bb},V_{\Bb})$
	constructed above
	(corresponding to bivariate `dynamics'
	$\Qs^{(2)}_{\Ab;\Bb}$,
	where $\Bb$ is $(\al),(\be)$, or $\Pl_\ga$)
	do not depend on the ``lower'' specialization $\Ab$,
	cf. Remark \ref{rmk:general_identity}.(3).
	This is the reason why we didn't include $\Ab$
	in the notation.
\end{remark}




\section{RSK-type algorithm for sampling HL-coherent measures} 
\label{sec:rsk_type_algorithm_for_sampling_hl_coherent_measures}

From now on we will assume that the Macdonald parameter $q$ is zero.\footnote{It is possible 
to develop randomized `sampling' algorithms (i.e., formal Markov `dynamics' with negative probabilities of certain elements in a `transition matrix')
for the general parameters $(q,t)$ by analogy, but we will not pursue this direction here.}
For $q=0$, we construct a randomized algorithm for sampling
HL-coherent measures on Young diagrams
(\S \ref{sub:_q_t_coherent_measures})
which is an honest probabilistic object, that is,
involves only nonnegative probabilities.

In contrast with the setting of \S \ref{sec:macdonald_processes_and_bivariate_continuous_time_dynamics_}
and \S \ref{sec:three_particular_dynamics_on_macdonald_processes},
the discussion of sampling algorithms
is simpler in the language of de-poissonized measures
$\HL_n^{\ab;\bb;\Pl_\ga}$
(cf. Remark \ref{rmk:depoiss}).
One can also readily describe the poissonized version of the algorithm,
see Remark \ref{rmk:poiss_alg} below.
It is worth noting that this does not affect the statement
of the Law of Large Numbers (Theorem \ref{thm:LLN_intro}).

\subsection{$t$-quantities} 
\label{sub:_t_quantities}

First, we need to understand how the quantities $T_i(\bar\nu,\la\,|\,q,t)$ and $S_j(\bar\nu,\la\,|\,q,t)$
\eqref{T_i}--\eqref{S_j} look like when we
take parameters $(0,t)$. Recall that they are defined for $\bar\nu\prech\la$.
Denote $\ell:=\ell(\bar\nu)$, so (possibly appending $\la$ by zeroes)
we may think that $\ell(\la)=\ell+1$.
\begin{proposition}\label{prop:T_0t}
	For each fixed $i=1,\ldots,\ell$, the quantity $T_i(\bar\nu,\la\,|\,0,t)$
	is determined according to the following rules.
	\begin{enumerate}[(T1)]
		\item[(T0)] If $\bar\nu_i=\la_{i+1}$, then $T_i(\bar\nu,\la\,|\,0,t)=0$.
		If $\la_i>\bar\nu_i$ and $\la_{i+1}<\bar\nu_i-1$, then
		$T_i(\bar\nu,\la\,|\,0,t)=1$.
	\end{enumerate}
	Otherwise, denote
	\begin{align*}
		R:=(\mbox{multiplicity of $\bar\nu_i$ in $\bar\nu$})-1,\qquad
		L:=(\mbox{multiplicity of $\bar\nu_i-1$ in $\bar\nu$})
	\end{align*}
	(any of these numbers can be zero).
	One checks that the multiplicity of $\bar\nu_i$ in $\la$ is either $R$ or $R+1$,
	and the multiplicity of $\bar\nu_i-1$ in $\la$ is either $L$ or $L+1$ (see Fig.~\ref{fig:T}). The value of
	$T_i$ in each of the four cases is given by
	\begin{enumerate}[(T1)]
		\item For $(L,R)$, we have $T_i(\bar\nu,\la\,|\,0,t)=({1-t^{L+1}})/({1-t})$.
		\item For $(L,R+1)$, we have $T_i(\bar\nu,\la\,|\,0,t)=({1-t^{L+1}})({1-t^{R+1}})/({1-t})$.
		\item For $(L+1,R)$, we have $T_i(\bar\nu,\la\,|\,0,t)=1/({1-t})$.
		\item For $(L+1,R+1)$, we have $T_i(\bar\nu,\la\,|\,0,t)=({1-t^{R+1}})/({1-t})$.
	\end{enumerate}
	Cases (T0)--(T4) exhaust all possible configurations.
\end{proposition}
\begin{proof}
	Direct $q=0$ substitution in formula \eqref{T_i}.
\end{proof}
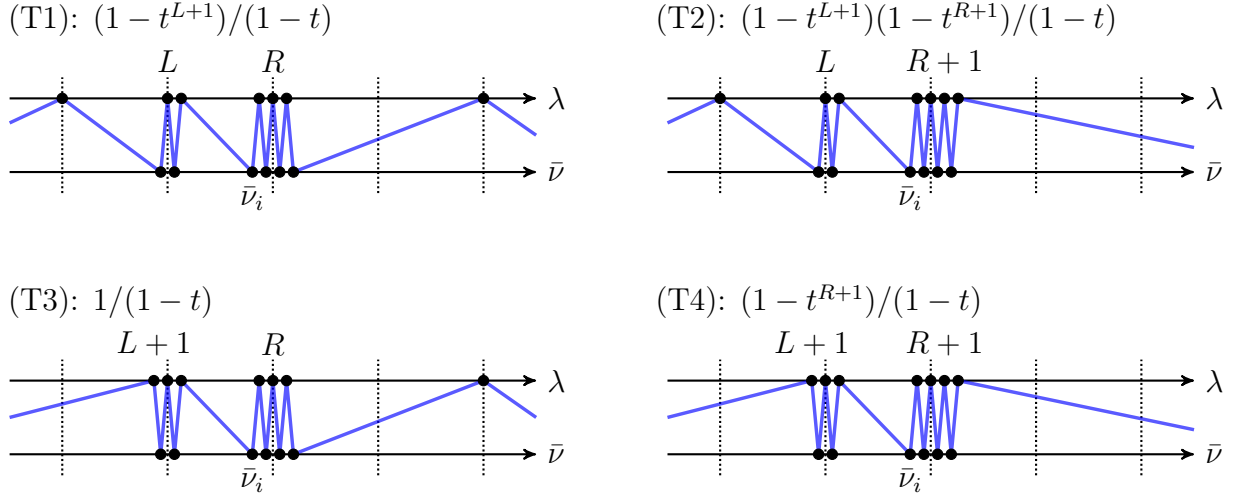
\begin{figure}[htbp]
\begin{center}
	\begin{tabular}{ll}
		(T1): $({1-t^{L+1}})/({1-t})$&\hspace{20pt}(T2): $({1-t^{L+1}})({1-t^{R+1}})/({1-t})$\\
		\begin{tikzpicture}[
		    scale=1.4,
		    axis/.style={thick, ->, >=stealth'}]
		    \def\y{.7}
		    \draw[axis] (0,0) -- (5,0) node(xline)[right]{$\bar\nu$};
		    \draw[axis] (0,\y) -- (5,\y) node(xline)[right]{$\la$};
		    \foreach \hh in {.5, 1.5, 2.5, 3.5, 4.5}
		    {
		    	\draw[densely dotted, thick] (\hh,-.2) -- (\hh,\y+.2);
		    }
		    \def\sp{.13};
		    \def\opac{.65}
			\draw[line width=1.3, color=blue, opacity=\opac]
			(5,\y/2)--(4.5,\y)--++(-2+3/2*\sp,-\y)--++(-\sp/2,\y)--++(-\sp/2,-\y)
			--++(-\sp/2,\y)--++(-\sp/2,-\y)--++(-\sp/2,\y)--++(-\sp/2,-\y)
			--(1.5+\sp,\y)--(1.5+\sp/2,0)--++(-\sp/2,\y)--++(-\sp/2,-\y)--(.5,\y)--(0,2*\y/3);
		    \foreach \pt in
		    {(1.5-\sp/2,0),(1.5+\sp/2,0),
		    (1.5+\sp,\y),(1.5,\y),
		    (2.5-\sp/2,0),(2.5+\sp/2,0),(2.5-3*\sp/2,0),(2.5+3*\sp/2,0),
		    (2.5,\y),(2.5+\sp,\y),(2.5-\sp,\y),(4.5,\y),(.5,\y)}
		    {
		    	\draw[fill] \pt circle (1.4pt);
		    }
		    \node at(2.5-3/2*\sp,-.25) {$\bar\nu_i$};
		    \node at(1.5,\y+.35) {$L$};
		    \node at(2.5,\y+.35) {$R$};
		\end{tikzpicture}
		&\hspace{20pt}
		\begin{tikzpicture}[
		    scale=1.4,
		    axis/.style={thick, ->, >=stealth'}]
		    \def\y{.7}
		    \draw[axis] (0,0) -- (5,0) node(xline)[right]{$\bar\nu$};
		    \draw[axis] (0,\y) -- (5,\y) node(xline)[right]{$\la$};
		    \foreach \hh in {.5, 1.5, 2.5, 3.5, 4.5}
		    {
		    	\draw[densely dotted, thick] (\hh,-.2) -- (\hh,\y+.2);
		    }
		    \def\sp{.13};
		    \def\opac{.65}
			\draw[line width=1.3, color=blue, opacity=\opac]
			(5,\y/3)--(2.5+2*\sp,\y)--(2.5+3/2*\sp,0)--++(-\sp/2,\y)--++(-\sp/2,-\y)
			--++(-\sp/2,\y)--++(-\sp/2,-\y)--++(-\sp/2,\y)--++(-\sp/2,-\y)
			--(1.5+\sp,\y)--(1.5+\sp/2,0)--++(-\sp/2,\y)--++(-\sp/2,-\y)--(.5,\y)--(0,2*\y/3);
		    \foreach \pt in
		    {(1.5-\sp/2,0),(1.5+\sp/2,0),
		    (1.5+\sp,\y),(1.5,\y),
		    (2.5-\sp/2,0),(2.5+\sp/2,0),(2.5-3*\sp/2,0),(2.5+3*\sp/2,0),
		    (2.5,\y),(2.5+\sp,\y),(2.5-\sp,\y),(2.5+2*\sp,\y),(.5,\y)}
		    {
		    	\draw[fill] \pt circle (1.4pt);
		    }
		    \node at(2.5-3/2*\sp,-.25) {$\bar\nu_i$};
		    \node at(1.5,\y+.35) {$L$};
		    \node at(2.5+\sp,\y+.35) {$R+1$};
		\end{tikzpicture}
		\\
		(T3): $1/({1-t})$&\hspace{20pt}(T4): $({1-t^{R+1}})/({1-t})$\rule{0pt}{30pt}
		\\
		\begin{tikzpicture}[
		    scale=1.4,
		    axis/.style={thick, ->, >=stealth'}]
		    \def\y{.7}
		    \draw[axis] (0,0) -- (5,0) node(xline)[right]{$\bar\nu$};
		    \draw[axis] (0,\y) -- (5,\y) node(xline)[right]{$\la$};
		    \foreach \hh in {.5, 1.5, 2.5, 3.5, 4.5}
		    {
		    	\draw[densely dotted, thick] (\hh,-.2) -- (\hh,\y+.2);
		    }
		    \def\sp{.13};
		    \def\opac{.65}
			\draw[line width=1.3, color=blue, opacity=\opac]
			(5,\y/2)--(4.5,\y)--++(-2+3/2*\sp,-\y)--++(-\sp/2,\y)--++(-\sp/2,-\y)
			--++(-\sp/2,\y)--++(-\sp/2,-\y)--++(-\sp/2,\y)--++(-\sp/2,-\y)
			--(1.5+\sp,\y)--(1.5+\sp/2,0)--++(-\sp/2,\y)--++(-\sp/2,-\y)--(1.5-\sp,\y)--(0,\y/2);
		    \foreach \pt in
		    {(1.5-\sp/2,0),(1.5+\sp/2,0),
		    (1.5+\sp,\y),(1.5,\y),
		    (2.5-\sp/2,0),(2.5+\sp/2,0),(2.5-3*\sp/2,0),(2.5+3*\sp/2,0),
		    (2.5,\y),(2.5+\sp,\y),(2.5-\sp,\y),(4.5,\y),(1.5-\sp,\y)}
		    {
		    	\draw[fill] \pt circle (1.4pt);
		    }
		    \node at(2.5-3/2*\sp,-.25) {$\bar\nu_i$};
		    \node at(1.5-\sp,\y+.35) {$L+1$};
		    \node at(2.5,\y+.35) {$R$};
		\end{tikzpicture}
		&\hspace{20pt}
		\begin{tikzpicture}[
		    scale=1.4,
		    axis/.style={thick, ->, >=stealth'}]
		    \def\y{.7}
		    \draw[axis] (0,0) -- (5,0) node(xline)[right]{$\bar\nu$};
		    \draw[axis] (0,\y) -- (5,\y) node(xline)[right]{$\la$};
		    \foreach \hh in {.5, 1.5, 2.5, 3.5, 4.5}
		    {
		    	\draw[densely dotted, thick] (\hh,-.2) -- (\hh,\y+.2);
		    }
		    \def\sp{.13};
		    \def\opac{.65}
			\draw[line width=1.3, color=blue, opacity=\opac]
			(5,\y/3)--(2.5+2*\sp,\y)--(2.5+3/2*\sp,0)--++(-\sp/2,\y)--++(-\sp/2,-\y)
			--++(-\sp/2,\y)--++(-\sp/2,-\y)--++(-\sp/2,\y)--++(-\sp/2,-\y)
			--(1.5+\sp,\y)--(1.5+\sp/2,0)--++(-\sp/2,\y)--++(-\sp/2,-\y)--(1.5-\sp,\y)--(0,\y/2);
		    \foreach \pt in
		    {(1.5-\sp/2,0),(1.5+\sp/2,0),
		    (1.5+\sp,\y),(1.5,\y),
		    (2.5-\sp/2,0),(2.5+\sp/2,0),(2.5-3*\sp/2,0),(2.5+3*\sp/2,0),
		    (2.5,\y),(2.5+\sp,\y),(2.5-\sp,\y),(2.5+2*\sp,\y),(1.5-\sp,\y)}
		    {
		    	\draw[fill] \pt circle (1.4pt);
		    }
		    \node at(2.5-3/2*\sp,-.25) {$\bar\nu_i$};
		    \node at(1.5-\sp,\y+.35) {$L+1$};
		    \node at(2.5+\sp,\y+.35) {$R+1$};
		\end{tikzpicture}
	\end{tabular}
\end{center}
\caption{Cases (T1)--(T4) used to determine the value of
$T_i(\bar\nu,\la\,|\,0,t)$. On the picture we have $R=3$ and $L=2$. Young diagrams $\bar\nu\prech\la$
are represented by interlacing particle configurations.}
\label{fig:T}
\end{figure}

\begin{proposition}\label{prop:S_0t}
	For each fixed $j=1,\ldots,\ell+1$, the quantity $S_j(\bar\nu,\la\,|\,0,t)$
	is determined according to the following rules.
	\begin{enumerate}[(S1)]
		\item[(S0)] If $\la_j=\bar\nu_{j-1}$, then $S_j(\bar\nu,\la\,|\,0,t)=0$.
		If $\la_j>\bar\nu_j$ and $\la_j<\bar\nu_{j-1}-1$, then $S_j(\bar\nu,\la\,|\,0,t)=1$.
	\end{enumerate}
	Otherwise, denote
	\begin{align*}
		R:=(\mbox{multiplicity of $\la_{j}+1$ in $\la$}),\qquad
		L:=(\mbox{multiplicity of $\la_j$ in $\la$})-1
	\end{align*}
	(any of these numbers can be zero).
	Clearly, the multiplicity of $\la_{j}+1$ in $\bar\nu$ can be either $R$ or $R+1$,
	and the multiplicity of $\la_j$ in $\bar\nu$ is either $L$ or $L+1$ (see Fig.~\ref{fig:S}).
	The value of
	$S_j$ in each of the four cases is given by
	\begin{enumerate}[(S1)]
		\item For $(L,R)$, we have $S_j(\bar\nu,\la\,|\,0,t)=({1-t^{R+1}})/({1-t})$.
		\item For $(L,R+1)$, we have $S_j(\bar\nu,\la\,|\,0,t)=1/({1-t})$.
		\item For $(L+1,R)$, we have $S_j(\bar\nu,\la\,|\,0,t)=({1-t^{L+1}})({1-t^{R+1}})/({1-t})$.
		\item For $(L+1,R+1)$, we have $S_j(\bar\nu,\la\,|\,0,t)=({1-t^{L+1}})/({1-t})$.
	\end{enumerate}	
	Cases (S0)--(S4) exhaust all possible configurations.
\end{proposition}
\begin{proof}
	Direct $q=0$ substitution in formula \eqref{S_j}.
\end{proof}
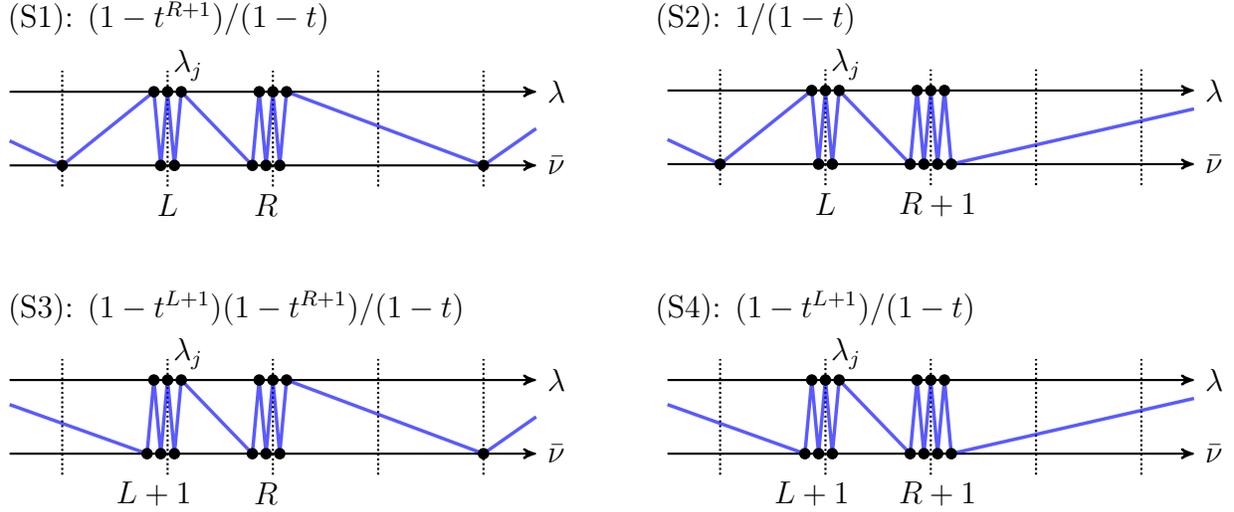
\begin{figure}[htbp]
\begin{center}
	\begin{tabular}{ll}
		(S1): $({1-t^{R+1}})/({1-t})$&\hspace{20pt}(S2): $1/({1-t})$\\
		\begin{tikzpicture}[
		    scale=1.4,
		    axis/.style={thick, ->, >=stealth'}]
		    \def\y{.7}
		    \draw[axis] (0,0) -- (5,0) node(xline)[right]{$\bar\nu$};
		    \draw[axis] (0,\y) -- (5,\y) node(xline)[right]{$\la$};
		    \foreach \hh in {.5, 1.5, 2.5, 3.5, 4.5}
		    {
		    	\draw[densely dotted, thick] (\hh,-.2) -- (\hh,\y+.2);
		    }
		    \def\sp{.13};
		    \def\opac{.65}
			\draw[line width=1.3, color=blue, opacity=\opac]
			(5,\y/2)--(4.5,0)--(2.5+\sp,\y)--++(-\sp/2,-\y)--++(-\sp/2,\y)
			--++(-\sp/2,-\y)--++(-\sp/2,\y)--++(-\sp/2,-\y)
			--(1.5+\sp,\y)--++(-\sp/2,-\y)--++(-\sp/2,\y)--++(-\sp/2,-\y)--++(-\sp/2,\y)
			--(.5,0)--(0,\y/3)
			;
		    \foreach \pt in
		    {(1.5-\sp/2,0),(1.5+\sp/2,0),
		    (1.5+\sp,\y),(1.5,\y),(1.5-\sp,\y),
		    (2.5-\sp/2,0),(2.5+\sp/2,0),(2.5-3*\sp/2,0),
		    (2.5,\y),(2.5+\sp,\y),(2.5-\sp,\y),(4.5,0),(.5,0)}
		    {
		    	\draw[fill] \pt circle (1.4pt);
		    }
		    \node at(1.5+3/2*\sp,\y+.25) {$\la_j$};
		    \node at(1.5,-.38) {$L$};
		    \node at(2.5-\sp/2,-.38) {$R$};
		\end{tikzpicture}
		&\hspace{20pt}
		\begin{tikzpicture}[
		    scale=1.4,
		    axis/.style={thick, ->, >=stealth'}]
		    \def\y{.7}
		    \draw[axis] (0,0) -- (5,0) node(xline)[right]{$\bar\nu$};
		    \draw[axis] (0,\y) -- (5,\y) node(xline)[right]{$\la$};
		    \foreach \hh in {.5, 1.5, 2.5, 3.5, 4.5}
		    {
		    	\draw[densely dotted, thick] (\hh,-.2) -- (\hh,\y+.2);
		    }
		    \def\sp{.13};
		    \def\opac{.65}
			\draw[line width=1.3, color=blue, opacity=\opac]
			(5,3*\y/4)--(2.5+3*\sp/2,0)--(2.5+\sp,\y)--++(-\sp/2,-\y)--++(-\sp/2,\y)
			--++(-\sp/2,-\y)--++(-\sp/2,\y)--++(-\sp/2,-\y)
			--(1.5+\sp,\y)--++(-\sp/2,-\y)--++(-\sp/2,\y)--++(-\sp/2,-\y)--++(-\sp/2,\y)
			--(.5,0)--(0,\y/3)
			;
		    \foreach \pt in
		    {(1.5-\sp/2,0),(1.5+\sp/2,0),
		    (1.5+\sp,\y),(1.5,\y),(1.5-\sp,\y),
		    (2.5-\sp/2,0),(2.5+\sp/2,0),(2.5-3*\sp/2,0),
		    (2.5,\y),(2.5+\sp,\y),(2.5-\sp,\y),(2.5+3*\sp/2,0),(.5,0)}
		    {
		    	\draw[fill] \pt circle (1.4pt);
		    }
		    \node at(1.5+3/2*\sp,\y+.25) {$\la_j$};
		    \node at(1.5,-.38) {$L$};
		    \node at(2.5+\sp/2,-.38) {$R+1$};
		\end{tikzpicture}
		\\
		(S3): $({1-t^{L+1}})({1-t^{R+1}})/({1-t})$&\hspace{20pt}(S4):
		$({1-t^{L+1}})/({1-t})$\rule{0pt}{30pt}
		\\
		\begin{tikzpicture}[
		    scale=1.4,
		    axis/.style={thick, ->, >=stealth'}]
		    \def\y{.7}
		    \draw[axis] (0,0) -- (5,0) node(xline)[right]{$\bar\nu$};
		    \draw[axis] (0,\y) -- (5,\y) node(xline)[right]{$\la$};
		    \foreach \hh in {.5, 1.5, 2.5, 3.5, 4.5}
		    {
		    	\draw[densely dotted, thick] (\hh,-.2) -- (\hh,\y+.2);
		    }
		    \def\sp{.13};
		    \def\opac{.65}
			\draw[line width=1.3, color=blue, opacity=\opac]
			(5,\y/2)--(4.5,0)--(2.5+\sp,\y)--++(-\sp/2,-\y)--++(-\sp/2,\y)
			--++(-\sp/2,-\y)--++(-\sp/2,\y)--++(-\sp/2,-\y)
			--(1.5+\sp,\y)--++(-\sp/2,-\y)--++(-\sp/2,\y)--++(-\sp/2,-\y)--++(-\sp/2,\y)
			--(1.5-3*\sp/2,0)--(0,2*\y/3)
			;
		    \foreach \pt in
		    {(1.5-\sp/2,0),(1.5+\sp/2,0),
		    (1.5+\sp,\y),(1.5,\y),(1.5-\sp,\y),
		    (2.5-\sp/2,0),(2.5+\sp/2,0),(2.5-3*\sp/2,0),
		    (2.5,\y),(2.5+\sp,\y),(2.5-\sp,\y),(4.5,0),(1.5-3*\sp/2,0)}
		    {
		    	\draw[fill] \pt circle (1.4pt);
		    }
		    \node at(1.5+3/2*\sp,\y+.25) {$\la_j$};
		    \node at(1.5-\sp,-.38) {$L+1$};
		    \node at(2.5-\sp/2,-.38) {$R$};
		\end{tikzpicture}
		&\hspace{20pt}
		\begin{tikzpicture}[
		    scale=1.4,
		    axis/.style={thick, ->, >=stealth'}]
		    \def\y{.7}
		    \draw[axis] (0,0) -- (5,0) node(xline)[right]{$\bar\nu$};
		    \draw[axis] (0,\y) -- (5,\y) node(xline)[right]{$\la$};
		    \foreach \hh in {.5, 1.5, 2.5, 3.5, 4.5}
		    {
		    	\draw[densely dotted, thick] (\hh,-.2) -- (\hh,\y+.2);
		    }
		    \def\sp{.13};
		    \def\opac{.65}
			\draw[line width=1.3, color=blue, opacity=\opac]
			(5,3*\y/4)--(2.5+3*\sp/2,0)--(2.5+\sp,\y)--++(-\sp/2,-\y)--++(-\sp/2,\y)
			--++(-\sp/2,-\y)--++(-\sp/2,\y)--++(-\sp/2,-\y)
			--(1.5+\sp,\y)--++(-\sp/2,-\y)--++(-\sp/2,\y)--++(-\sp/2,-\y)--++(-\sp/2,\y)
			--(1.5-3*\sp/2,0)--(0,2*\y/3)
			;
		    \foreach \pt in
		    {(1.5-\sp/2,0),(1.5+\sp/2,0),
		    (1.5+\sp,\y),(1.5,\y),(1.5-\sp,\y),
		    (2.5-\sp/2,0),(2.5+\sp/2,0),(2.5-3*\sp/2,0),
		    (2.5,\y),(2.5+\sp,\y),(2.5-\sp,\y),(2.5+3*\sp/2,0),(1.5-3*\sp/2,0)}
		    {
		    	\draw[fill] \pt circle (1.4pt);
		    }
		    \node at(1.5+3/2*\sp,\y+.25) {$\la_j$};
		    \node at(1.5-\sp,-.38) {$L+1$};
		    \node at(2.5+\sp/2,-.38) {$R+1$};
		\end{tikzpicture}
	\end{tabular}
\end{center}
\caption{Cases (S1)--(S4) used to determine the value of
$S_j(\bar\nu,\la\,|\,0,t)$. On the picture we have $R=3$ and $L=2$.}
\label{fig:S}
\end{figure}

\begin{remark}\label{rmk:T_S_correspondence}
	Observe that cases
	(T1), (T2), (T3), and (T4)
	on Fig.~\ref{fig:T}
	describing the configuration around
	the
	particle $\bar\nu_i$
	correspond to cases (S4), (S3), (S2), and (S1), respectively,
	for the configuration around the particle $\la_{i+1}$
	(see Fig.~\ref{fig:S}).
\end{remark}

For purposes
of `dynamics' with a dual variable (\S \ref{sub:duality}),
we will also need the same quantities
$T_i(\bar\nu,\la\,|\,q,t)$ and $S_j(\bar\nu,\la\,|\,q,t)$
\eqref{T_i}--\eqref{S_j}
with
parameters $(t,0)$:
\begin{proposition}[{\cite[\S8.1]{BorodinPetrov2013NN}}]\label{prop:t0}
	Let $\bar\nu\prech\la$, $\ell(\bar\nu)=\ell$,
	$\ell(\la)=\ell+1$. Then
	\begin{align*}
		T_i(\bar\nu,\la\,|\,t,0)&=\frac{(1-t^{\bar\nu_i-\la_{i+1}})
		(1-t^{\bar\nu_{i-1}-\bar\nu_{i}+1}\mathbf{1}_{i>1})}
		{1-t^{\la_i-\bar\nu_i+1}};
		\\\rule{0pt}{20pt}
		S_j(\bar\nu,\la\,|\,t,0)&=
		\frac{(1-t^{\bar\nu_{j-1}-\la_j}\mathbf{1}_{j>1})(1-
		t^{\la_j-\la_{j+1}+1}\mathbf{1}_{j<\ell+1})}
		{1-t^{\la_j-\bar\nu_j+1}\mathbf{1}_{j<\ell+1}},
	\end{align*}
	where $i=1,\ldots,\ell$ and $j=1,\ldots,\ell+1$.
\end{proposition}


\subsection{Pushing `probabilities'} 
\label{sub:pushing_and_pulling_probabilities}

Let us write down explicit formulas for
the pushing `probabilities' $\rp_j^{h}(\bar\nu,\la)$ \eqref{r_j_h}
in the cases of Macdonald parameters $(0,t)$ and $(t,0)$.
Here $h\in\{1,2,\ldots\}\cup\{+\infty\}$ is an additional parameter as before.

Recall \eqref{V_alpha} that the `probabilities' $\rp_j^{h}(\bar\nu,\la)$
are defined for all $j\in\{1,\ldots,\ell(\bar\nu)\}$
such that $\bar\nu_j>\la_{j+1}$. For each such $j$,
$\rp_j^{h}(\bar\nu,\la)$ represents the `probability' that
the particle $\bar\nu_j$ which has just moved on the lower level
will push its first free upper right neighbor.
With the complement `probability' $1-\rp_j^{h}(\bar\nu,\la)$,
the particle $\bar\nu_j$ will pull its upper left neighbor $\la_{j+1}$.
See Fig.~\ref{fig:rl} above.

From \eqref{r_j_h} one has that
$\rp_j^{h}(\bar\nu,\la)=\rp_j^{+\infty}(\bar\nu,\la)-
{T_j^{-1}(\bar\nu,\la\,|\,q,t)}{\mathbf{1}_{j\ge h}}$.
For the Macdonald parameters $(0,t)$, the quantities
$\rp_j^{+\infty}$ have the following form:
\begin{proposition}\label{prop:rj_0t}
	Assume that $\bar\nu\prech\la$, and $j=1,\ldots,\ell(\bar\nu)$ is
	such that $\bar\nu_j>\la_{j+1}$. Denote
	\begin{align*}
		D:=(\mbox{multiplicity of $\bar\nu_j-1$ in $\bar\nu$}),
		\qquad
		U:=(\mbox{multiplicity of $\bar\nu_j-1$ in $\la$})
	\end{align*}
	(any of the numbers can be zero).
	Clearly,
	$U=D$ or $U=D+1$
	(see Fig.~\ref{fig:r}). There are two cases:
	\begin{enumerate}[(r1)]
		\item If $U=D$, then $\rp_j^{+\infty}(\bar\nu,\la\,|\,0,t)=(1-t)/(1-t^{D+1})$.
		\item If $U=D+1$, then $\rp_j^{+\infty}(\bar\nu,\la\,|\,0,t)=1-t$.
	\end{enumerate}
	Cases (r1)--(r2) exhaust all possible configurations.
	Note that when particles are apart,
	more precisely, when
	$\la_{j+1}<\bar\nu_j-1$, then
	$\rp_j^{+\infty}(\bar\nu,\la\,|\,0,t)=1$ by (r1).
\end{proposition}
\begin{figure}[htbp]
\begin{center}
	\begin{tabular}{ll}
		(r1)&\hspace{20pt}(r2)\\
		\begin{tikzpicture}[
		    scale=1.4,
		    axis/.style={thick, ->, >=stealth'},
		    block/.style ={rectangle, draw=red,
			align=center, rounded corners, minimum height=1em}]
		    \def\y{.7}
		    \draw[axis] (0,0) -- (5,0) node(xline)[right]{$\bar\nu$};
		    \draw[axis] (0,\y) -- (5,\y) node(xline)[right]{$\la$};
		    \foreach \hh in {.5, 1.5, 2.5, 3.5, 4.5}
		    {
		    	\draw[densely dotted, thick, opacity=.5] (\hh,-.13) -- (\hh,\y+.13);
		    }
		    \def\sp{.13};
		    \def\opac{.25}
			\draw[line width=1.3, color=blue, opacity=\opac]
			(0,2*\y/3)--(.5,\y)--(2.5-\sp,0)--++(\sp/2,\y)--++(\sp/2,-\y)
			--++(\sp/2,\y)--(3.5-\sp/2,0)--++(\sp/2,\y)--++(\sp/2,-\y)
			--(4.5,\y)--(5,3*\y/4);
		    \foreach \pt in
		    {(2.5-\sp,0),(2.5,0),
		    (3.5+\sp/2,0),(3.5-\sp/2,0),
		    (.5,\y),(2.5-\sp/2,\y),(2.5+\sp/2,\y),(3.5,\y),
		    (4.5,\y)}
		    {
		    	\draw[fill] \pt circle (1.4pt);
		    }
		    \draw[->,densely dotted, very thick, color = blue]
	    	(2.5+\sp-.03,-0.03) to [in=180, out=-60] (3,-.4) to [in=-120, out=0] (3.5-\sp/2-.03,-0.05);
	    	\draw[->,densely dotted, ultra thick]
	    	(3.5-\sp/2-.03,-0.05) to [in=-120, out=60] (4.5-.03,\y-0.05);
	    	\draw[->,densely dotted, ultra thick]
	    	(3.5-\sp/2,0) to [in=-80, out=170] (2.5+\sp/2+.02,\y-0.05);
	    	\node at (2.5-\sp,-.26) {$D$};
	    	\node at (2.5,\y+.3) {$D$};
	    	\node at (3.2, -.6) {\scriptsize\color{blue}just moved};
	    	\node at (3.35,.2) {$\bar\nu_j$};
	    	\draw (4.3,\y+.5) node[block] (r) {$\frac{1-t}{1-t^3}$};
	    	\draw (1.4,\y+.5) node[block] (l) {$1-\frac{1-t}{1-t^3}$};
	    	\draw[color=red] (l.south) -- (2.86,.21);
	    	\draw[color=red] (r.south) -- (4.13,.35);
		\end{tikzpicture}
		&\hspace{20pt}
		\begin{tikzpicture}[
		    scale=1.4,
		    axis/.style={thick, ->, >=stealth'},
		    block/.style ={rectangle, draw=red,
			align=center, rounded corners, minimum height=1em}]
		    \def\y{.7}
		    \draw[axis] (0,0) -- (5,0) node(xline)[right]{$\bar\nu$};
		    \draw[axis] (0,\y) -- (5,\y) node(xline)[right]{$\la$};
		    \foreach \hh in {.5, 1.5, 2.5, 3.5, 4.5}
		    {
		    	\draw[densely dotted, thick, opacity=.5] (\hh,-.13) -- (\hh,\y+.13);
		    }
		    \def\sp{.13};
		    \def\opac{.25}
			\draw[line width=1.3, color=blue, opacity=\opac]
			(0,\y/4)--(2.5-3*\sp/2,\y)--(2.5-\sp,0)--++(\sp/2,\y)--++(\sp/2,-\y)
			--++(\sp/2,\y)--(3.5-\sp/2,0)--++(\sp/2,\y)--++(\sp/2,-\y)
			--(4.5,\y)--(5,3*\y/4);
		    \foreach \pt in
		    {(2.5-\sp,0),(2.5,0),
		    (3.5+\sp/2,0),(3.5-\sp/2,0),
		    (2.5-3*\sp/2,\y),(2.5-\sp/2,\y),(2.5+\sp/2,\y),(3.5,\y),
		    (4.5,\y)}
		    {
		    	\draw[fill] \pt circle (1.4pt);
		    }
		    \draw[->,densely dotted, very thick, color = blue]
	    	(2.5+\sp-.03,-0.03) to [in=180, out=-60] (3,-.4) to [in=-120, out=0] (3.5-\sp/2-.03,-0.05);
	    	\draw[->,densely dotted, ultra thick]
	    	(3.5-\sp/2-.03,-0.05) to [in=-120, out=60] (4.5-.03,\y-0.05);
	    	\draw[->,densely dotted, ultra thick]
	    	(3.5-\sp/2,0) to [in=-80, out=170] (2.5+\sp/2+.02,\y-0.05);
	    	\node at (2.5-\sp,-.26) {$D$};
	    	\node at (2.5-\sp/2,\y+.3) {$D+1$};
	    	\node at (3.2, -.6) {\scriptsize\color{blue}just moved};
	    	\node at (3.35,.2) {$\bar\nu_j$};
	    	\draw (4.2,\y+.5) node[block] (r) {$1-t$};
	    	\draw (3.2,\y+.5) node[block] (l) {$t$};
	    	\draw[color=red] (l.south) -- (2.86,.21);
	    	\draw[color=red] (r.south) -- (4.13,.35);
		\end{tikzpicture}
	\end{tabular}
\end{center}
\caption{Cases (r1)--(r2) used to determine the value of
$\rp_j^{+\infty}(\bar\nu,\la\,|\,0,t)$. On the picture we have $D=2$.}
\label{fig:r}
\end{figure}
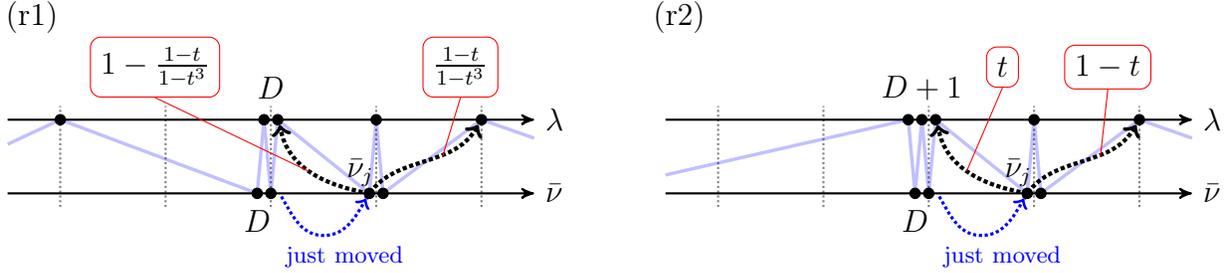
\begin{proof}
	Let us represent the interlacing configuration
	$\bar\nu\prech\la$ as the union of blocks of particles sitting
	at the same position. There are three
	possible types of such blocks depending on
	the difference between the number of
	particles on the upper and the lower levels
	(notation reflects typical shape of zigzags; we understand
	that $\NN$ could also mean $\reflectbox{$\NN$}$):
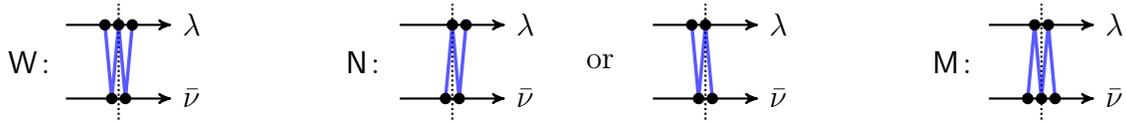
	\begin{align*}
		\begin{tikzpicture}[
		    scale=1.4,
		    axis/.style={thick, ->, >=stealth'}]
		    \def\y{.7}
		    \node at (.7,\y/2) {$\WW\colon$};
		    \draw[axis] (1,0) -- (2,0) node(xline)[right]{$\bar\nu$};
		    \draw[axis] (1,\y) -- (2,\y) node(xline)[right]{$\la$};
		    \foreach \hh in {1.5}
		    {
		    	\draw[densely dotted, thick] (\hh,-.2) -- (\hh,\y+.2);
		    }
		    \def\sp{.13};
		    \def\opac{.65}
			\draw[line width=1.3, color=blue, opacity=\opac]
			(1.5+\sp,\y)--++(-\sp/2,-\y)--++(-\sp/2,\y)--++(-\sp/2,-\y)
			--++(-\sp/2,\y);
		    \foreach \pt in
		    {(1.5-\sp/2,0),(1.5+\sp/2,0),
		    (1.5+\sp,\y),(1.5,\y),(1.5-\sp,\y)}
		    {
		    	\draw[fill] \pt circle (1.4pt);
		    }
		\end{tikzpicture}
		\qquad
		\qquad
		\begin{tikzpicture}[
		    scale=1.4,
		    axis/.style={thick, ->, >=stealth'}]
		    \def\y{.7}
		    \node at (.7,\y/2) {$\NN\colon$};
		    \draw[axis] (1,0) -- (2,0) node(xline)[right]{$\bar\nu$};
		    \draw[axis] (1,\y) -- (2,\y) node(xline)[right]{$\la$};
		    \foreach \hh in {1.5}
		    {
		    	\draw[densely dotted, thick] (\hh,-.2) -- (\hh,\y+.2);
		    }
		    \def\sp{.13};
		    \def\opac{.65}
			\draw[line width=1.3, color=blue, opacity=\opac]
			(1.5+\sp,\y)--++(-\sp/2,-\y)--++(-\sp/2,\y)--++(-\sp/2,-\y)
			;
		    \foreach \pt in
		    {(1.5-\sp/2,0),(1.5+\sp/2,0),
		    (1.5+\sp,\y),(1.5,\y)}
		    {
		    	\draw[fill] \pt circle (1.4pt);
		    }
		\end{tikzpicture}
		\quad
		\begin{tikzpicture}[
		    scale=1.4,
		    axis/.style={thick, ->, >=stealth'}]
		    \def\y{.7}
		    \node at (.5,\y/2) {or};
		    \draw[axis] (1,0) -- (2,0) node(xline)[right]{$\bar\nu$};
		    \draw[axis] (1,\y) -- (2,\y) node(xline)[right]{$\la$};
		    \foreach \hh in {1.5}
		    {
		    	\draw[densely dotted, thick] (\hh,-.2) -- (\hh,\y+.2);
		    }
		    \def\sp{.13};
		    \def\opac{.65}
			\draw[line width=1.3, color=blue, opacity=\opac]
			(1.5+\sp/2,0)--(1.5,\y)--(1.5-\sp/2,0)--(1.5-\sp,\y)
			;
		    \foreach \pt in
		    {(1.5-\sp/2,0),(1.5+\sp/2,0),
		    (1.5-\sp,\y),(1.5,\y)}
		    {
		    	\draw[fill] \pt circle (1.4pt);
		    }
		\end{tikzpicture}
		\qquad
		\qquad
		\begin{tikzpicture}[
		    scale=1.4,
		    axis/.style={thick, ->, >=stealth'}]
		    \def\y{.7}
		    \node at (.7,\y/2) {$\MMM\colon$};
		    \draw[axis] (1,0) -- (2,0) node(xline)[right]{$\bar\nu$};
		    \draw[axis] (1,\y) -- (2,\y) node(xline)[right]{$\la$};
		    \foreach \hh in {1.5}
		    {
		    	\draw[densely dotted, thick] (\hh,-.2) -- (\hh,\y+.2);
		    }
		    \def\sp{.13};
		    \def\opac{.65}
			\draw[line width=1.3, color=blue, opacity=\opac]
			(1.5+\sp,0)--++(-\sp/2,\y)--++(-\sp/2,-\y)
			--++(-\sp/2,\y)--++(-\sp/2,-\y);
		    \foreach \pt in
		    {(1.5-\sp,0),(1.5,0),(1.5+\sp,0),
		    (1.5-\sp/2,\y),(1.5+\sp/2,\y)}
		    {
		    	\draw[fill] \pt circle (1.4pt);
		    }
		\end{tikzpicture}
	\end{align*}
	If there are $k$ particles on the lower level, then
	we will denote such block by
	$\WW_k$, $\NN_k$, or $\MMM_k$, respectively
	(for $\WW_k$, $k$ is allowed to be zero).
	There are 7 ways in which two blocks can follow
	one another (the distance between the blocks is arbitrary):
	\begin{align}
		\WW\MMM
		\qquad
		\WW\NN
		\qquad
		\NN\WW
		\qquad
		\NN\MMM
		\qquad
		\NN\NN
		\qquad
		\MMM\WW
		\qquad
		\MMM\NN{}
		\label{blocks_7}
	\end{align}

	On the other hand,
	$\rp_j^{+\infty}=T_j^{-1}\big(
	(S_1-T_0)+(S_2-T_1)+\ldots+(S_j-T_{j-1})\big)$
	\eqref{r_j_h}
	is determined by
	accumulating
	successive differences
	$S_{i}-T_{i-1}$, where $i$ runs over all indices
	from $1$ to $j$ for which $\la_i$
	is free to move to the right
	(by agreement, $T_0\equiv0$).
	Nontrivial such differences arise when
	$\la_i$ is be the rightmost particle
	in one of the blocks $\WW_b$ or $\NN_b$
	while $\bar\nu_{i-1}$ is
	the leftmost particle
	in $\MMM_a$ or $\NN_a$ to the right.
	This reduces the number of combinations
	from 7 in \eqref{blocks_7} to 4.
	Using Propositions \ref{prop:T_0t} and \ref{prop:S_0t},
	one readily computes the corresponding
	differences
	$S_{i}-T_{i-1}$:
	\begin{align}
		\begin{array}{c|c|c|c|c}
			&\WW_b \MMM_a
			&
			\WW_b \NN_a
			&
			\NN_b \MMM_a
			&
			\NN_b \NN_a
			\\\hline
			\rule{0pt}{12pt}
			S_i-T_{i-1}&
			0&
			t^{a}
			&-t^{b}
			&t^{a}-t^{b}
		\end{array}
		\label{WNM_conf}
	\end{align}
	(the difference $S_i-T_{i-1}$ does not depend on the
	distance between the blocks which can be arbitrary).
	For $i=1$, there could be only two configurations,
	$\WW_b\NN_0$ or $\NN_b\NN_0$,
	for which $S_1=1$ and $1-t^{b}$, respectively. This
	agrees with \eqref{WNM_conf}.

	Let us now explain how one can
	compute the quantities $\rp_j^{+\infty}$.
	In the case (r1), the particle $\la_{j+1}$
	belongs to a block of type $\NN$
	which can be followed by either $\MMM$ or $\NN$ on the right.
	The case (r2) consists of two other possibilities, $\WW\MMM$ and
	$\WW\NN$ ($\la_{j+1}$ belongs to $\WW$). See Fig.~\ref{fig:r}.
	
	Let us consider the case $\NN_b\NN_a$,
	$\la_{j+1}\in\NN_b$, $\bar\nu_j\in\NN_a$.
	Using \eqref{WNM_conf},
	one checks that (thanks to successive cancellations),
	$S_1+(S_2-T_1)+\ldots+(S_{j}-T_{j-1})$
	is equal to
	$1-t^{a}$ regardless of the configuration of blocks
	to the right of $\NN_a$.
	In fact, due to interlacing, any such configuration must have
	the same number of $\WW$ and $\MMM$ blocks
	(we exclude the situation $\reflectbox{$\NN$}_b\reflectbox{$\NN$}_a$
	because it does not contribute a nontrivial difference $S_i-T_{i-1}$). Moreover,
	by Proposition \ref{prop:T_0t}, we have
	$T_j={(1-t^{b+1})(1-t^{a})}/{(1-t)}$, which yields
	$\rp_{j}^{+\infty}=({1-t})/({1-t^{b+1}})$, as desired.
	The three remaining cases are analogous.
\end{proof}

For the Macdonald parameters $(t,0)$,
it is convenient to represent
$\rp_j^{h}(\bar\nu,\la)=\rp_j^{1}(\bar\nu,\la)+
{T_j^{-1}(\bar\nu,\la\,|\,q,t)}{\mathbf{1}_{j<h}}$
(see \eqref{r_j_h}).
The quantities $\rp_j^{1}$
are given in the next proposition:
\begin{proposition}[{\cite[\S8.2]{BorodinPetrov2013NN}}]\label{prop:rj_t0}
	Assume that $\bar\nu\prech\la$, and let
	$j=1,\ldots,\ell(\bar\nu)$ be such that $\bar\nu_{j}>\la_{j+1}$.
	Then
	\begin{align*}
		\rp_j^{1}(\bar\nu,\la\,|\,t,0)=t^{\la_{j}-\bar\nu_j+1}
		\frac{1-t^{\bar\nu_{j-1}-\la_j}\mathbf{1}_{j>1}}
		{1-t^{\bar\nu_{j-1}-\bar\nu_j+1}\mathbf{1}_{j>1}}.
	\end{align*}
\end{proposition}

The cases when the quantities $\rp_j^{h}$ are between 0 and 1
(i.e., when they are honest probabilities)
can also be readily described:
\begin{proposition}\label{prop:nonneg}
	For the Macdonald parameters $(0,t)$,
	\begin{align*}
		0\le \rp_j^{+\infty}(\bar\nu,\la\,|\,0,t)\le 1
		\qquad \mbox{for all possible $\bar\nu\prech\la$ and $j$},
	\end{align*}
	and for other $h=1,2,\ldots$, it can happen that
	$\rp_j^{h}(\bar\nu,\la\,|\,0,t)$ is negative.

	For the Macdonald parameters $(t,0)$,
	\begin{align*}
		0\le \rp_j^{1}(\bar\nu,\la\,|\,t,0)\le 1
		\qquad \mbox{for all possible $\bar\nu\prech\la$ and $j$},
	\end{align*}
	and for other $h=2,3,\ldots;+\infty$, it can happen that
	$\rp_j^{h}(\bar\nu,\la\,|\,t,0)$ is greater than $1$.
\end{proposition}
\begin{proof}
	Straightforward verification.
\end{proof}


\subsection{Preliminary sampling algorithms} 
\label{sub:preliminary_sampling_algorithms}

To better understand the desired sampling algorithm
presented in \S \ref{sub:full_sampling_algorithm} below,
let us begin with two simpler constructions involving univariate and bivariate dynamics.

A ``trivial'' way to sample the measure $\HL_n^{\ab;\bb;\Pl_\ga}$
is to use univariate dynamics (\S \ref{sub:univariate_dynamics}) as follows.
Let us write $\Ab=(\ab;\bb;\Pl_\ga)$ for short.

\begin{sampling}\label{alg:toy_sampling}
	Start with an empty Young diagram $\la(0)=\varnothing$.
	At each step $k=1,\ldots,n$, add to the current Young diagram
	$\la=\la(k-1)$ one of the boxes $\square\in\Us(\la)$
	with probability
	\begin{align}\label{toy_sampling}
		\Prob(\la\to\la+\square)=
		\frac{1}{p_1(\Ab)}\frac{P_{\la+\square}(\Ab\,|\,0,t)}{P_\la(\Ab\,|\,0,t)}
		\psi'_{\la+\square/\la}(0,t).
	\end{align}
	Then the distribution of $\la(n)$ is $\HL^{\Ab}_{n}$.
\end{sampling}

Indeed, under $\Qs_{\Ab}$ boxes are added to the Young
diagram (with probabilities \eqref{toy_sampling})
in continuous time according to a Poisson
process of rate $p_1(\Ab)$.
Conditioning on the event that there are
exactly
$n$ points of this Poisson process
during the time segment $[0,\tau]$,
we arrive at Algorithm \ref{alg:toy_sampling}
producing the measure
$\HL^{\Ab}_{n}$.
The fact that this algorithm works
follows from
Remark \ref{rmk:tau=0} and property \eqref{MM_P_update}.

The probabilities \eqref{toy_sampling} have a rather complicated
form due to the presence of the Hall--Littlewood polynomials evaluated
at a specialization $\Ab=(\ab;\bb;\Pl_\ga)$.
Using bivariate dynamics, it is possible to
construct sampling algorithms with simpler probabilities as follows.

\medskip

We will now explain how to sample the measure
$\HL^{\Ab\cup\Bb}_{n}$
($\Ab$ and $\Bb$ are arbitrary HL-nonnegative specializations)
using
the univariate dynamics $\Qs_\Ab$
and an RSK-type bivariate `dynamics'
$\Qs^{(2)}_{\Ab;\Bb}$ (corresponding to functions $(W,V)$, see
\S \ref{sub:bivariate_dynamics}).
One can think inductively
that we \emph{know}
how to sample the measure
$\HL^{\Ab}_{n}$ and produce
a sampling procedure for the measure
$\HL^{\Ab\cup\Bb}_{n}$ with the specialization $\Bb$ added.
In \S \ref{sub:full_sampling_algorithm}
below
we fully employ such an inductive
idea, and also explain how to choose RSK-type bivariate dynamics
with nonnegative jump rates (with the help of results of \S \ref{sub:pushing_and_pulling_probabilities}).

\begin{sampling}\label{alg:sampling_bi}
	Let $\xi_1,\ldots,\xi_n$ be independent identically distributed
	Bernoulli random variables
	with values in the two-letter
	alphabet $\{a,b\}$, and such that for $k=1,\ldots,n$,
	\begin{align*}
		\Prob(\xi_k=a)=\frac{p_1(\Ab)}{p_1(\Ab\cup\Bb)},\qquad
		\Prob(\xi_k=b)=\frac{p_1(\Bb)}{p_1(\Ab\cup\Bb)}.
	\end{align*}
	Start with a pair of empty Young diagrams
	$\bi{\bar\la(0)}{\la(0)}=\bi{\varnothing}{\varnothing}$.
	Clearly, this pair belongs to $\Yb^{(2)}(\Ab;\Bb)$.	
	At each step $k=1,\ldots,n$,
	
	$\bullet$ If $\xi_k=a$,
	add to the lower
	diagram $\bar\la=\bar\la(k-1)$
	one of the boxes $\bar\square\in\Us(\bar\la)$
	as in Algorithm~\ref{alg:toy_sampling}, i.e., with probabilities \eqref{toy_sampling}.
	Then
	add to the upper diagram $\la=\la(k-1)$
	one of the boxes $\square\in\Us(\la)$ chosen according
	to the (conditional) `probabilities'
	\begin{align*}
		\qProb(\la\to\la+\square)=V(\la,\la+\square\,|\,\bar\la,\bar\la+\bar\square).
	\end{align*}

	$\bullet$ Else, if $\xi_k=b$,
	add to the upper Young diagram $\la=\la(k-1)$
	one of the boxes $\square\in\Us(\la)$ chosen according
	to `probabilities'
	\begin{align*}
		\qProb(\la\to\la+\square)
		=\frac{1}{p_1(\Bb)}W\big(\la,\la+\square\,|\,\bar\la(k-1)\big),
	\end{align*}
	and let the lower diagram remain the same.\footnote{Thus, the lower Young diagram $\bar\la(k)$
	has a random number of boxes,
	but $\la(k)$ has exactly $k$ boxes.}

	\smallskip

	The distribution
	of the
	upper Young diagram $\la(n)$ is $\HL^{\Ab\cup\Bb}_{n}$.
\end{sampling}
The fact that this algorithm indeed
produces the desired measure
$\HL^{\Ab\cup\Bb}_{n}$ readily follows from
the definition and properties of bivariate `dynamics' (\S \ref{sub:bivariate_dynamics}).

\begin{remark}\label{rmk:empty_lower_spec}
	Algorithm \ref{alg:sampling_bi} allows the specialization
	$\Ab$ to be the trivial.
	In this case, the lower Young diagram always stays
	empty, and the bivariate `dynamics'
	$\Qs^{(2)}_{\Ab;\Bb}$ is reduced to the
	univariate dynamics
	$\Qs_{\Bb}$ (this can be seen from \eqref{general_identity}
	because in this case $\Yb(\Ab)=\{\varnothing\}$).
\end{remark}

\begin{lemma}\label{lemma:squashed}
	If in the course of Algorithm
	\ref{alg:sampling_bi} one has
	$\#\{k\colon\xi_k=b\}=0$,
	then the
	resulting lower and upper Young diagrams
	coincide, $\bar\la(n)=\la(n)$.
\end{lemma}
In this case we say that the diagrams $\bar\la(n)$ and
$\la(n)$ are \emph{squashed together}.
\begin{proof}
	Observe that the bivariate `dynamics'
	lives on the space $\Yb^{(2)}(\Ab;\Bb)$,
	and for any pair of Young diagrams
	$\bi{\bar\la}{\la}\in\Yb^{(2)}(\Ab;\Bb)$
	one has $\bar\la\subseteq\la$.
	
	Since $\xi_k=a$ for all $k$,
	at each $k$th step of Algorithm \ref{alg:sampling_bi}
	we add boxes to
	both diagrams $\bar\la(k-1)$ and $\la(k-1)$
	(i.e., Algorithm \ref{alg:sampling_bi} is always in
	case (1)).
	Arguing by induction, we may assume that
	$\bar\la(k-1)=\la(k-1)$ (because
	at step $0$ both diagrams are empty and hence coincide). Adding a box
	$\bar\square$
	to the lower Young diagram
	$\bar\la(k-1)$ violates the condition
	$\bar\la\subseteq\la$, and so
	one must simultaneously add the \emph{same} box $\square=\bar\square$ to $\la(k-1)=\bar\la(k-1)$
	because it is the only possible
	way to restore this condition.
	Indeed,
	in this case by
	\eqref{inf_skew_Cauchy} and \eqref{general_identity}
	it must be that $V(\la,\la+\square\,|\,\bar\la,\bar\la+\bar\square)=1$
	(cf. the short-range pushing mechanism \eqref{short_range_alpha}).
	Thus, we see that $\bar\la(k)=\la(k)$, and by induction
	$\bar\la(n)=\la(n)$ as well.
\end{proof}

One can also readily see that
the diagrams $\bar\la(k)$ and $\la(k)$
remain squashed
for several
initial steps
of the sampling algorithm, more precisely,
until $\xi_k=b$ for the first time.
Moreover, once the diagrams
\emph{separate},
they
can never get squashed again.
Indeed, under an RSK-type bivariate `dynamics'
it is impossible to add a box to the lower Young diagram
without adding a box to the upper~one.


\subsection{Towers of Young diagrams, $\A$-tableaux, and interlacing arrays} 
\label{sub:state_space_towers_of_young_diagrams_a_tableaux_and_interlacing_configurations}

Before describing 
our full sampling algorithm
producing the
HL-coherent measures $\HL^{\ab;\bb;\Pl_\ga}_{n}$,
let us discuss the state space of this algorithm.
This space admits several equivalent descriptions
explained in Definitions
\ref{def:tower}, \ref{def:Atableaux}, and
\ref{def:interlacing_config} below.

\begin{definition}[Towers of Young diagrams]\label{def:tower}
	The state space of the sampling algorithm consists of infinite
	\emph{towers of Young diagrams}
	$\boldsymbol\la=\{\la^{(a)}\}_{a\in\A}$, where $a$ runs over the alphabet
	\begin{align}
		\A:=\{1,2,3,\ldots\}\cup \{\hat 1, \hat2, \hat 3,\ldots\}
		\cup(0,1).
		\label{alphabet}
	\end{align}
	Also denote $\A^{[\ab]}:=\{1,2,\ldots\}$ (\emph{usual letters}),
	$\A^{[\boldsymbol\be]}:=\{\hat1,\hat2,\ldots\}$ (\emph{dual letters}),
	$\A^{[\Pl]}:=(0,1)$ (\emph{Plancherel letters}).
	We assume that the alphabet
	is linearly ordered in some way, say,
	\begin{align}\label{A_order}
		1<2<3<\ldots<\hat1<\hat2<\hat3<\ldots<\A^{[\Pl]},
	\end{align}
	where points of the continuous segment $\A^{[\Pl]}$
	are linearly ordered in the usual way.
	Young diagrams in the tower should be
	also ordered (by inclusion)
	in the same way:
	for any $a,b\in\A$ with $a<b$,
	it must be $\la^{(a)}\subseteq\la^{(b)}$.
\end{definition}

A tower of Young diagrams will be called \emph{bounded}
if the numbers 
$\{|\la^{(a)}|\}_{a\in\A}$ do not tend to~$+\infty$.
A bounded 
towers of Young diagrams 
possesses a unique
maximal 
(by inclusion)
diagram $\la^{\max}\in\Yb$,
and can be interpreted as a 
certain Young tableaux. 
(In fact, all towers of Young diagrams
produced as 
outcomes of the full sampling algorithm described
in \S\ref{sub:full_sampling_algorithm} below are bounded.)
\begin{definition}[$\A$-tableaux]\label{def:Atableaux}
	Let $\la$ be a Young diagram. A mapping
	\begin{align*}
		T\colon\{\mbox{boxes of $\la$}\}\to\A
	\end{align*}
	is called an
	\emph{$\A$-tableaux of shape $\la$}
	(cf. \cite{Vershik1986}, \cite{BufetovCLT})
	if
	the entries $T(\square)$
	weakly increase (in the sense
	of the linear order \eqref{A_order} on $\A$)
	both along rows and down columns,
	and, moreover, the following constraints are satisfied
	(see an example on Fig.~\ref{fig:Atableaux}):
	\begin{enumerate}[(1)]
		\item In any column of $\la$ there cannot be
		two identical letters from $\A^{[\ab]}$.
		\item In any row of $\la$ there cannot be
		two identical letters from $\A^{[\boldsymbol\be]}$.
		\item Each letter from $\A^{[\Pl]}$ may appear
		only once in the tableau.
	\end{enumerate}
\end{definition}
\begin{figure}[htbp]
\begin{center}
	\begin{adjustbox}{max width=.22\textwidth}\begin{tikzpicture}
		[scale=.9, thick]
		\def\h{.8}
		\def\epp{0.065}
		\draw (0,\h) -- (5,\h);
		\draw (0,0) -- (5,0);
		\draw (0,-\h) -- (5,-\h);
		\draw (0,-2*\h) -- (2,-2*\h);
		\draw (0,-3*\h) -- (2,-3*\h);
		\draw (0,-4*\h) -- (1,-4*\h);
		\draw (0,\h) -- (0,-4*\h);
		\draw (1,\h) -- (1,-4*\h);
		\draw (2,\h) -- (2,-3*\h);
		\draw (3,\h) -- (3,-1*\h);
		\draw (4,\h) -- (4,-1*\h);
		\draw (5,\h) -- (5,-1*\h);
		\node at (1/2,\h/2) {$2$};
		\node at (3/2,\h/2) {$2$};
		\node at (5/2,\h/2) {$3$};
		\node at (7/2,\h/2) {$5$};
		\node at (9/2,\h/2) {$0.1$};
		\node at (1/2,-\h/2) {$3$};
		\node at (3/2,-\h/2) {$5$};
		\node at (5/2,-\h/2) {$5$};
		\node at (7/2,-\h/2+\epp) {$\hat 7$};
		\node at (9/2,-\h/2) {$0.34$};
		\node at (1/2,-3*\h/2) {$5$};
		\node at (3/2,-3*\h/2+\epp) {$\hat 5$};
		\node at (1/2,-5*\h/2+\epp) {$\hat 2$};
		\node at (3/2,-5*\h/2+\epp) {$\hat 5$};
		\node at (1/2,-7*\h/2+\epp) {$\hat 2$};
	\end{tikzpicture}
	\end{adjustbox}
\end{center}
\caption{An $\A$-tableau of shape
$\la=(5,5,2,2,1)$.}
\label{fig:Atableaux}
\end{figure}

The equivalence between $\A$-tableaux and bounded towers of Young diagrams
is given by
\begin{align*}
	\mbox{Young diagram $\la^{(a)}$}
	\ =\
	\mbox{Shape inside the $\A$-tableau occupied by
	all letters $b\in\A$, $b\le a$},
\end{align*}
and the maximal diagram $\la^{\max}$ 
of a tower $\boldsymbol\la$
is the same as the shape of the $\A$-tableau.

\begin{remark}\label{rmk:SSYT_SYT}
	Note that if an $\A$-tableau contains only usual letters,
	then it is simply a semistandard Young tableau, cf. \cite[Ch. 7]{Stanley1999}.
	If it contains only dual letters,
	then it is a transpose
	of a semistandard Young tableau. Finally,
	if it contains only Plancherel letters, then it is a standard
	Young tableau.
\end{remark}

\begin{definition}[Interlacing arrays]
\label{def:interlacing_config}
	$\A$-tableaux (equivalently, bounded towers of Young diagrams)
	can be represented
	as interlacing particle arrays as on Fig.~\ref{fig:A_particle} (see also
	\S \ref{sub:adding_a_usual_variable}).
	Such configurations consist of two parts, the lower $\boldsymbol\la^{[\ab]}$ and
	the upper $\boldsymbol\la^{[\boldsymbol\be\cup\Pl]}$.
	Levels of $\boldsymbol\la^{[\ab]}$ are indexed by usual letters
	$\in\A^{[\ab]}$
	which enter the corresponding
	$\A$-tableau, and levels of $\boldsymbol\la^{[\boldsymbol\be\cup\Pl]}$
	are indexed by dual and Plancherel letters entering this $\A$-tableau.
	In $\boldsymbol\la^{[\ab]}$, Young diagrams on consecutive levels differ by a \emph{horizontal} strip,
	and particles correspond to \emph{row} lengths of these Young diagrams.
	In $\boldsymbol\la^{[\boldsymbol\be\cup\Pl]}$,
	consecutive Young diagrams differ by a \emph{vertical} strip, and
	particles correspond to their \emph{column} lengths.
	We also add the transpose of the uppermost diagram from $\boldsymbol\la^{[\ab]}$
	(on Fig.~\ref{fig:A_particle} this is $(\la^{(5)})'$)
	to $\boldsymbol\la^{[\boldsymbol\be\cup\Pl]}$ because the next diagram
	is obtained from it by adding a vertical strip.
	Let, by agreement, the number of particles increase by one
	from level to level inside each of the parts
	$\boldsymbol\la^{[\ab]}$ and $\boldsymbol\la^{[\boldsymbol\be\cup\Pl]}$.
	This can always be achieved by appending sequences of row/column lengths
	by zeroes.
\end{definition}

\begin{figure}[htbp]
\begin{center}
	\begin{tikzpicture}
		[scale=.9, thick]
		\def\y{.7}
		\foreach \level in {0,1,2,3,4,5,6,7,8}
		{
			\draw[->] (-.5,\y*\level) -- (6.5,\y*\level);
		}		
	    \foreach \hh in {0,1,2,3,4,5,6}
	    {
	    	\draw[densely dotted, thick, opacity=.4] (\hh,8*\y+.22) -- (\hh,-.22);
	    	\node at (\hh,-.42) {\hh};
	    }
	    \def\sp{.2};
	    \def\opac{.65}
	    \def\w{1.1}
	    \def\e{0.03}
	 	\draw[line width=\w, color=blue, opacity=\opac] (3,\y-\e)--(2,0+\e)--(1,\y-\e);
	 	\draw[line width=\w, color=red, opacity=\opac] (4,2*\y-\e)--(3,\y+\e)--(3,2*\y-\e)--(1,\y+\e)--(1,2*\y-\e);
	 	\draw[line width=\w, color=blue, opacity=\opac]
	 	(5,4*\y-\e)--++(-2,-\y+2*\e)--(2+\sp/2,4*\y-\e)--++(0,-\y+2*\e)--(2-\sp/2,4*\y-\e)--++(0,-\y+2*\e)--(1,4*\y-\e)
	 	--++(0,-\y+2*\e)--(0,4*\y-\e);
	 	\draw[line width=\w, color=red, opacity=\opac]
	 	(5,5*\y-\e)--++(0,-\y+2*\e)--(4,5*\y-\e)--++(-2+\sp/2,-\y+2*\e)--(2,5*\y-\e)--++(-\sp/2,-\y+2*\e)--(1,5*\y-\e)
	 	--++(0,-\y+2*\e)--(0+\sp/2,5*\y-\e)--++(-\sp/2,-\y+2*\e)--(0-\sp/2,5*\y-\e);
	 	\draw[line width=\w, color=blue, opacity=\opac]
		(5,6*\y-\e)--++(0,-\y+2*\e)--(4,6*\y-\e)--++(0,-\y+2*\e)--(2+\sp/2,6*\y-\e)--++(-\sp/2,-\y+2*\e)--
		(2-\sp/2,6*\y-\e)--++(-1+\sp/2,-\y+2*\e)--(0+\sp,6*\y-\e)--++(-\sp/2,-\y+2*\e)--(0,6*\y-\e)
		--++(-\sp/2,-\y+2*\e)--(0-\sp,6*\y-\e);
	 	\draw[line width=\w, color=red, opacity=\opac]
		(5,7*\y-\e)--++(0,-\y+2*\e)--(4,7*\y-\e)--++(0,-\y+2*\e)--(2+\sp/2,7*\y-\e)--++(0,-\y+2*\e)--
		(2-\sp/2,7*\y-\e)--++(0,-\y+2*\e)--(1,7*\y-\e)--++(-1+\sp,-\y+2*\e)--
		(0+\sp,7*\y-\e)--++(-\sp,-\y+2*\e)--(0,7*\y-\e)--++(-\sp,-\y+2*\e)--(0-\sp,7*\y-\e);
	 	\draw[line width=\w, color=blue, opacity=\opac]
		(5,8*\y-\e)--++(0,-\y+2*\e)--(4,8*\y-\e)--++(0,-\y+2*\e)--(2+\sp,8*\y-\e)--++(-\sp/2,-\y+2*\e)--(2,8*\y-\e)--++
		(0-\sp/2,-\y+2*\e)--(2-\sp,8*\y-\e)
		--++(-1+1*\sp,-\y+2*\e)--(0+3*\sp/2,8*\y-\e)--++(-\sp/2,-\y+2*\e)--(0+1*\sp/2,8*\y-\e)--++(-\sp/2,-\y+2*\e)--
		(0-1*\sp/2,8*\y-\e)--++(-\sp/2,-\y+2*\e)--(0-3*\sp/2,8*\y-\e);

	    \foreach \pt in
	    {(2,0),(3,\y),(1,\y),
	    (4,2*\y),(3,2*\y),(1,2*\y),
	    (3,3*\y),(2+\sp/2,3*\y),(2-\sp/2,3*\y),(1,3*\y),
	    (5,4*\y),(2-\sp/2,4*\y),(2+\sp/2,4*\y),(1,4*\y),(0,4*\y),
	    (5,5*\y),(4,5*\y),(2,5*\y),(1,5*\y),(0+\sp/2,5*\y),(0-\sp/2,5*\y),
	    (5,6*\y),(4,6*\y),(2+\sp/2,6*\y),(2-\sp/2,6*\y),(0+\sp,6*\y),(0,6*\y),(0-\sp,6*\y),
	    (5,7*\y),(4,7*\y),(2+\sp/2,7*\y),(2-\sp/2,7*\y),(1,7*\y),(0+\sp,7*\y),(0,7*\y),(0-\sp,7*\y),
	    (5,8*\y),(4,8*\y),(2+\sp,8*\y),(2,8*\y),(2-\sp,8*\y),(0+3*\sp/2,8*\y),(0+1*\sp/2,8*\y),(0-1*\sp/2,8*\y),(0-3*\sp/2,8*\y)}
	    {
	    	\draw[fill] \pt circle (2.7pt);
	    }
	    \draw[color=black, dashed, ultra thick] (13.8,\y*2.45)
	    -- (-1.1,\y*2.45)
	    node [left,yshift=1] {transpose};
	    \draw[color=black, dashed, ultra thick] (8.9,-.5*\y) --++ (0,\y*9) ;
	    \node at (7.1,0) {$\la^{(2)}$};
	    \node at (7.1,\y) {$\la^{(3)}$};
	    \node at (7.1,2*\y) {$\la^{(5)}$};
	    \node at (7.2,3*\y) {$(\la^{(5)})'$};
	    \node at (7.2,4*\y) {$(\la^{(\hat 2)})'$};
	    \node at (7.2,5*\y) {$(\la^{(\hat 5)})'$};
	    \node at (7.2,6*\y) {$(\la^{(\hat 7)})'$};
	    \node at (7.3,7*\y) {$(\la^{(0.1)})'$};
	    \node at (7.35,8*\y) {$(\la^{(0.34)})'$};
	    \draw [decorate,decoration={brace,amplitude=5pt,raise=4pt},yshift=0pt] (-.8,-.1) -- (-.8,2*\y+.1)
	    node [midway,xshift=-28] {$\boldsymbol\la^{[\ab]}$};
	    \draw [decorate,decoration={brace,amplitude=5pt,raise=4pt},yshift=0pt] (-.8,3*\y-.1) -- (-.8,8*\y+.1)
	    node [midway,xshift=-28] {$\boldsymbol\la^{[\boldsymbol\be\cup\Pl]}$};
	  	\def\t{.27}
	  	\def\g{12}
	  	\node at (\g,0) {2};
	    \node at (\g+\t,\y) {3};\node at (\g-\t,\y) {1};
	    \node at (\g+2*\t,2*\y) {4};\node at (\g+0*\t,2*\y) {3};\node at (\g-2*\t,2*\y) {1};
	    \node at (\g+3*\t,3*\y) {3};\node at (\g+1*\t,3*\y) {2};\node at (\g-1*\t,3*\y) {2};\node at (\g-3*\t,3*\y) {1};
	    \node at (\g+4*\t,4*\y) {5};\node at (\g+2*\t,4*\y) {2};\node at (\g+0*\t,4*\y) {2};\node at (\g-2*\t,4*\y) {1};\node at (\g-4*\t,4*\y) {0};
	    \node at (\g+5*\t,5*\y) {5};\node at (\g+3*\t,5*\y) {4};\node at (\g+1*\t,5*\y) {2};\node at (\g-1*\t,5*\y) {1};\node at (\g-3*\t,5*\y) {0};\node at (\g-5*\t,5*\y) {0};
	    \node at (\g+6*\t,6*\y) {5};\node at (\g+4*\t,6*\y) {4};\node at (\g+2*\t,6*\y) {2};\node at (\g+0*\t,6*\y) {2};\node at (\g-2*\t,6*\y) {0};\node at (\g-4*\t,6*\y) {0};\node at (\g-6*\t,6*\y) {0};
	    \node at (\g+7*\t,7*\y) {5};\node at (\g+5*\t,7*\y) {4};\node at (\g+3*\t,7*\y) {2};\node at (\g+1*\t,7*\y) {2};\node at (\g-1*\t,7*\y) {1};\node at (\g-3*\t,7*\y) {0};\node at (\g-5*\t,7*\y) {0};\node at (\g-7*\t,7*\y) {0};
	    \node at (\g+8*\t,8*\y) {5};\node at (\g+6*\t,8*\y) {4};\node at (\g+4*\t,8*\y) {2};\node at (\g+2*\t,8*\y) {2};\node at (\g+0*\t,8*\y) {2};\node at (\g-2*\t,8*\y) {0};\node at (\g-4*\t,8*\y) {0};\node at (\g-6*\t,8*\y) {0};\node at (\g-8*\t,8*\y) {0};
	\end{tikzpicture}
\end{center}
\caption{A representation of
the $\A$-tableau on Fig.~\ref{fig:Atableaux}
in terms of interlacing arrays; zigzags indicate interlacing of consecutive
levels (colors are only to better distinguish the levels).
The corresponding array of integers is displayed on the right. Its top
level records column lengths of the Young diagram on 
Fig.~\ref{fig:Atableaux}.}
\label{fig:A_particle}
\end{figure}
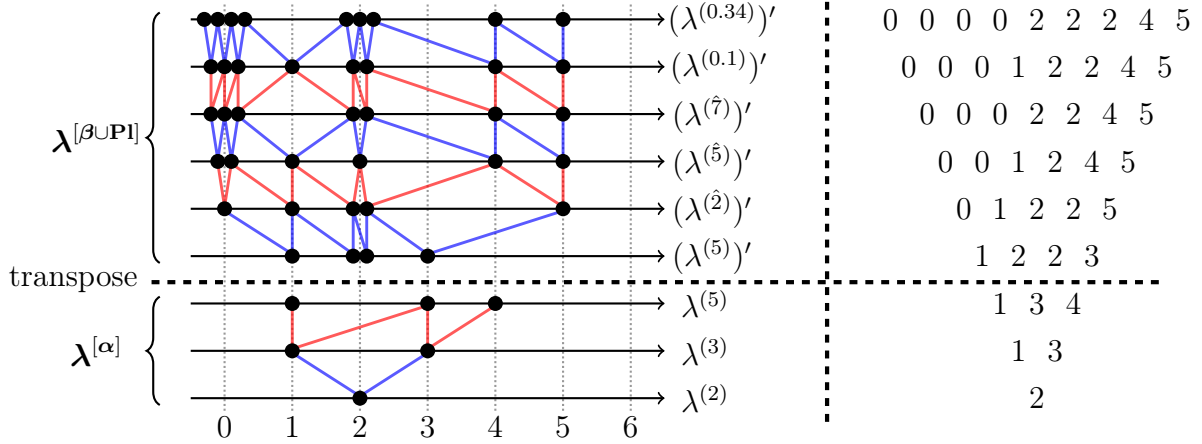


\subsection{Full sampling algorithm for $\HL_n^{\ab;\bb;\Pl_\ga}$} 
\label{sub:full_sampling_algorithm}

We are now in a position to describe
a sampling algorithm
producing the
HL-coherent measures $\HL^{\ab;\bb;\Pl_\ga}_{n}$.
Let $n=1,2,\ldots$ and a specialization $(\ab;\bb;\Pl_\ga)$
be fixed.
According to Remark \ref{rmk:scale_inv_of_coherent},
we may assume that
\begin{align}\label{p1_one}
	p_1(\ab;\bb;\Pl_\ga)=
	\sum_{i\ge1}\al_i+
	\frac{1}{1-t}\bigg(\ga+\sum_{i\ge1}\be_i\bigg)=1.
\end{align}
Let $\mathrm{m}=\mathrm{m}^{\ab;\bb;\Pl_\ga}$ be a probability measure
on the alphabet $\A$
with
\begin{align}\label{measure_m_on_A}
	\mathrm{m}(r)=\al_r,\quad
	\mathrm{m}(\hat r)=\frac{\be_r}{1-t},\qquad r=1,2,3,\ldots,
\end{align}
and such that
on the segment $\A^{[\Pl]}$,
$\mathrm{m}$ reduces
to a (continuous) uniform
measure with $\mathrm{m}\big(\A^{[\Pl]}\big)={\ga}/({1-t})$.
By \eqref{p1_one}, this is indeed a probability measure.

The input of the sampling algorithm
is a random word $w=\xi_1\xi_2 \ldots\xi_n$
of length $n$, where $\xi_i\in\A$ are independent random letters
distributed according to $\mathrm{m}^{\ab;\bb;\Pl_\ga}$.
Our sampling algorithm 
is a combination of the following bivariate
dynamics (which are honest Markov processes by 
Proposition \ref{prop:nonneg}):
\begin{align}
	\left\{\begin{array}{lll}
	(W_{(\al)}^{h=+\infty},V_{(\al)}^{h=+\infty})
	&\mbox{from \S \ref{sub:adding_a_usual_variable} corresponding to a usual letter};
	\\\rule{0pt}{14pt}
	(W_{(\be)}^{h=1},V_{(\be)}^{h=1})&\mbox{from \S \ref{sub:duality} corresponding to a dual letter};
	\\\rule{0pt}{14pt}
	(W_{\circ}^{\be;h=1},V_{\circ}^{\be;h=1})&\mbox{from \S \ref{sub:adding_a_plancherel_parameter}
	corresponding to a Plancherel letter}.
	\end{array}\right.
	\label{WV_choice}
\end{align}
It is helpful to keep this list in mind
while reading the description of the sampling algorithm below.
In particular,
conditional probabilities ``$V$'' above
lead to propagation rules described in step {\rm{}\bf(IV)}.

\begin{sampling}\label{alg:main}
	{\rm{}\bf(I)} The algorithm starts at step $k=0$ 
	with an empty interlacing array $\boldsymbol\la(0)$ containing no particles.

	{\rm{}\bf(II)} At each step $k=1,\ldots,n$,
	the modification of the array $\boldsymbol\la(k-1)\to\boldsymbol\la(k)$ 
	begins at level $\la^{(\xi_k)}$.
	If there is no such level $\la^{(\xi_k)}$ because the letter $\xi_k$ never appeared before
	in the random input, then $\la^{(\xi_k)}(k-1)$ is created by cloning 
	the previous existing level $\la^{(\eta)}$:
	\begin{itemize}
		\item If $k=1$, then let $\la^{(\xi_1)}(0)$ 
		be a single particle at position 0.
		\item For $k\ge2$, let 
		$\eta$ be the maximal letter $\le \xi_k$
		(in terms of the order \eqref{A_order} on $\A$)
		which has already appeared in the input word, and set $\la^{(\xi_k)}(k-1):=\la^{(\eta)}(k-1)$.
		\item After that, add 
		one extra particle at position $0$
		to all existing levels $\la^{(a)}$, $a\ge \xi_k$.
	\end{itemize}
	
	{\rm{}\bf(III)} The modification
	$\la^{(\xi_k)}(k-1)\to\la^{(\xi_k)}(k)$
	consists in 
	the leftmost free (respectively, the rightmost) 
	particle at level $\xi_k$
	moving to the right by one
	if $\xi_k\in\A^{[\ab]}$ (respectively, 
	$\xi_k\in\A^{[\boldsymbol\be]}\cup\A^{[\Pl]}$).

	{\rm{}\bf(IV)} Further modifications of the interlacing array at the same step 
	$k$ consist in propagation of moves through all existing levels 
	$>\xi_k$ (with respect to order \eqref{A_order}). 
	This propagation is performed inductively, from the level immediately above $\xi_k$
	upwards. Let us fix any two existing levels $d>c\ge\xi_k$ of the interlacing array 
	such that there are no other existing
	levels between them. 
	\begin{itemize}
		\item If $d\in\A^{[\ab]}$,
		we argue in terms of row lengths.
		Suppose that a particle $\la^{(c)}_{j}$
		has moved at level 
		$c$. Then it 
		pushes its first upper right neighbor
		$\la^{(d)}_{\nt(j)}$
		at level $d$ which is free to move with 
		probability $\rp_j^{+\infty}(\la^{(c)}(k),\la^{(d)}(k-1)\mid 0,t)$,
		or 
		pulls its immediate upper left neighbor
		$\la^{(d)}_{j+1}$ with the complementary probability 
		$1-\rp_j^{+\infty}(\la^{(c)}(k),\la^{(d)}(k-1)\mid 0,t)$, 
		where $\rp_j^{+\infty}(\cdot,\cdot\mid 0,t)$ is given in Proposition \ref{prop:rj_0t}.

		\item If $d\in\A^{[\bb]}\cup\A^{[\Pl]}$, 
		we will argue in terms of column lengths (this is possible even if 
		$c\in\A^{[\ab]}$, cf. the ``transpose'' line on Fig.~\ref{fig:A_particle}).
		Suppose that a particle $(\la^{(c)})'_{j}$
		has moved at level $c$. 
		Then it 
		pushes its immediate upper right neighbor
		$(\la^{(d)})'_{j}$
		at level $d$ with 
		probability 
		$\rp_j^{1}\big((\la^{(c)}(k))',(\la^{(d)}(k-1))'\mid 0,t\big)$,
		or 
		pulls its immediate upper left neighbor
		$(\la^{(d)})'_{j+1}$ with the complementary probability 
		$1-\rp_j^{1}\big((\la^{(c)}(k))',(\la^{(d)}(k-1))'\mid 0,t\big)$, 
		where $\rp_j^{1}(\cdot,\cdot\mid t,0)$ 
		is given in Proposition \ref{prop:rj_t0}.
	\end{itemize}

	{\rm{}\bf(V)} Finally, existing levels $\la^{(a)}(k-1)$ for $a<\xi_k$ (in order \eqref{A_order})
	are not modified at step $k$.
\end{sampling}
In {\rm\bf(III)}--{\rm\bf(IV)}, if the moving particle is blocked from below, then its move is donated to the
first free particle to the right of itself (see~\S \ref{sub:adding_a_usual_variable} and
Fig.~\ref{fig:rl} in particular). The mandatory short-range pushing mechanism
\eqref{short_range_alpha} is also present in all moves in {\rm\bf(IV)}.
Note also that the cloning operation in {\rm{}\bf(II)}
corresponds to Lemma \ref{lemma:squashed}. 

Since all probabilities involved in the description of Algorithm \ref{alg:main}
are nonnegative (by Proposition \ref{prop:nonneg}), 
this is indeed an honest randomized sampling procedure.

\begin{theorem}\label{thm:it_samples}
	Sampling Algorithm \ref{alg:main}
	indeed produces the HL-coherent measure
	$\HL^{\ab;\bb;\Pl_\ga}_{n}$
	(as
	the distribution of the maximal Young diagram
	$\la^{\max}(n)$ at step $n$).
\end{theorem}
\begin{proof}
	For finitely many nonzero $\al_i$'s and
	$\be_j$'s in the specialization $(\ab;\bb;\Pl_\ga)$
	this readily follows
	from
	Algorithm~\ref{alg:sampling_bi}
	together with
	properties
	of bivariate
	dynamics (see \S \ref{sub:bivariate_dynamics} and \S \ref{sec:three_particular_dynamics_on_macdonald_processes}).

	When at least one of the sequences
	$\{\al_i\}$ or $\{\be_j\}$ is infinite,
	the desired result can be derived
	via an ``algebraic'' limit transition.
	Observe that
	for each $\la\in\Yb_n$,
	the probability weight $\HL^{\ab;\bb;\Pl_\ga}_{n}(\la)$
	\eqref{coherent_def}
	is a
	power series
	(with bounded degrees of monomials)
	which is symmetric
	in two infinite sets of variables $\{\al_i\}_{i=1}^{\infty}$
	and $\{\be_j\}_{j=1}^{\infty}$.
	Indeed, this is because the Hall--Littlewood symmetric functions $P_\la$ are polynomials
	in the Newton power sums $p_1,p_2,\ldots$, and the specialization
	$(\ab;\bb;\Pl_\ga)$ reduces to the mapping $p_1\to 1$, and
	$p_k\to p_k(\ab;\bb;\Pl_\ga)$ for $k\ge 2$ (see \eqref{al_be_ga_Newton} with $q=0$).
	
	Thus, similarly to Remark \ref{rmk:Sym_proj_lim},
	one can also view $\HL^{\ab;\bb;\Pl_\ga}_{n}(\la)$
	as a ``projective limit'' of a sequence of weights
	\begin{align*}
		\HL^{[N]}_{n}(\la)=
		{n!}\,
		P_\la(\ab^{[N]};\bb^{[N]};\Pl_\ga\,|\,0,t)
		Q_\la(\Pl_{1}\,|\,0,t),
		\qquad
		N\to\infty,
	\end{align*}
	where $\ab^{[N]}=(\al_1,\ldots,\al_N)$ and
	$\boldsymbol\be^{[N]}=(\be_1,\ldots,\be_N)$
	are the truncated sequences of parameters.
	In other words, $\HL^{[N]}_{n}(\la)$
	is the restriction of the
	desired measure $\HL^{\ab;\bb;\Pl_\ga}_{n}(\la)$
	to realizations of the randomized
	Algorithm \ref{alg:main}
	in which there are no letters $\xi_i$ equal to $\al_{N+1},\al_{N+2},\ldots$
	or
	$\be_{N+1},\be_{N+2},\ldots$
	(note that the weights $\HL^{[N]}_{n}(\la)$ do
	not sum to one).

	One can readily check that our Algorithm
	\ref{alg:main} is compatible with that ``projective limit''
	as well. This completes the proof
	for all specializations
	$(\ab;\bb;\Pl_\ga)$.
\end{proof}

Let us make several comments
on Sampling Algorithm \ref{alg:main}.

\begin{remark} 
	It is possible to reword
	the above description
	purely in terms of
	operations on $\A$-tableaux.
	Namely, at each step $k$
	the new letter $\xi_k\in\A$ distributed according to
	the measure \eqref{measure_m_on_A} is \emph{inserted}
	into the existing $\A$-tableau. The insertion always starts from the
	first column of the tableau.
	All following elementary \emph{bumping/shifting} operations
	in the course of the insertion
	are performed at random
	according to probabilities in step {\rm\bf(IV)} of Algorithm \ref{alg:main}
	(this also depends on the
	type of the ketter $\xi_k$).
	See also \cite[\S7.2]{BorodinPetrov2013NN} for more detail
	on how the setting of 
	interlacing arrays
	can be translated into the language of Young tableaux.
\end{remark}

\begin{remark}\label{rmk:ordering}
	As can be readily seen from the construction of the full sampling algorithm,
	other choices
	of a linear order on $\A$
	lead to other (different) algorithms
	which, however, sample \emph{the same} HL-coherent measure.
	Equivalently, permuting the probability
	weights assigned by the measure $\mathrm{m}$
	(e.g., setting $\mathrm{m}^{\ab;\bb;\Pl_\ga}(i)=\al_{\sigma(i)}$
	for any permutation $\sigma$ of
	natural numbers) also amounts to sampling of the same HL-coherent measure.
	We will utilize such different orderings below
	in \S \ref{sec:proof_of_the_law_of_large_numbers}.
\end{remark}

\begin{remark}\label{rmk:poiss_alg}
	Notice that if letters $\xi_i$ in the random input word
	$w$ appear in continuous time according to certain independent
	Poisson clocks,
	then after time $\tau\ge0$,
	Algorithm~\ref{alg:main} produces
	the poissonized version of
	the HL-coherent measure, i.e., the Macdonald measure
	$\MM_\tau^{\ab;\bb;\Pl_\ga}$
	(\S \ref{sub:poissonization_and_macdonald_measures})
	with $q=0$.
	More precisely, letters $r$ and $\hat r$ ($r=1,2,\ldots$)
	should appear at rates
	$\al_r$ and $\frac{\be_r }{1-t}$, respectively,
	and letters from $\A^{[\Pl]}$
	should appear at total rate $\frac{\ga }{1-t}$.
	Note that each letter from $\A^{[\Pl]}$
	is almost surely new, i.e., it has not appeared in the word
	before (cf. \eqref{measure_m_on_A}).
\end{remark}

\begin{remark}\label{rmk:coupling}
	In the course of the sampling algorithm,
	we construct a sequence of
	random Young diagrams $\la(n)$,
	$|\la(n)|=n$, on the same probability space.
	The marginal distribution of
	each $\la(n)$ is $\HL^{\ab;\bb;\Pl_\ga}_{n}$.
	It can be readily seen that the
	joint distribution of this sequence
	of random Young diagrams reflects the coherency
	property of the measures $\HL^{\ab;\bb;\Pl_\ga}_{n}$ (see \S \ref{sub:plancherel_specializations_and_young_graph}).
	When $t=\q^{-1}=p^{-d}$ is the inverse of a prime power, this
	``big'' probability space corresponds
	to the distribution of the infinite
	uni-uppertriangular random matrix
	over the finite field $F_\q$,
	see \S \ref{sub:infinite_random_matrices_over_a_finite_field}
	and Remark \ref{rmk:coupling_intro} in particular.
\end{remark}

\begin{remark}
	In the
	Schur case, i.e., for $t=0$, all
	randomized insertion
	steps become \emph{deterministic} (see also \cite[\S7]{BorodinPetrov2013NN} for more discussion),
	and one arrives at (a column version of) the generalization of the
	Robinson--Schensted--Knuth (RSK)
	insertion introduced
	in \cite{Vershik1986} (and also employed in, e.g., \cite{BufetovCLT}, \cite{Sniady2013}).
	In fact, it is possible to obtain the algorithm of \cite{Vershik1986}
	itself by taking a suitable order on $\A$ (cf. \eqref{A_order}),
	and also suitable other indices $h$ in
	\eqref{WV_choice}. 
\end{remark}


\subsection{An example of the randomized insertion} 
\label{sub:an_example}

Let us demonstrate the randomized
insertion of the letter $\xi=3$ into
the $\A$-tableau on Fig.~\ref{fig:Atableaux}.
We will use integer arrays as on Fig.~\ref{fig:A_particle}, right.
The letter $3$ is already present in the tableau, 
so no cloning of levels will occur.
The initial insertion means that the leftmost particle at level $\la^{(3)}$ 
jumps to the right by one (observe that this particle is free to 
jump, so no move donation occurs).
Then, with probability
$\rp_2^{+\infty}\big(
(3,2),(4,3,1)\,|\,0,t\big)=1-t$, this particle pushes the first free particle to the right on the next level.
With the complementary probability $t$ it pulls its upper left neighbor. Thus, after propagation of the move
to
the level $\la^{(5)}$, the random configuration on levels $\la^{(2)}\prech\la^{(3)}\prech\la^{(5)}$
looks as follows (here and below ``w.p.'' stands for ``with probability''):
\begin{align}
\framebox{\scalebox{.7}{
\begin{tikzpicture}
[scale=.9, thick]
\def\y{.7}
\node at (7.1,0) {$\la^{(2)}$};
\node at (7.1,\y) {$\la^{(3)}$};
\node at (7.1,2*\y) {$\la^{(5)}$};
	\def\t{.27}
	\def\g{9}
	\node at (\g,0) {2};
\node at (\g+\t,\y) {3};\node at (\g-\t,\y) {\framebox{2}};
\node at (\g+2*\t,2*\y) {\framebox{5}};\node at (\g+0*\t,2*\y) {3};\node at (\g-2*\t,2*\y) {1};
\end{tikzpicture}}}
\raisebox{20pt}{\quad\mbox{w.p. $1-t$,}}\qquad \qquad
\framebox{\scalebox{.7}{
\begin{tikzpicture}
[scale=.9, thick]
\def\y{.7}
\node at (7.1,0) {$\la^{(2)}$};
\node at (7.1,\y) {$\la^{(3)}$};
\node at (7.1,2*\y) {$\la^{(5)}$};
	\def\t{.27}
	\def\g{9}
	\node at (\g,0) {2};
\node at (\g+\t,\y) {3};\node at (\g-\t,\y) {\framebox{2}};
\node at (\g+2*\t,2*\y) {{4}};\node at (\g+0*\t,2*\y) {3};\node at (\g-2*\t,2*\y) {\framebox{2}};
\end{tikzpicture}}}
\raisebox{20pt}{\quad\mbox{w.p. $t$.}}
\label{alg_example_1}
\end{align}

In the first case, the propagation to several next levels is deterministic due to the short-range
pushing, cf. \eqref{short_range_alpha}.
Note that here we had to add one more (leftmost) particle to each level in the upper part
$\boldsymbol\la^{[\boldsymbol\be\cup\Pl]}$.
The resulting distribution in the first case looks as follows:
\begin{align}
	\framebox{\scalebox{.7}{
	\begin{tikzpicture}
	[scale=.9, thick]
	\def\y{.7}
	\draw[color=black, dashed, thick] (13.8,\y*2.5)
	-- (8.6,\y*2.5);
		\def\t{.27}
		\def\g{11.5}
		\node at (\g-\t,0) {2};
	\node at (\g,\y) {3};\node at (\g-2*\t,\y) {\framebox{2}};
	\node at (\g+1*\t,2*\y) {\framebox{5}};\node at (\g-1*\t,2*\y) {3};\node at (\g-3*\t,2*\y) {1};
	\node at (\g+3*\t,3*\y) {3};\node at (\g+1*\t,3*\y) {2};\node at (\g-1*\t,3*\y) {2};\node at (\g-3*\t,3*\y) {1};
	\node at (\g+4*\t,4*\y) {5};\node at (\g+2*\t,4*\y) {2};\node at (\g+0*\t,4*\y) {2};\node at (\g-2*\t,4*\y) {1};\node at (\g-4*\t,4*\y) {\framebox{1}};
	\node at (\g+5*\t,5*\y) {5};\node at (\g+3*\t,5*\y) {4};\node at (\g+1*\t,5*\y) {2};\node at (\g-1*\t,5*\y) {1};\node at (\g-3*\t,5*\y) {\framebox{1}};\node at (\g-5*\t,5*\y) {0};
	\node at (\g+6*\t,6*\y) {5};\node at (\g+4*\t,6*\y) {4};\node at (\g+2*\t,6*\y) {2};\node at (\g+0*\t,6*\y) {2};\node at (\g-2*\t,6*\y) {\framebox{1}};\node at (\g-4*\t,6*\y) {0};\node at (\g-6*\t,6*\y) {0};
	\node at (\g+7*\t,7*\y) {5};\node at (\g+5*\t,7*\y) {4};\node at (\g+3*\t,7*\y) {2};\node at (\g+1*\t,7*\y) {2};\node at (\g-1*\t,7*\y) {1};\node at (\g-3*\t,7*\y) {\framebox{1}};\node at (\g-5*\t,7*\y) {0};\node at (\g-7*\t,7*\y) {0};
	\node at (\g+8*\t,8*\y) {5};\node at (\g+6*\t,8*\y) {4};\node at (\g+4*\t,8*\y) {2};\node at (\g+2*\t,8*\y) {2};\node at (\g+0*\t,8*\y) {2};\node at (\g-2*\t,8*\y) {\framebox{1}};\node at (\g-4*\t,8*\y) {0};\node at (\g-6*\t,8*\y) {0};\node at (\g-8*\t,8*\y) {0};
	\node at (\g-5*\t, 3*\y) {\framebox{1}};
	\node at (\g-6*\t, 4*\y) {{0}};
	\node at (\g-7*\t, 5*\y) {{0}};
	\node at (\g-8*\t, 6*\y) {{0}};
	\node at (\g-9*\t, 7*\y) {{0}};
	\node at (\g-10*\t, 8*\y) {{0}};
	\end{tikzpicture}}}
	\raisebox{50pt}{\quad\mbox{w.p. $(1-t)\dfrac1{1+t}$,}}
	\qquad \quad
	\framebox{\scalebox{.7}{
	\begin{tikzpicture}
	[scale=.9, thick]
	\def\y{.7}
	\draw[color=black, dashed, thick] (13.8,\y*2.5)
	-- (8.6,\y*2.5);
		\def\t{.27}
		\def\g{11.5}
		\node at (\g-\t,0) {2};
	\node at (\g,\y) {3};\node at (\g-2*\t,\y) {\framebox{2}};
	\node at (\g+1*\t,2*\y) {\framebox{5}};\node at (\g-1*\t,2*\y) {3};\node at (\g-3*\t,2*\y) {1};
	\node at (\g+3*\t,3*\y) {3};\node at (\g+1*\t,3*\y) {2};\node at (\g-1*\t,3*\y) {2};\node at (\g-3*\t,3*\y) {1};
	\node at (\g+4*\t,4*\y) {5};\node at (\g+2*\t,4*\y) {2};\node at (\g+0*\t,4*\y) {2};\node at (\g-2*\t,4*\y) {1};\node at (\g-4*\t,4*\y) {\framebox{1}};
	\node at (\g+5*\t,5*\y) {5};\node at (\g+3*\t,5*\y) {4};\node at (\g+1*\t,5*\y) {2};\node at (\g-1*\t,5*\y) {1};\node at (\g-3*\t,5*\y) {\framebox{1}};\node at (\g-5*\t,5*\y) {0};
	\node at (\g+6*\t,6*\y) {5};\node at (\g+4*\t,6*\y) {4};\node at (\g+2*\t,6*\y) {2};\node at (\g+0*\t,6*\y) {2};\node at (\g-2*\t,6*\y) {\framebox{1}};\node at (\g-4*\t,6*\y) {0};\node at (\g-6*\t,6*\y) {0};
	\node at (\g+7*\t,7*\y) {5};\node at (\g+5*\t,7*\y) {4};\node at (\g+3*\t,7*\y) {2};\node at (\g+1*\t,7*\y) {2};\node at (\g-1*\t,7*\y) {\framebox{2}};\node at (\g-3*\t,7*\y) {0};\node at (\g-5*\t,7*\y) {0};\node at (\g-7*\t,7*\y) {0};
	\node at (\g+8*\t,8*\y) {5};\node at (\g+6*\t,8*\y) {4};\node at (\g+4*\t,8*\y) {2};\node at (\g+2*\t,8*\y) {2};\node at (\g+0*\t,8*\y) {2};\node at (\g-2*\t,8*\y) {\framebox{1}};\node at (\g-4*\t,8*\y) {0};\node at (\g-6*\t,8*\y) {0};\node at (\g-8*\t,8*\y) {0};
	\node at (\g-5*\t, 3*\y) {\framebox{1}};
	\node at (\g-6*\t, 4*\y) {{0}};
	\node at (\g-7*\t, 5*\y) {{0}};
	\node at (\g-8*\t, 6*\y) {{0}};
	\node at (\g-9*\t, 7*\y) {{0}};
	\node at (\g-10*\t, 8*\y) {{0}};
	\end{tikzpicture}}}
	\raisebox{50pt}{\quad\mbox{w.p. $(1-t)\dfrac{t}{1+t}$.}}
	\label{alg_example_2}
\end{align}
Indeed, this is because
$\rp_5^{1}\big(
(5,4,2,2,1,0,0,0),(5,4,2,2,1,0,0,0,0)\,|\,t,0\big)=\frac{t}{1+t}$
(propagation from $(\la^{(\hat 7)})'$ to $(\la^{(0.1)})'$, cf. Fig.~\ref{fig:A_particle}), and
$\rp_5^{1}\big(
(5,4,2,2,2,0,0,0,0),(5,4,2,2,2,0,0,0,0,0)\,|\,t,0\big)=0$
(propagation from $(\la^{(0.1)})'$ to $(\la^{(0.34)})'$).
\begin{remark}\label{rmk:no_donations}
	In fact, in the upper part
	$\boldsymbol\la^{[\boldsymbol\be\cup\Pl]}$
	of the array
	there are no donations of jumps or pushes, see \cite[\S8.2.1]{BorodinPetrov2013NN}.
\end{remark}

In the second case, the resulting distribution looks as follows:
\begin{align}
	\framebox{\scalebox{.7}{
	\begin{tikzpicture}
	[scale=.9, thick]
	\def\y{.7}
	\draw[color=black, dashed, thick] (13.8,\y*2.5)
	-- (9.2,\y*2.5);
	\def\t{.27}
	\def\g{11.5}
	\node at (\g,0) {2};
	\node at (\g+\t,\y) {3};\node at (\g-\t,\y) {\framebox{2}};
	\node at (\g+2*\t,2*\y) {4};\node at (\g-0*\t,2*\y) {3};\node at (\g-2*\t,2*\y) {\framebox{2}};
	\node at (\g+3*\t,3*\y) {3};\node at (\g+1*\t,3*\y) {\framebox{3}};\node at (\g-1*\t,3*\y) {2};\node at (\g-3*\t,3*\y) {1};
	\node at (\g+4*\t,4*\y) {5};\node at (\g+2*\t,4*\y) {\framebox{3}};\node at (\g+0*\t,4*\y) {2};\node at (\g-2*\t,4*\y) {1};\node at (\g-4*\t,4*\y) {0};
	\node at (\g+5*\t,5*\y) {5};\node at (\g+3*\t,5*\y) {4};\node at (\g+1*\t,5*\y) {\framebox{3}};\node at (\g-1*\t,5*\y) {1};\node at (\g-3*\t,5*\y) {0};\node at (\g-5*\t,5*\y) {0};
	\node at (\g+6*\t,6*\y) {5};\node at (\g+4*\t,6*\y) {4};\node at (\g+2*\t,6*\y) {\framebox{3}};\node at (\g+0*\t,6*\y) {2};\node at (\g-2*\t,6*\y) {0};\node at (\g-4*\t,6*\y) {0};\node at (\g-6*\t,6*\y) {0};
	\node at (\g+7*\t,7*\y) {5};\node at (\g+5*\t,7*\y) {4};\node at (\g+3*\t,7*\y) {\framebox{3}};\node at (\g+1*\t,7*\y) {2};\node at (\g-1*\t,7*\y) {0};\node at (\g-3*\t,7*\y) {0};\node at (\g-5*\t,7*\y) {0};\node at (\g-7*\t,7*\y) {0};
	\node at (\g+8*\t,8*\y) {5};\node at (\g+6*\t,8*\y) {4};\node at (\g+4*\t,8*\y) {\framebox{3}};\node at (\g+2*\t,8*\y) {2};\node at (\g+0*\t,8*\y) {2};\node at (\g-2*\t,8*\y) {0};\node at (\g-4*\t,8*\y) {0};\node at (\g-6*\t,8*\y) {0};\node at (\g-8*\t,8*\y) {0};
	\end{tikzpicture}}}
	\raisebox{50pt}{\quad\mbox{w.p. $t\,\dfrac{1+t}{1+t+t^2}$,}}
	\qquad \qquad \qquad
	\framebox{\scalebox{.7}{
	\begin{tikzpicture}
	[scale=.9, thick]
	\def\y{.7}
	\draw[color=black, dashed, thick] (13.8,\y*2.5)
	-- (9.2,\y*2.5);
	\def\t{.27}
	\def\g{11.5}
	\node at (\g,0) {2};
	\node at (\g+\t,\y) {3};\node at (\g-\t,\y) {\framebox{2}};
	\node at (\g+2*\t,2*\y) {4};\node at (\g-0*\t,2*\y) {3};\node at (\g-2*\t,2*\y) {\framebox{2}};
	\node at (\g+3*\t,3*\y) {3};\node at (\g+1*\t,3*\y) {\framebox{3}};\node at (\g-1*\t,3*\y) {2};\node at (\g-3*\t,3*\y) {1};
	\node at (\g+4*\t,4*\y) {5};\node at (\g+2*\t,4*\y) {\framebox{3}};\node at (\g+0*\t,4*\y) {2};\node at (\g-2*\t,4*\y) {1};\node at (\g-4*\t,4*\y) {0};
	\node at (\g+5*\t,5*\y) {5};\node at (\g+3*\t,5*\y) {\framebox{5}};\node at (\g+1*\t,5*\y) {2};\node at (\g-1*\t,5*\y) {1};\node at (\g-3*\t,5*\y) {0};\node at (\g-5*\t,5*\y) {0};
	\node at (\g+6*\t,6*\y) {5};\node at (\g+4*\t,6*\y) {\framebox{5}};\node at (\g+2*\t,6*\y) {2};\node at (\g+0*\t,6*\y) {2};\node at (\g-2*\t,6*\y) {0};\node at (\g-4*\t,6*\y) {0};\node at (\g-6*\t,6*\y) {0};
	\node at (\g+7*\t,7*\y) {5};\node at (\g+5*\t,7*\y) {\framebox{5}};\node at (\g+3*\t,7*\y) {2};\node at (\g+1*\t,7*\y) {2};\node at (\g-1*\t,7*\y) {0};\node at (\g-3*\t,7*\y) {0};\node at (\g-5*\t,7*\y) {0};\node at (\g-7*\t,7*\y) {0};
	\node at (\g+8*\t,8*\y) {5};\node at (\g+6*\t,8*\y) {\framebox{5}};\node at (\g+4*\t,8*\y) {2};\node at (\g+2*\t,8*\y) {2};\node at (\g+0*\t,8*\y) {2};\node at (\g-2*\t,8*\y) {0};\node at (\g-4*\t,8*\y) {0};\node at (\g-6*\t,8*\y) {0};\node at (\g-8*\t,8*\y) {0};
	\end{tikzpicture}}}
	\raisebox{50pt}{\quad\mbox{w.p. $t\,\dfrac{t^2}{1+t+t^2}$.}}
	\label{alg_example_3}
\end{align}
Indeed, one can check that
$\rp_2^{1}\big(
(5, 3, 2, 1, 0),(5, 4, 2, 1, 0, 0)\,|\,t,0\big)=\frac{t^{2}}{1+t+t^{2}}$
(propagation from $(\la^{(\hat 2)})'$ to $(\la^{(\hat 5)})'$, cf. Fig.~\ref{fig:A_particle}),
and propagation to all further levels is deterministic.

Thus, we see that the result of the insertion is a random $\A$-tableau described in
\eqref{alg_example_2}--\eqref{alg_example_3}.




\section{Proof of the Law of Large Numbers} 
\label{sec:proof_of_the_law_of_large_numbers}

This section is devoted to the proof of the main result of the present paper (see \S \ref{sub:law_of_large_numbers}
in the Introduction).
Throughout the section we assume that the Macdonald
parameter $q$ is set to zero.
Fix two sequences $\ab=(\al_1,\al_2,\ldots)$ and $\boldsymbol\be=(\be_1,\be_2,\ldots)$ such that
$p_1(\ab;\bb)=\sum_{i=1}^{\infty}\al_i+\frac{1}{1-t}\sum_{i=1}^{\infty}\be_i=1$
(cf. \eqref{p1_one}). That is, we will consider
a HL-nonnegative specialization
$(\ab;\bb;\Pl_\ga)=(\ab;\bb;\Pl_0)$ with the Plancherel parameter $\gamma$ set to zero.

\begin{theorem}\label{thm:main_s7}
	Let for each $n\ge1$, $\la(n)$ be a random Young diagram
	with $n$ boxes distributed according to the
	HL-coherent measure $\HL_n^{\ab;\bb;\Pl_0}$
	(defined in \S \ref{sub:_q_t_coherent_measures};
	we assume that the parameters satisfy
	\eqref{p1_one}).
	The following Law of Large Numbers holds:
	\begin{align}\label{LLN_s7}
		\frac{\la_i(n)}{n}\to\al_i,\qquad
		\frac{\la_i'(n)}{n}\to \frac{\be_i}{1-t},\qquad n\to \infty,
	\end{align}
	where $\la_i'$ are the column lengths of $\la$ (i.e., row lengths of the
	transposed diagram $\la$).

	The convergence in \eqref{LLN_s7} is almost sure
	with respect
	to the probability space
	carrying all random Young diagrams
	$\la(n)$, $n=1,2,3,\ldots$
	(cf.~Remarks \ref{rmk:coupling} and
	\ref{rmk:coupling_intro}).
\end{theorem}

See \S \ref{sub:law_of_large_numbers} in the Introduction
for further conjectures related to Theorem \ref{thm:main_s7}.

\subsection{Strategy of the proof} 
\label{sub:strategy_of_the_proof}

Let us briefly outline our strategy of the proof of
Theorem~\ref{thm:main_s7}.
Assume first that both sequences $\ab$ and $\boldsymbol\be$
are finite,
\begin{align*}
	\ab=\{\al_1\ge\al_2\ge \ldots\ge\al_\anum\},\qquad
	\boldsymbol\be=\{\be_1\ge\be_2\ge \ldots\ge\be_\bnum\}.
\end{align*}
Alphabet $\A$ \eqref{alphabet} then reduces to
\begin{align*}
	\A=\{1,2,\ldots,\anum\}\cup\{\hat 1,\hat 2,\ldots,\hat{\bnum}\}.
\end{align*}
Let us slightly modify
(with the help of Remark \ref{rmk:ordering})
the probability distribution
$\mathrm{m}=\mathrm{m}^{\ab;\bb}$ \eqref{measure_m_on_A} on
$\A$
such that
\begin{align}\label{how_define_measure}
	\mathrm{m}^{\ab;\bb}(i)=\al_i,\quad i=1,\ldots,\anum;\qquad \qquad
	\mathrm{m}^{\ab;\bb}(\hat j)=\frac{\be_{\bnum+1-j}}{1-t},\quad j=1,\ldots,\bnum,
\end{align}
so that
$\mathrm{m}(1)\ge\mathrm{m}(2)\ge \ldots\ge \mathrm{m}(\anum)$, and
$\mathrm{m}(\hat 1)\le\mathrm{m}(\hat 2)\le \ldots\le \mathrm{m}(\hat{\bnum})$.

\subsubsection{Row lengths} 
\label{ssub:row_lengths}

To deal with row lengths, we order the alphabet $\A$
as
\begin{align*}
	1<2<\ldots<\anum<\hat 1<\hat 2<\ldots<\hat{\bnum}
\end{align*}
(this ordering coincides with the one considered
in \S \ref{sec:rsk_type_algorithm_for_sampling_hl_coherent_measures}).
Then the lower part $\boldsymbol\la^{[\ab]}$ of the
array on Fig.~\ref{fig:A_particle}
consists of a finite number of interlacing particles, namely,
$\{\la^{(m)}_{i}\}$, where
$m=1,\ldots,\anum$, $j=1,\ldots,m$.
To establish the
row part of Theorem \ref{thm:main_s7} (i.e., the
convergence $\la_j(n)/n\to\al_j$), it suffices to
show that the coordinates of the
particles at the top $\anum$th level behave as
$\la^{(\anum)}_{j}(n)/n\to\al_j$ as $n\to\infty$, $j=1,\ldots,\anum$.
Indeed, recall the bijection of
arrays with $\A$-tableaux 
(\S \ref{sub:state_space_towers_of_young_diagrams_a_tableaux_and_interlacing_configurations}), and observe that
in each row of an $\A$-tableau (as on Fig.~\ref{fig:Atableaux})
there can be at most one letter of the form
$\hat \imath$, where $i=1,\ldots,\bnum$.

Next, note that the evolution of the particle coordinates
$\la^{(\anum)}_{j}(n)$ depends only on the particles in the lower part
$\boldsymbol\la^{[\ab]}$ of the array. Thus, we can
freely assume that all the parameters $\be_j$ of our specialization
are zero, and
\begin{align}\label{finite_alphas_def}
	\al_1\ge \al_2\ge \ldots\ge\al_\anum\ge0,\qquad
	\sum_{i=1}^{\anum}\al_i=1.
\end{align}
This is the setup which we are going to apply
to prove the row part of Theorem \ref{thm:main_s7} (in the
case of both sequences $\ab$ and $\boldsymbol\be$
being finite).

Our strategy of the proof relies on the observation
that when (in the course of the sampling algorithm)
each particle
is far enough from its upper left neighbor,
a move of $\la^{(m)}_{j}$ triggers
the move of its upper right neighbor $\la^{(m+1)}_{j}$,
which in turn pushes $\la^{(m+2)}_{j}$, and so on.
Our choice of speeds of independently jumping particles
$\la^{(m)}_m$ (i.e., the bottommost
particle $\la^{(1)}_{1}$ is the fastest one)
informally suggests that
the particles will be in that generic situation most of the
time, and
the whole
array will asymptotically look like on Fig.~\ref{fig:alpha_array}.
That is, particles
$\la^{(m)}_{j}$ with fixed $j$
and $m$ running over $j,j+1,\ldots,\anum$
will have horizontal coordinate $\sim\al_j n$
(plus random fluctuations of order $\sqrt n$, cf. Remark \ref{rmk:CLT}).
It is helpful to keep this observation in mind while reading the
rigorous arguments below in this section.
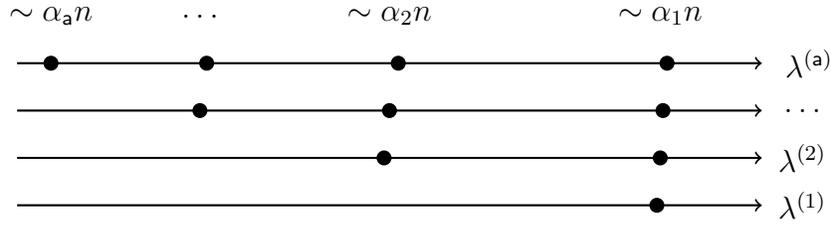
\begin{figure}[htbp]
\begin{center}
	\begin{tikzpicture}
		[scale=.9, thick]
		\def\y{.7}
		\foreach \level in {0,1,2,3}
		{
			\draw[->] (-.5,\y*\level) -- (10.5,\y*\level);
		}		
	
	    \def\sp{.2};
	    \def\opac{.65}
	    \def\w{1.1}
	    \def\e{0.03}
	    \def\a{9}
	    \def\aa{5}
	    \def\aaa{2.2}
	    \def\aaaa{0}
	    \foreach \pt in
	    {(\a-.05,0),(\a,\y),(\a+.04,2*\y),(\a+.1,3*\y),
	    (\aa-.08,\y),(\aa,2*\y),(\aa+.13,3*\y),
	    (\aaa,2*\y),(\aaa+.1,3*\y),
	    (\aaaa,3*\y)}
	    {
	    	\draw[fill] \pt circle (2.7pt);
	    }
	    \node at (11.1,0) {$\la^{(1)}$};
	    \node at (11.1,\y) {$\la^{(2)}$};
	    \node at (11.1,2*\y) {$\ldots$};
	    \node at (11.2,3*\y) {$\la^{(\anum)}$};
	    \node at (\a,4*\y) {$\sim\al_1 n$};
	    \node at (\aa,4*\y) {$\sim\al_2 n$};
	    \node at (\aaa,4*\y) {$\ldots$};
	    \node at (\aaaa,4*\y) {$\sim\al_{\anum} n$};
	\end{tikzpicture}
\end{center}
\caption{Clusters of particles under a
specialization $\ab=(\al_1,\ldots,\al_\anum)$
into $\anum$ usual variables.}
\label{fig:alpha_array}
\end{figure}


\subsubsection{Column lengths} 
\label{ssub:column_lengths}

For the purpose of dealing with column lengths, we reorder $\A$
as
\begin{align*}
	\hat 1<\hat 2<\ldots<\hat{\bnum}<1<2<\ldots<\anum
\end{align*}
(this is allowed by Remark \ref{rmk:ordering}).

In this case by a similar argument
we see that it also
suffices to consider only the case when all $\al_j$'s
are zero, and the specialization has the parameters
\begin{align}\label{finite_betas_def}
	\be_1\ge\be_2\ge \ldots\ge\be_\bnum,\qquad
	\sum_{i=1}^{\bnum}\be_i=(1-t).
\end{align}
Then in the array $\boldsymbol\la$
there is a finite number of interlacing particles
$(\la^{(\hat m)}_{j})'$, $m=1,\ldots,\bnum$, $j=1,\ldots,m$,
on the first $\bnum$ levels.
To shorten the notation, denote their coordinates by
$\tau^{(m)}_j:=(\la^{(\hat m)}_{j})'$.\footnote{In fact, these particles evolve
according to Dynamics 8 in \cite{BorodinPetrov2013NN},
up to renaming $t$ by $q$ and considering the evolution in continuous time,
cf. Remark \ref{rmk:poiss_alg}.}
To establish the column part of Theorem \ref{thm:main_s7},
we need to show that the top
row particles behave as $\tau^{(\bnum)}_{j}(n)/n\to\be_j/(1-t)$ as $n\to\infty$,
$j=1,\ldots,\bnum$.

Our choice of speeds of independently jumping particles
$\tau^{(m)}_1$ (i.e., the bottommost
particle $\tau^{(1)}_{1}$ is the slowest one)
informally suggests that
our interlacing particles behave as on Fig.~\ref{fig:beta_array}
(plus random fluctuations of order $\sqrt n$, cf. Remark \ref{rmk:CLT}).
Indeed, when the particles are sufficiently
far from each other, a move of any particle $\la^{(m)}_{j}$
will trigger the move of its upper left neighbor
$\la^{(m+1)}_{j+1}$. It is also helpful
to keep this observation in mind while reading the
rigorous arguments below.

\begin{figure}[htbp]
\begin{center}
	\begin{tikzpicture}
		[scale=.9, thick]
		\def\y{.7}
		\foreach \level in {0,1,2,3}
		{
			\draw[->] (-.5,\y*\level) -- (10.5,\y*\level);
		}		
	
	    \def\sp{.2};
	    \def\opac{.65}
	    \def\w{1.1}
	    \def\e{0.03}
	    \def\a{9}
	    \def\aa{6}
	    \def\aaa{2.2}
	    \def\aaaa{0}
	    \foreach \pt in
	    {(\aaaa+.15,0),(\aaa+.07,\y),(\aa+.04,2*\y),(\a+.1,3*\y),
	    (\aaaa+.08,\y),(\aaa,2*\y),(\aa-.1,3*\y),
	    (\aaaa,2*\y),(\aaa,3*\y),
	    (\aaaa-.02,3*\y)}
	    {
	    	\draw[fill] \pt circle (2.7pt);
	    }
	    \node at (11.1,0) {$\tau^{(1)}$};
	    \node at (11.1,\y) {$\tau^{(2)}$};
	    \node at (11.1,2*\y) {$\ldots$};
	    \node at (11.2,3*\y) {$\tau^{(\bnum)}$};
	    \node at (\a,4*\y) {$\sim\frac{\be_1 n}{1-t}$};
	    \node at (\aa,4*\y) {$\sim\frac{\be_2 n}{1-t}$};
	    \node at (\aaa,4*\y) {$\ldots$};
	    \node at (\aaaa,4*\y) {$\sim\frac{\be_\bnum n}{1-t}$};
	\end{tikzpicture}
\end{center}
\caption{Clusters of particles under a
specialization $\boldsymbol\be=(\be_1,\ldots,\be_\bnum)$
into $\bnum$ dual variables.}
\label{fig:beta_array}
\end{figure}
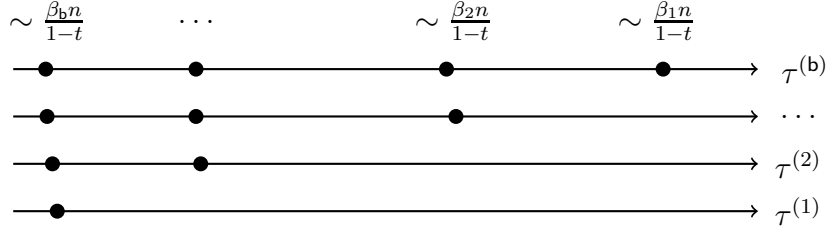

\begin{remark}\label{rmk:qpushtasep}
The evolution of $\{\tau^{(m)}_{1}\}_{m=1}^{\bnum}$
becomes the $q$-PushTASEP
\cite{BorodinPetrov2013NN},
\cite{CorwinPetrov2013}
after the renaming $t$ by $q$ and considering the evolution in continuous time. Thus, Theorem
\ref{thm:main_s7} that we are proving implies
existence of asymptotic speeds
of particles under the
$q$-PushTASEP.
\end{remark}


\subsubsection{Outline of the section} 
\label{ssub:outline}

In \S \ref{sub:auxiliary_dynamics} we define certain simple Markov
dynamics which are used in coupling arguments in
the case when both sequences of parameters
$\ab$ and $\boldsymbol\be$ are finite. Proofs in this case
are presented in \S \ref{sub:row_lengths_finite_case}
(row part)
and \S \ref{sub:column_lengths_finitely_many_dual_variables}
(column part).
With the help of
an additional argument, the case when
sequences
$\ab$ and $\boldsymbol\be$
can be infinite
is reduced to the
case when both these sequences are finite.
We deal with this reduction in
\S \ref{sub:completing_the_proof_when_one_of_the_specializations_is_infinite}
below.



\subsection{Auxiliary dynamics and estimates} 
\label{sub:auxiliary_dynamics}

Let $\mm(n)$ ($n\ge0$ --- discrete time) be a Markov chain
with state space $\Z_{\ge0}$ depending on
parameters $\pu\ge \pd>0$ and $\pz>0$
with one-step transitions defined as
\begin{align*}
	\mm(n+1)=\begin{cases}
		\mm(n)+1,&\mbox{with probability $\pu$};\\
		\mm(n)-1,&\mbox{with probability $\pd$};\\
		\mm(n),&\mbox{with probability $1-\pu-\pd$}
	\end{cases}
\end{align*}
if $\mm(n)\ne0$, and
\begin{align*}
	\mm(n+1)=\begin{cases}
		1,&\mbox{with probability $\pz$};\\
		0,&\mbox{with probability $1-\pz$}.
	\end{cases}
\end{align*}

For any $c\in(0,1)$, define
\begin{align}\label{w_phi}
	\wm(n):=
	\sum_{i=0}^{n}
	\mathbf{1}_{\mm(i)=0},
	\qquad
	\phi(n):=
	\sum_{i=0}^{n}
	c^{\mm(i)}.
\end{align}

\begin{lemma}\label{lemma:s7_1}
	As $n\to\infty$,
	\begin{align*}
		\E\wm(n)=O(\sqrt n),
		\qquad
		\E\phi(n)=O(\sqrt n).
	\end{align*}
\end{lemma}
\begin{proof}
	Clearly, it suffices to establish the statement of the
	lemma for the case $\pu=\pd$.
	Let $\{\nu_i\}_{i\ge 1}$ be independent random variables with
	distribution $\pd \delta_{-1}+(1-2\pd)\delta_0+\pd \delta_1$
	(here $\delta$ means the Dirac probability distribution at a
	point). Denote
	\begin{align*}
		\mm'(k):=|\nu_1+\nu_2+\ldots+\nu_k|,\qquad
		\wm'_{r}(k):=\sum_{i=0}^{k}
		\mathbf{1}_{\mm(i)=r},
	\end{align*}
	where $r=0,1,2,\ldots$.
	
	It can be readily shown
	(for example, using the standard Central Limit Theorem) that
	\begin{align*}
		P(\mm'(k)=r)\le \frac{const}{\sqrt k}
	\end{align*}
	for some $const$ not depending on $k$.
	Therefore,
	\begin{align*}
		\E \wm'_{r}(n)=
		\sum_{k=0}^{n}P(\mm'(k)=r)\le
		\sum_{k=0}^{n} \frac{const}{\sqrt k}
		\le const \sqrt n.
	\end{align*}

	The chain $\mm(n)$ differs from $\mm'(n)$
	only by transition probabilities at zero.
	Therefore, because on average the chain $\mm(n)$
	spends time $1/\pz$ at zero, we have
	\begin{align*}
		\E\sum_{i=0}^{n}\mathbf{1}_{\mm(i)=r}\le
		\frac{1}{\pz}\E\wm'_r(n)\le const \sqrt n
	\end{align*}
	for any $r=0,1,\ldots$.
	This readily implies the desired estimates.
\end{proof}

Now let us introduce two auxiliary
many-particle systems on $\Z$.
The first of the systems is (the discrete-time version of,
cf. Remark \ref{rmk:poiss_alg})
the well-known Totally Asymmetric Simple
Exclusion Process (TASEP) with a finite number
$\anum=1,2,\ldots$ of particles, and with particle-dependent
speeds $\{\al_i\}_{i=1}^{\anum}$ satisfying \eqref{finite_alphas_def}.

Let us denote positions of the particles in that process by
$\T_1(n)\ge \T_2(n)\ge \ldots\ge T_\anum(n)$ ($n$ is the discrete time).
The dynamics preserves this ordering.
The evolution
of this TASEP
goes as follows. Initially, $\T_1(0)= \T_2(0)= \ldots= T_\anum(0)=0$.
At each moment of the discrete time,
exactly one of the particles receives a ``jumping signal'';
the probability that this is $\T_i$ is equal to $\al_i$ (independently of previous ``signals''). Then, if
$\T_i(n)<\T_{i-1}(n)$, the coordinate of $\T_i$ increases by one.
Otherwise, if $\T_i(n)=\T_{i-1}(n)$, then no jump occurs
(so the particle $\T_i$ can be blocked by the next one $\T_{i-1}$).

\begin{proposition}\label{prop:s7_T}
	The dynamics $\{\T_i(n)\}_{i=1}^{\anum}$ satisfies the following
	Law of Large Numbers:
	\begin{align*}
		\frac{\T_i(n)}n\to\al_i,\qquad n\to\infty, \qquad \mbox{almost surely}.
	\end{align*}
\end{proposition}
\begin{proof}
	This fact is well-known. 
	However, we were not able to find an exact reference in the literature. For the sake of completeness 
	we provide one of many possible explanations.

	The TASEP $\{\T_i(n)\}_{i=1}^{\anum}$
	can be identified
	with 
	the dynamics
	of the leftmost particles in an interlacing
	array under the column RSK insertion process, 
	e.g., see \cite[\S7.1.3]{BorodinPetrov2013NN}. 
	(This fact can be traced back to 
	\cite{johansson2000shape}, \cite{baik1999distribution},
	\cite{OConnell2003Trans}, \cite{OConnell2003},
	see also \cite{BorFerr2008DF}.)
	Therefore, 
	$\T_i(n)$ can be identified
	with
	the last row $\la_i$ of 
	a Young diagram $\la$ with $\ell(\la)\le i$
	distributed as follows. First, let $k$
	be a binomial random variable with distribution
	\begin{align*}
		\Prob(k)=\binom{n}{k}(\al_1+\ldots+\al_i)^{k}
		(\al_{i+1}+\ldots+\al_{\anum})^{n-k}
	\end{align*}
	(recall the condition \eqref{finite_alphas_def} on $\al_j$'s).
	Then, given $k$, let $\la$ be a Young diagram with the distribution
	\begin{align}\label{Schur_coherent_measures}
		\Prob(\la):=\frac{k!}{(\al_1+\ldots+\al_{i})^{k}}
		\,s_\la(\al_1,\ldots,\al_{i})s_\la(\Pl_1),
		\qquad |\la|=k,\quad \ell(\la)\le i
	\end{align}
	(cf. \eqref{Schur_coherent_intro}, note also that 
	this distribution is a particular
	$q=t$ case of the coherent measures discussed in \S \ref{sub:_q_t_coherent_measures}).

	The Law of Large Numbers for 
	the measures \eqref{Schur_coherent_measures}
	(for nonrandom growing $k$)
	was established in 
	\cite{VK81AsymptoticTheory}, stating that 
	\begin{align*}
		\frac{\la_i(k)}{k}\to \frac{\al_i}{\al_1+\ldots+\al_i},
		\qquad k\to\infty.
	\end{align*}
	Since $k$ satisfies the classical Law of Large Numbers,
	$k\approx (\al_1+\ldots+\al_i)n$, 
	a standard argument shows that the desired 
	statement holds. 
\end{proof}

Let us introduce the second dynamics which has $\bnum\ge1$ particles on $\Z$
and depends on our Hall--Littlewood parameter $0<t<1$, as well as on parameters
$\{\be_{j}\}_{j=1}^{\bnum}$ satisfying \eqref{finite_betas_def}.
Let us denote positions of the particles in our second process by
$\Q_1(n)\le \Q_2(n)\le \ldots\le Q_\bnum(n)$ ($n$ is the discrete time).
The dynamics preserves this ordering.
The evolution of this system is described as follows.
Initially, all particles start at zero.
Next, at each moment of the discrete time,
exactly one of the particles,
namely, $\Q_i$ with probability $\be_{\bnum+1-j}/(1-t)$,
receives a ``jumping signal''.
The particle $\Q_i$ which received this signal
jumps to the right by one (it cannot be blocked).
After that, with probability $t^{\Q_{i+1}(n)-\Q_{i}(n)}$, \emph{all}
particles $\Q_{i+1},\Q_{i+2},\Q_{i+3},\ldots$ also move to the right by one. With
the complementary probability $1-t^{\Q_{i+1}(n)-\Q_{i}(n)}$,
no particles other than $\Q_{i}$ (which has already jumped) move.
Note that if $\Q_{i+1}(n)=\Q_{i}(n)$, then the pushing probability reduces to $1$.

We have chosen the speeds of independently jumping particles
in this dynamics so that the first particle $\Q_{1}$ is the slowest one,
cf. \S \ref{ssub:column_lengths}.

\begin{remark}\label{rmk:qpushtasep_difference}
	Note the difference between the Markov dynamics $\{\Q_{j}(n)\}_{j=1}^{\bnum}$
	and the Markov evolution of the particles $\{\tau^{(m)}_{1}\}_{m=1}^{\bnum}$
	from \S \ref{ssub:column_lengths}.
	In the former dynamics,
	if a long-range push happens (with probability $t^{\Q_{i+1}(n)-\Q_{i}(n)}$),
	then all particles with indices $i+1,i+2,\ldots$ move.
	In the latter dynamics, a long-range push
	is applied only to the $(i+1)$-st particle,
	and further pushes (of particles $i+2,i+3,\ldots$) happen under an additional randomness.
\end{remark}

\begin{proposition}\label{prop:s7_Q}
	The dynamics $\{\Q_j(n)\}_{j=1}^{\bnum}$ satisfies the following
	Law of Large Numbers:
	\begin{align*}
		\frac{\Q_j(n)}n\to
		\frac{\be_{\bnum+1-j}}{1-t},\qquad n\to\infty, \qquad \mbox{almost surely}.
	\end{align*}
\end{proposition}
\begin{proof}
	Denote $\mathsf{d}_i(n):=\Q_{i+1}(n)-\Q_{i}(n)$.
	This quantity increases when $\Q_{i+1}$ jumps independently,
	and decreases when $\Q_{i}$ jumps independently and does not push $\Q_{i+1}$.
	A natural coupling with the process $\mm(n)$ from Lemma \ref{lemma:s7_1}
	implies
	\begin{align*}
		\sum_{k=0}^{n}t^{\mathsf{d}_i(k)}\le const \sqrt n.
	\end{align*}

	Therefore, the expectation of
	the number of times (before time $n$)
	when an independent jump
	of $\Q_{i}$ leads to a move of $\Q_{i+1}$
	can be estimated by $const \sqrt n$. We see that the
	evolution of each particle $\Q_i$
	is affected by pushes $\le const \sqrt n$ times,
	so the asymptotic speed of this particle
	is determined by its independent jumps. This concludes the proof.
\end{proof}

It is helpful to look at Fig.~\ref{fig:alpha_array} and
\ref{fig:beta_array} in connection with
Propositions \ref{prop:s7_T} and \ref{prop:s7_Q}, respectively.

\begin{remark}\label{rmk:CLT}
	When $\gamma=0$ and the remaining parameters
	are distinct (when they are positive), i.e.,
	$\al_1>\al_2> \ldots$
	and $\be_1>\be_2> \ldots$,
	the estimates of this subsection
	can be readily improved from
	$O(\sqrt n)$ to $O(1)$.
	In this case, arguments of
	\S \S \ref{sub:row_lengths_finite_case}--\ref{sub:completing_the_proof_when_one_of_the_specializations_is_infinite}
	should give a Central Limit Theorem
	(Conjecture \ref{conj:CLT_intro}),
	which would follow
	from the corresponding Central Limit
	Theorem for the independent random
	letters in the input word
	$w=\xi_1\xi_2 \ldots$
	(\S \ref{sub:full_sampling_algorithm}).
	This explains the nature of the covariance matrix in Conjecture \ref{conj:CLT_intro}.
\end{remark}


\subsection{Finitely many usual variables, and row lengths} 
\label{sub:row_lengths_finite_case}

Here we will consider the evolution
of an interlacing particle array $\{\la^{(m)}_{j}(n)\}$, where $m=1,\ldots,\anum$,
$j=1,\ldots,m$, which depends on parameters $\{\al_i\}_{i=1}^{\anum}$
satisfying \eqref{finite_alphas_def}
(see \S \ref{ssub:row_lengths} and Fig.~\ref{fig:alpha_array}).
We can assume that all $\be_j$'s are zero, see \S \ref{sub:strategy_of_the_proof}.
Under this evolution, only the leftmost particles $\la^{(m)}_{m}$, $m=1,\ldots,\anum$, can jump independently.
Let $N_m (n)$ be the number of independent jumps performed by the particle $\la^{(m)}_m$.
Note that the bottommost particle
$\la^{(1)}_{1}$ is the fastest one.

The goal of this section is to prove Proposition \ref{prop:row_proof_s7}.

\begin{lemma}\label{lemma:row_proof_s7}
	The following
	convergence holds for every $i=1,\ldots,\anum$:
	\begin{align*}
		\frac{\la^{(i)}_{i}(n)}n\to
		\mathrm{m}^{\ab;\bb}(i)=\al_i,
		\qquad n\to\infty, \qquad \mbox{almost surely}.
	\end{align*}
\end{lemma}
\begin{proof}
	The key point in the proof is to bound the quantities
	$\la^{(i)}_{i}(n)$ \emph{from below}, which is done by
	means of their \emph{coupling} with the dynamics $\{\T_i(n)\}_{i=1}^{\anum}$
	from \S \ref{sub:auxiliary_dynamics}.
	We will put both dynamics
	$\{\la^{(m)}_{i}(n)\}$ and $\{\T_i(n)\}$
	on the same probability space $\Omega_n$ (with $n$ being arbitrary), such that
	\begin{align}\label{TASEP_estimate}
		\la^{(i)}_{i}(n)\ge\T_i(n)\qquad \mbox{everywhere on $\Omega_n$}.
	\end{align}
	We are assuming that both dynamics
	depend on the same parameters $\{\al_i\}_{i=1}^{\anum}$
	satisfying \eqref{finite_alphas_def}.

	Let us begin with defining the space $\Omega_n$ for the dynamics on
	interlacing triangular
	arrays $\{\la^{(m)}_{i}(n)\}$.
	Initially, $\Omega_0$ contains one point, and all $\la^{(m)}_{i}(0)$'s are equal to zero.
	The passage from $\Omega_n$ to $\Omega_{n+1}$ amounts to
	splitting every element of $\Omega_n$
	into $\anum$ parts of relative probability weights $\al_1,\ldots,\al_{\anum}$;
	the $i$-th such part corresponds to the particle $\la^{(i)}_{i}$ jumping independently.
	Furthermore, one has to
	account for probabilities of subsequent
	pushes (see \S \ref{sub:full_sampling_algorithm}).

	Thus, every element of $\Omega_{n+1}$ can be described as a word
	$w=\xi_1 \ldots \xi_{n+1}$, $\xi_i\in\{1,2,\ldots,\anum\}$ (this is the input word of the
	sampling algorithm, cf. \S \ref{sub:full_sampling_algorithm}),
	plus the whole history of the dynamics of the interlacing array
	\begin{align*}
		\{\la^{(m)}_{i}(k)\colon m=1,\ldots,\anum; \ i=1,\ldots,m; \ k=1,\ldots,n+1\}.
	\end{align*}
	It is clear how the dynamics
	on interlacing arrays assigns
	a probability weight to each such element of
	$\Omega_{n+1}$.

	The dynamics $\{\T_i(n)\}$ can also be put onto probability space $\Omega_n$
	in a rather straightforward
	way. Namely, projecting an element of $\Omega_n$ to the word $w=\xi_1 \ldots \xi_{n}$
	in the alphabet $\{1,2,\ldots,\anum\}$, we associate to each new incoming
	letter $\xi_j$ the independent jump of the particle $\T_{\xi_j}$ in the TASEP.
	In other words, particles $\la^{(i)}_{i}$ and $\T_i$ receive ``jump signals'' simultaneously.

	Let us now establish \eqref{TASEP_estimate}.
	Initially, at time $n=0$, \eqref{TASEP_estimate} clearly holds. By induction on $n$,
	assume that these inequalities hold for each element of $\Omega_n$.
	Let the new letter be $\xi_{n+1}=k$. Then
	\begin{enumerate}[$\bullet$]
		\item If $\la^{(k)}_k(n)>\T_k(n)$ at step $n$, then at step $n+1$ all the inequalities
		continue to hold.
		Indeed, at the next step each of the coordinates $\la^{(k)}_k$
		and $\T_k$ can potentially increase by 1, and no other coordinates
		of the form $\la^{(j)}_j$ and $\T_j$, $j\ne k$ can increase.\footnote{Of course,
		particles in the bulk of the interlacing array
		(i.e., all particles except the leftmost ones $\la^{(i)}_i$)
		will move in some way, but this
		cannot break the desired inequalities. This remark applies to other two cases as well.}
		\item If $\la^{(k)}_k(n)=\T_k(n)$ and $\T_{k-1}(n)>\T_{k}(n)$, then
		$\la^{(k)}_k(n)<\la^{(k-1)}_{k-1}(n)$
		by the induction hypothesis. Thus, both particles $\la^{(k)}_k$ and $\T_k$ are not blocked and thus
		jump,
		and the inequalities continue to hold
		(since no other particles
		$\la^{(j)}_j$ and $\T_j$, $j\ne k$, move).
		\item If $\la^{(k)}_k(n)=\T_k(n)$ and $\T_{k-1}(n)=\T_{k}(n)$,
		then the particles $\T_j$, $j=1,\ldots,\anum$, do not move at the next step.
		On the other hand, $\la^{(k)}_k$ can potentially jump (if it is not blocked),
		while other particles $\la^{(j)}_j$, $j\ne k$, will not move. Thus, the inequalities
		continue to hold.
	\end{enumerate}
	This argument completes the proof of \eqref{TASEP_estimate}.

	\smallskip

	Now let us finish the proof of the lemma.
	From \eqref{TASEP_estimate} and
	Proposition \ref{prop:s7_T} it follows that
	\begin{align*}
		\liminf_{n\to\infty}\frac{\la^{(i)}_{i}(n)}n\ge\al_i,\qquad
		\mbox{almost surely}
	\end{align*}
	for every $i=1,\ldots,\anum$.
	Moreover,
	$\la^{(\anum)}_{i}(n)\ge\la^{(i)}_{i}(n)$, and $\sum_{i=1}^{n}\la^{(\anum)}_{i}(n)=n$.
	By an elementary
	contradiction argument,
	this implies the claim of the lemma.
\end{proof}

\begin{proposition}\label{prop:row_proof_s7}
	The following
	Law of Large Numbers holds for every $m=1, \dots, \anum$ and every $i=1,\ldots, m$:
	\begin{align}\label{row_proof_s7}
		\frac{\la^{(m)}_{i}(n)}n\to
		\mathrm{m}^{\ab;\bb}(i)=\al_i,
		\qquad n\to\infty, \qquad \mbox{almost surely}.
	\end{align}
\end{proposition}
\begin{proof}
	We argue by
	induction on $i=m,m-1,\ldots,1$,
	and the case $i=m$, i.e., the convergence of $\la^{(m)}_{m}(n)/n$,
	follows from Lemma~\ref{lemma:row_proof_s7}.
	Now assume that we have established the convergence
	of $\la^{(m)}_{i}(n)/n$ for all $i=r+1,\ldots,m$,
	and let us prove that $\la^{(m)}_{r}(n)/n\to\al_r$.

	Clearly, $\la^{(m)}_{i}(n)\ge\la^{(i)}_{i}(n)$
	due to interlacing. Therefore, by Lemma \ref{lemma:row_proof_s7},
	\begin{align}
		\liminf_{n\to\infty}\frac{\la^{(m)}_{i}(n)}{n}\ge \al_i
		\qquad
		\mbox{almost sure}.
		\label{liminf_row}
	\end{align}

	On the other hand, we know that
	$$
    \sum_{i=1}^{m}\la^{(m)}_i(n)= \sum_{i=1}^m N_i (n),
    $$
    (because at the $n$th step of the
	sampling algorithm the Young tableau has exactly $n$ boxes,
	cf. \S \ref{sec:rsk_type_algorithm_for_sampling_hl_coherent_measures}). From the classical Law of Large Numbers for independent
    random variables we have
    $$
    \frac{N_i (n)}{n} \to \al_i, \qquad n \to \infty \qquad \mbox{almost surely}.
    $$
	This implies
	\begin{align*}&
		\al_{m}+\ldots+\al_{r+1}+\limsup_{n\to\infty}\frac{\la^{(m)}_{r}(n)}{n}
		\\&\qquad=
		\limsup_{n\to\infty}
		\frac{\la^{(m)}_{m}(n)+\la^{(m)}_{m-1}(n)+\ldots+\la^{(m)}_{r}(n)}{n}
		\\&\qquad=
		\limsup_{n\to\infty}
		\frac{(N_1(n) + \dots + N_m (n))-\la^{(m)}_{1}(n)-\la^{(m)}_{2}(n)- \ldots-\la^{(m)}_{r-1}(n)}{n}
		\\&\qquad=
		(\al_1+ \dots + \al_m) -\liminf_{n\to\infty}
		\frac{\la^{(m)}_{1}(n)+\la^{(m)}_{2}(n)+ \ldots+\la^{(m)}_{r-1}(n)}{n}
		\le (\al_1 +\dots + \al_m) -\al_1- \ldots-\al_{r-1}
        \\&\qquad = \al_r + \dots + \al_m.
	\end{align*}
	Therefore,
	\begin{align*}
		\limsup_{n\to\infty}\frac{\la^{(m)}_{r}(n)}{n}\le \al_r
		\qquad
		\mbox{almost sure},
	\end{align*}
	which (together with \eqref{liminf_row}) implies the desired convergence.
\end{proof}


\subsection{Finitely many dual variables, and column lengths} 
\label{sub:column_lengths_finitely_many_dual_variables}

Now let us consider the evolution
of an interlacing particle array $\{\tau^{(m)}_{j}(n)\}$, where $m=1,\ldots,\bnum$,
$j=1,\ldots,m$, which depends on parameters $\{\be_i\}_{i=1}^{\bnum}$
satisfying \eqref{finite_betas_def} (see \S \ref{ssub:column_lengths} and Fig.~\ref{fig:beta_array}).
We can assume that all $\al_j$'s are zero, see \S \ref{sub:strategy_of_the_proof}.
Under this evolution, only the rightmost particles $\tau^{(m)}_1$, $m=1,\ldots,\bnum$,
can jump independently. Moreover, they form a Markov chain, cf. Remark \ref{rmk:qpushtasep}.
Note that the bottommost particle
$\tau^{(1)}_{1}$ is the slowest one. We will prove

The main goal of this section is to prove Proposition \ref{prop:col_proof_s7}.

\begin{lemma}\label{lemma:col_proof_s7}
	The following
	convergence holds for every $i=1,\ldots,\bnum$:
	\begin{align*}
		\frac{\tau^{(i)}_{1}(n)}n\to
		\mathrm{m}^{\ab;\bb}(\hat \imath)=\frac{\be_{\bnum+1-i}}{1-t},
		\qquad n\to\infty, \qquad \mbox{almost surely}.
	\end{align*}
\end{lemma}
\begin{proof}
	The key point in the proof is to bound the quantities
	$\tau^{(i)}_{1}(n)$ \emph{from above}, which is achieved by
	means of their \emph{coupling} with the dynamics $\{\Q_i(n)\}_{i=1}^{\bnum}$
	from \S \ref{sub:auxiliary_dynamics}.
	This coupling is more complicated than that
	of Proposition \ref{prop:row_proof_s7} (cf. Remark \ref{rmk:qpushtasep_difference}).
	We will put both dynamics
	$\{\tau^{(m)}_{i}(n)\}$ (on interlacing arrays)
	and
	$\{\Q_i(n)\}$
	on the same probability space $\Omega'_n$ (with $n$ being arbitrary),
	such that
	\begin{align}\label{qPushTASEP_estimate}
		\tau^{(i)}_{1}(n)\le\Q_i(n)\qquad \mbox{everywhere on $\Omega'_n$}.
	\end{align}
	We are assuming that both dynamics
	depend on the same parameters $\{\be_i\}_{i=1}^{\bnum}$
	satisfying \eqref{finite_betas_def}.

	We will construct the desired probability spaces $\Omega'_n$
	by induction on $n$. Initially, for $n=0$,
	all particles $\{\tau^{(m)}_{i}(0)\}$ and
	$\{\Q_i(0)\}$ are at zero, and $\Omega'_0$
	consists of one point.
	At time $n$, each point in the space $\Omega'_n$ can be realized
	as a sequence of $n$ triples
	$(\xi_k,\eta_k,\zeta_k)$, $k=1,\ldots,n$,
	plus the whole history
	of the dynamics of the interlacing array
	\begin{align*}
		\{\tau^{(m)}_{i}(k)\colon m=1,\ldots,\bnum; \ i=1,\ldots,m; \ k=1,\ldots,n\}.
	\end{align*}
	Here $w=\xi_1 \ldots\xi_n$,
	where $\xi_i\in\{\hat 1,\hat 2,\ldots,\hat\bnum\}$, is
	the input word
	of the sampling algorithm
	(\S \ref{sub:full_sampling_algorithm}).
	If $\xi_n=\hat \imath$ at any time $n$, then $\eta_n\in\{1,\ldots,\bnum-i-1\}$
	encodes the number of particles
	of the form $\tau^{(j)}_{1}$, $j>i$, which are pushed
	(during time step $n-1\to n$) by the jump of
	$\tau^{(i)}_{1}$.\footnote{Note that once a particle $\tau^{(m)}_{1}$, for some $m>i$
	was not pushed, all upper particles $\tau^{(j)}_{1}$ ($j>m$)
	also cannot be pushed, see Remark \ref{rmk:no_donations}.}
	The quantity $\zeta_n\in\{+,-\}$ encodes the event of pushing
	in the
	dynamics of $\{\Q_j\}$
	(during time step $n-1\to n$).
	Namely, if $\xi_n=\hat \imath$, then $\zeta_n=+$ means that
	all particles $\Q_{j}$, $j>i$, were pushed, and
	$\zeta_n=-$ means that no such particles were pushed.
	See Fig.~\ref{fig:beta_omega}.
	We will denote elements of $\Omega'_n$
	by $(\boldsymbol\xi,\boldsymbol\eta,\boldsymbol\zeta,\boldsymbol\tau)$.
	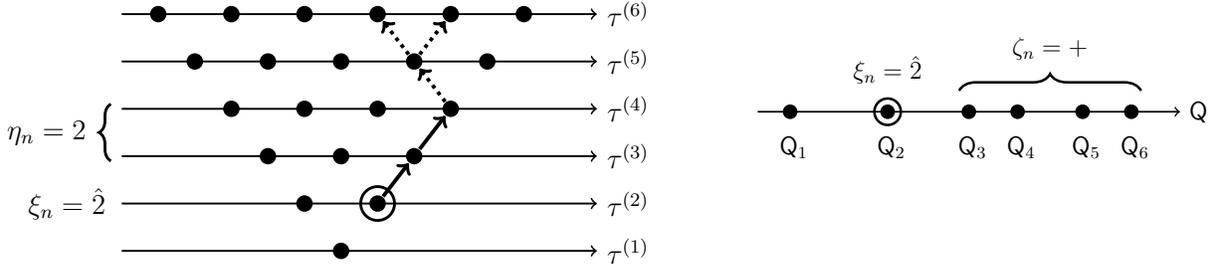
\begin{figure}[htbp]
	\begin{center}
		\begin{tabular}{ll}
			\scalebox{.9}{\begin{tikzpicture}
				[scale=1, very thick]
				\def\sp{0.1};
				\def\eps{0.15};
				\def\x{.54};
				\def\y{.7};
				\foreach \pt in
				{
					(0,0),
					(-\x,\y), (\x,\y),
					(-2*\x,2*\y), (0,2*\y), (2*\x,2*\y),
					(-3*\x,3*\y), (-1*\x,3*\y), (1*\x,3*\y), (3*\x,3*\y),
					(-4*\x,4*\y), (-2*\x,4*\y), (0*\x,4*\y), (2*\x,4*\y), (4*\x,4*\y),
					(-5*\x,5*\y), (-3*\x,5*\y), (-1*\x,5*\y), (1*\x,5*\y), (3*\x,5*\y), (5*\x,5*\y)
				}
				{
					\draw[fill] \pt circle (\sp);
				}
				\draw (\x,\y) circle (2.5*\sp);
				\draw[->, ultra thick] (\x+\eps*\x,\y+\eps*\y) --++ (\x-2*\eps*\x,\y-2*\eps*\y);
				\draw[->, ultra thick] (2*\x+\eps*\x,2*\y+\eps*\y) --++ (\x-2*\eps*\x,\y-2*\eps*\y);
				\draw[->, ultra thick, dotted] (3*\x-\eps*\x,3*\y+\eps*\y) --++ (-\x+2*\eps*\x,\y-2*\eps*\y);
				\draw[->, ultra thick,dotted] (2*\x-\eps*\x,4*\y+\eps*\y) --++ (-\x+2*\eps*\x,\y-2*\eps*\y);
				\draw[->, ultra thick,dotted] (2*\x+\eps*\x,4*\y+\eps*\y) --++ (\x-2*\eps*\x,\y-2*\eps*\y);
				\foreach \ll in {1,2,3,4,5,6}
				{
					\draw[->, thick] (-6*\x,\ll*\y-\y) -- (7*\x,\ll*\y-\y)
					node[right] {$\tau^{(\ll)}$};
				}
				\node at (-7.5*\x,\y) {$\xi_{n}=\hat2$};

				\draw [decorate,decoration={brace,amplitude=5pt,raise=3pt},yshift=0pt]
				(-6.1*\x,1.9*\y) -- (-6.1*\x,3.1*\y)
				node [midway,xshift=-30] {$\eta_{n}=2$};
			\end{tikzpicture}}&\hspace{20pt}
			\raisebox{40pt}{\scalebox{.8}
			{
			\begin{tikzpicture}
				[scale=1, very thick]
				\def\sp{0.1};
				\def\eps{0.15};
				\def\x{.54};
				\def\y{.7};
				\foreach \pt in
				{
					-5,-2,.5,2,4,5.5
				}
				{
					\draw[fill] (\pt*\x,0) circle (\sp);
				}
				\def\hit{-.6};
				\def\xs{0.06};
				\node at (-5*\x+\xs,\hit) {$\Q_{1}$};
				\node at (-2*\x+\xs,\hit) {$\Q_{2}$};
				\node at (.5*\x+\xs,\hit) {$\Q_{3}$};
				\node at (2*\x+\xs,\hit) {$\Q_{4}$};
				\node at (4*\x+\xs,\hit) {$\Q_{5}$};
				\node at (5.5*\x+\xs,\hit) {$\Q_{6}$};

				\draw (-2*\x,0) circle (2.2*\sp);
				\node at (-2*\x,\y) {$\xi_{n}=\hat2$};
				\draw[->, thick] (-6*\x,0) -- (7*\x,0)
					node[right] {$\Q$};

				\draw [decorate,decoration={brace,amplitude=8pt,
				raise=-10pt},yshift=0pt]
				(.2*\x,\y) -- (5.8*\x,\y)
				node [midway,yshift=10] {$\zeta_{n}=+$};
			\end{tikzpicture}
			}}
		\end{tabular}
	\end{center}
	\caption{Quantities $(\xi_{n},\eta_{n},\zeta_{n})$.
	Particles $\tau^{(2)}_{1}$ and $\Q_{2}$ jump together;
	simultaneously particles $\tau^{(3)}_{1}$ and $\tau^{(4)}_{1}$ (but not $\tau^{(5)}_{1}$) are pushed
	in the dynamics $\{\tau^{(m)}_{i}\}$,
	and all particles $\Q_3,\Q_4,\ldots$ are pushed in the dynamics
	$\{\Q_{i}\}$.}
	\label{fig:beta_omega}
	\end{figure}

	We will let (random)
	interactions between particles in the bulk
	of the interlacing array
	(i.e., all particles except the rightmost ones $\tau^{(i)}_1$)
	to be independent of
	the dynamics $\{\Q_{i}(n)\}$, so it suffices to consider
	only the projection of $\Omega'_n$
	to $(\boldsymbol\xi,\boldsymbol\eta,\boldsymbol \zeta)$.

	Employing our induction on $n$, let us assign
	a \emph{joint} probability distribution
	to the triple $(\xi_n,\eta_n,\zeta_n)$
	which is added during time $n-1\to n$.
	First, we want
	both dynamics to receive the same
	``jumping signal'' $\xi_n$ with
	probability
	$P(\xi_n=\hat \imath)=\mathrm{m}^{\ab;\bb}(\hat \imath)={\be_{\bnum+1-i}}/({1-t})$
	(this part is similar to the proof of Lemma \ref{lemma:row_proof_s7}).

	Now, conditioning on
	$\xi_n=\hat \imath$ for some fixed $\hat \imath$,
	let us denote by $r_j$, $j=0,1,2,\ldots$ the \emph{marginal} probability that $\eta_n=j$,
	and by $r_{\pm}$ the \emph{marginal} probability that $\zeta_n=\pm$.
	(It is worth noting that these probabilities depend on the
	state of the corresponding dynamics at time $n-1$,
	which is in turn determined by the
	history
	$(\eta_1,\ldots,\eta_{n-1})$ or $(\zeta_1,\ldots,\zeta_{n-1})$,
	respectively.)
	Let us choose $u=0,1,2,\ldots$ such that
	\begin{align}\label{joint_probab_u}
		r_0+r_1+\ldots+r_{u-1}\le r_{-},\qquad
		r_0+r_1+\ldots+r_{u}> r_{-},
	\end{align}
	and assign the following joint probability weights
	(recall that they are conditional on $\xi_n=\hat \imath$):
	\begin{align}
		\begin{array}{rcll}
			P(\eta_n=j, \zeta_n=-)&=&r_j,&\qquad j=0,1,\ldots,u-1;\\
			P(\eta_n=u, \zeta_n=-)&=&r_-;\\
			P(\eta_n=u, \zeta_n=+)&=&r_0+r_1+\ldots+r_u-r_-;\\
			P(\eta_n=j, \zeta_n=+)&=&r_j,&\qquad j=u+1,u+2,\ldots.
		\end{array}
		\label{joint_probab}
	\end{align}
	All other joint probabilities are set to zero.
	Clearly, thus defined probability distribution on
	$(\xi_n,\eta_n,\zeta_n)$
	projects in a desired way to
	$(\xi_n,\eta_n)$ (which corresponds to the dynamics
	on interlacing arrays) and to
	$(\xi_n,\zeta_n)$
	(which corresponds to the auxiliary dynamics $\{\Q_i\}$).

	Formulas \eqref{joint_probab} can be interpreted as follows:
	if $\ge u$ of the rightmost particles is pushed in the
	dynamics on interlacing arrays, then
	the pushing event happens in the
	auxiliary dynamics $\{\Q_j\}$ as well.

	Now let us prove the inequalities \eqref{qPushTASEP_estimate}.
	By induction, assume that they hold at time $n$.
	Let $\xi_{n+1}=\hat \imath$ for some $\hat \imath$.
	Consider the following cases:
	\begin{enumerate}[$\bullet$]
		\item $\Q_{i+\ell}(n)>\tau^{(i+\ell)}_{1}(n)$ for each $\ell=0,1,2,\ldots$.
		Then, since during the time $n\to n+1$ all coordinates can increase at most
		by one, the inequalities continue to hold at time $n+1$.
		\item $\Q_{i}(n)>\tau^{(i)}_{1}(n)$, and for some $\ell>0$
		we have $\Q_{i+\ell}(n)=\tau^{(i+\ell)}_1(n)$ (assume that this $\ell$
		is the smallest among all indices with this second property).
		Then, in the notation before \eqref{joint_probab_u}, we clearly have
		\begin{align*}
			r_+\ge t^{\Q_{i+\ell}-\Q_{i}},\qquad
			r_\ell+r_{\ell+1}+r_{\ell+2}+\ldots=t^{\tau^{(i+\ell)}_1-\tau^{(i)}_{1}}.
		\end{align*}
		Therefore,
		\begin{align*}
			r_+\ge r_\ell+r_{\ell+1}+r_{\ell+2}+\ldots,
		\end{align*}
		so $u\le \ell$, where $u$ is defined in \eqref{joint_probab_u}.
		This implies that if
		the particles $\tau^{(i+1)}_{1},\ldots,\tau^{(i+\ell)}_{1}$
		(and maybe some of the particles $\tau^{(j)}_{1}$,
		$j>i+\ell$, as well)
		were pushed, then the pushing event also happened in the auxiliary dynamics,
		so all the particles $\Q_j$, $j>i$, were pushed.
		This readily implies that the desired inequalities continue to hold at time $n+1$.
		\item $\Q_i(n)=\tau^{(i)}_{1}(n)$. Then
		also $\Q_{i+1}(n)\ge\tau^{(i+1)}_{1}(n)$ by the induction hypothesis,
		and thus $r_0\ge r_-$ in the notation before \eqref{joint_probab_u}.
		By \eqref{joint_probab},
		this means that if
		the particle $\tau^{(i+1)}_{1}$
		(and maybe some of the particles $\tau^{(j)}_{1}$,
		$j>i+1$, as well)
		was pushed, then
		then the pushing event also happened in the auxiliary dynamics
		$\{\Q_j\}$. Thus, inequalities \eqref{qPushTASEP_estimate} continue to hold at time $n+1$
		as well.
	\end{enumerate}
	The above cases show that \eqref{qPushTASEP_estimate} holds.

	\smallskip

	The desired limiting
	bound from below is achieved
	with the help of \eqref{qPushTASEP_estimate}
	and Proposition \ref{prop:s7_Q}
	in exactly the same way as
	in the end of the proof of Lemma \ref{lemma:row_proof_s7}.
	This completes the proof of the claim.
\end{proof}

\begin{proposition}\label{prop:col_proof_s7}
	The following
	Law of Large Numbers holds for every $m=1, \dots, \bnum$ and every $i=1,\ldots,m$:
	\begin{align}\label{col_proof_s7}
		\frac{\tau^{(m)}_{i}(n)}n\to\frac{\be_{i}}{1-t},
		\qquad n\to\infty, \qquad \mbox{almost surely}.
	\end{align}
\end{proposition}
\begin{proof}
	Since $\tau^{(m)}_i\le\tau^{(m+1-i)}_{1}$ for $i=1,\ldots,m$ due to
	interlacing, this proposition
	follows from Lemma \ref{lemma:col_proof_s7}
	by induction in exactly the same way as
	Proposition \ref{prop:row_proof_s7} follows from Lemma~\ref{lemma:row_proof_s7}.
\end{proof}


\subsection{Completing the proof} 
\label{sub:completing_the_proof_when_one_of_the_specializations_is_infinite}

With the results of
\S \ref{sub:row_lengths_finite_case} and \S \ref{sub:column_lengths_finitely_many_dual_variables},
we have now proved Theorem \ref{thm:main_s7} in the case
when both sequences $\ab$ and $\bb$ are finite.
Let us now extend this statement to the general case when these sequences
are allowed to be infinite, so the specialization depends on
\begin{align}\label{al_be_infinite}
	\al_1\ge\al_2\ge \ldots\ge0,
	\qquad
	\be_1\ge\be_2\ge \ldots\ge0,
	\qquad
	\sum_{i=1}^{\infty}\al_i+
	\frac{1}{1-t}\sum_{i=1}^{\infty}\be_i=1.
\end{align}

\smallskip

Assume that $\la(n)$ is the Young diagram
distributed according to the measure
$\HL_n^{\ab;\bb;\Pl_0}$. Recall that by $\la_j(n)$
and $\la'_j(n)$ we denote row and column lengths of this diagram.

\begin{lemma}\label{lemma:main_s7_lower}
	For any $k=1,2,\ldots$,
	we have
	\begin{align*}
		\liminf_{n\to\infty}\frac{\la_k(n)}{n}\ge\al_k,
		\qquad
		\liminf_{n\to\infty}\frac{\la_k'(n)}{n}\ge
		\frac{\be_k}{1-t}.
	\end{align*}
\end{lemma}
\begin{proof}
	Let us fix $k$ and prove that $\liminf_{n\to\infty}\frac{\la_k(n)}{n}\ge\al_k$. Consider the following ordering of the alphabet
	$\A$:
	\begin{align*}
		1<2<\ldots<k<\mbox{rest of the letters (ordered arbitrarily)}.
	\end{align*}
	The sampling algorithm
	for the measure $\HL_n^{\ab;\bb;\Pl_0}$
	under this ordering
	(\S \ref{sub:state_space_towers_of_young_diagrams_a_tableaux_and_interlacing_configurations}--\S \ref{sub:full_sampling_algorithm})
	is a Markov dynamics with state space consisting
	of towers of Young diagrams. First $k$
	floors in such a tower constitute an interlacing
	particle configuration
	$\{\la^{(m)}_{i}(n)\}_{1\le i\le m\le k}$.
	Moreover, employing the bijection with $\A$-tableaux,
	we see that $\la^{(k)}_{k}(n)$
	is the number of letters $k$ in the $k$-th row
	of the Young diagram $\la(n)$
	(this diagram is the shape of the $\A$-tableau).
	Therefore, $\la_k(n)\ge\la^{(k)}_k(n)$.

	On the other hand, note that the behavior of the first $k$ floors does not depend on what is happening above them.
    Therefore, we are in a position to apply Proposition
	\ref{prop:row_proof_s7} to the first $k$ floors (formally, we can consider the dynamics with finitely many
    $\al$- parameters $\al_1$, $\al_2$, $\dots$, $\al_k$, and $1 - \sum_{i=1}^k \al_i$, and apply Proposition
	\ref{prop:row_proof_s7} to this dynamics; the distribution of the first $k$ floors will be the same).
    We obtain that
	$\la^{(k)}_k(n)/n\to\al_k$.

	The corresponding statement about column lengths
	follows by considering the ordering
	\begin{align*}
		\hat 1<\hat 2<\ldots<\hat k<\mbox{rest of the letters (ordered arbitrarily)},
	\end{align*}
	and referring to Proposition \ref{prop:col_proof_s7}; again, we are able to apply this proposition due to the fact that the behavior
    of the first $k$ floors does not depend on what is happening above them.
\end{proof}

To finish the proof of
Theorem \ref{thm:main_s7},
it now remains to establish
upper bounds corresponding to the
lower bounds of Lemma \ref{lemma:main_s7_lower}:

\begin{lemma}\label{lemma:main_s7_upper}
	For any $k=1,2,\ldots$ and any $\varepsilon>0$,
	we have
	\begin{align*}
		\limsup_{n\to\infty}\frac{\la_k(n)}{n}<\al_k+\varepsilon,
		\qquad
		\limsup_{n\to\infty}\frac{\la_k'(n)}{n}<
		\frac{\be_k}{1-t}+\varepsilon.
	\end{align*}
\end{lemma}
\begin{proof}
	Let us prove the bound for the row lengths
	(the case of the column lengths is analogous).
	We argue similarly to the
	proof of
	Proposition \ref{prop:row_proof_s7}.

	Fix $k$ and $\varepsilon$.
	Using \eqref{al_be_infinite}, choose $\anum$ and $\bnum$
	so large that $\anum>k$ and that
	\begin{align*}
		\sum_{i=1}^{\anum}\al_i+
		\sum_{j=1}^{\bnum}\frac{\be_j}{1-t}>1-\varepsilon.
	\end{align*}
	
	Moreover, we know that
	\begin{align*}
		\sum_{i=1}^{\anum}\la_i(n)+
		\sum_{j=1}^{\bnum}\la_j'(n)\le n+\anum\bnum,
	\end{align*}
	because the Young diagram $\la(n)$
	has $n$ boxes,
	and in the summation over $i$ and $j$
	above we can count twice only the boxes
	from first $\anum$ rows and $\bnum$ columns.

	Therefore, we can write (omitting dependence
	on $n$ in $\la_j$ and $\la_j'$)
	\begin{align*}
		\limsup_{n\to\infty}
		\frac{\la_k(n)}{n}&\le
		\limsup_{n\to\infty}
		\frac{1}{n}
		\Big(n-\anum\bnum-\la_1-\la_2- \ldots-\la_{k-1}
		-\la_{k+1}
		-\ldots-\la_\anum
		-\la_1'- \ldots-\la'_\bnum
		\Big)
		\\&\le
		1-\liminf_{n\to\infty}
		\frac{1}{n}
		\Big(
		\la_1+ \ldots+\la_{k-1}
		+\la_{k+1}
		+\ldots+\la_\anum
		+\la_1'+ \ldots+\la'_\bnum
		\Big)
		\\&\le
		\al_k+1-\sum_{i=1}^{\anum}\al_i-\sum_{j=1}^{\bnum}
		\frac{\be_j}{1-t}< \al_k+\varepsilon.	
	\end{align*}
	Here in the last estimate for $\liminf$ we have used
	Lemma \ref{lemma:main_s7_lower}.
\end{proof}

Lemmas \ref{lemma:main_s7_lower}
and \ref{lemma:main_s7_upper} readily
imply Theorem \ref{thm:main_s7}.



\providecommand{\bysame}{\leavevmode\hbox to3em{\hrulefill}\thinspace}
\providecommand{\MR}{\relax\ifhmode\unskip\space\fi MR }
\providecommand{\MRhref}[2]{%
  \href{http://www.ams.org/mathscinet-getitem?mr=#1}{#2}
}
\providecommand{\href}[2]{#2}

\end{document}